\documentclass[11pt]{article}

\usepackage{amsmath}
\usepackage{amsfonts,amssymb}
\usepackage{theorem}
\usepackage[dvips]{graphics}
\usepackage{graphicx}
\title{Morita equivalence and characteristic classes of star products}

\author{H. Bursztyn, V. Dolgushev and S. Waldmann}

\date{}

\numberwithin{equation}{section}

\newcommand{\ring}[1]    {\mathsf{#1}}
\newcommand{\id}         {\mathrm{id}}
\newcommand{\ad}         {\mathrm{ad}}
\newcommand{\Hoch}{{\rm H o c h }}
\newcommand{\CH}{{\rm C H }}
\newcommand{\MC}{{\rm M C }}

\newcommand{\ma}{\mathfrak{a}}

\newcommand{\Pic}{{\rm P i c }}
\newcommand{\Def}{{\rm D e f }}
\newcommand{\End}{{\rm E n d }}
\newcommand{\Hom}{{\rm H o m }}

\newcommand{\FPois}{{\rm F P o i s s}}
\newcommand{\rE}{\mathrm{E}}
\newcommand{\E}{\mathrm{e}}
\newcommand{\DiffOp}{\operatorname{\mathrm{DiffOp}}}
\newcommand{\Exp}{\operatorname{\mathrm{Exp}}}

\newcommand{\sh}{\sharp}

\newcommand{\tcL}{\widetilde{\cal L}}
\newcommand{\tcK}{\widetilde{\cal K}}

\newcommand{\txi}{\widetilde{\xi}}

\newcommand{\tbeta}{\widetilde{\beta}}
\newcommand{\tf}{\widetilde{f}}

\newcommand{\tK}{\widetilde{K}}
\newcommand{\tmu}{\widetilde{\mu}}
\newcommand{\ttau}{\widetilde{\tau}}

\newcommand{\lrarrow}{\,\longrightarrow \,}
\newcommand{\llarrow}{\,\longleftarrow \,}
\newcommand{\brarrow}{\succ\rightarrow}

\newcommand{\SM}{{\cal S}M}

\newcommand{\scTp}{\cT^{\bul+1}_{poly}}
\newcommand{\cTp}{\cT^{\bul}_{poly}}

\newcommand{\Omb}{\Om^{\bul}}

\newcommand{\OM}{\cO_M}

\newcommand{\sCbu}{C^{\bullet+1}}
\newcommand{\Cbu}{C^{\bullet}}

\newcommand{\Linf}{L_{\infty}}
\newcommand{\Dif}{{\rm Diff}}

\newcommand{\tpi}{\tilde{\pi}}

\newcommand{\tstar}{\widetilde{*}}
\newcommand{\wstar}{\,\widetilde{\star}\,}

\newcommand{\tD}{\widetilde{D}}
\newcommand{\tA}{\widetilde{A}}

\newcommand{\tn}{\widetilde{\nabla}}

\newcommand{\X}{\mathcal{X}}
\newcommand{\al}{{\alpha}}
\newcommand{\la}{{\lambda}}
\newcommand{\vro}{{\varrho}}
\newcommand{\h}{{\hbar}}
\newcommand{\bul}{{\bullet}}

\newcommand{\mG}{{\mathfrak{G}}}

\newcommand{\mF}{{\mathfrak{F}}}

\newcommand{\mg}{{\mathfrak{g}}}

\newcommand{\mgl}{{\mathfrak{gl}}}

\newcommand{\om}{{\omega}}

\newcommand{\Om}{{\Omega}}

\newcommand{\si}{{\sigma}}
\newcommand{\ga}{{\gamma}}

\newcommand{\ve}{{\varepsilon}}
\newcommand{\ka}{{\kappa}}
\newcommand{\G}{{\Gamma}}
\newcommand{\cF}{{\cal F}}
\newcommand{\pa}{{\partial}}

\newcommand{\bs}{{\bf s}}

\newcommand{\cK}{{\cal K}}

\newcommand{\cL}{{\cal L}}
\newcommand{\cD}{{\cal D}}
\newcommand{\cA}{{\cal A}}
\newcommand{\cG}{{\cal G}}
\newcommand{\cH}{{\cal H}}
\newcommand{\cT}{{\cal T}}

\newcommand{\cB}{{\cal B}}
\newcommand{\cX}{{\cal X}}
\newcommand{\cO}{{\cal O}}

\newcommand{\bbC}{{\mathbb C}}
\newcommand{\bbR}{{\mathbb R}}

\newcommand{\bbQ}{{\mathbb Q}}

\newcommand{\La}{{\Lambda}}

\newcommand{\n}{{\nabla}}
\newcommand{\te}{\theta}

\newcommand{\de}{{\delta}}
\newcommand{\D}{{\Delta}}

\newcommand{\tQ}{{\widetilde{Q}}}
\newcommand{\tF}{{\widetilde{F}}}

\newcommand{\emp} [1]     {\textit{{#1}}}
\newcommand{\Cour}[1]      {[\![#1]\!]}
\newcommand{\Lie}        {\mathcal{L}}
\newcommand{\SP} [1]     {{\left\langle {{#1}} \right\rangle}}

\newcommand{\Pifib}   {\Pi_{\mathrm{fib}}}
\newcommand{\tPifib}  {\tilde{\Pi}_{\mathrm{fib}}}
\newcommand{\tdiamond} {\mathbin{\tilde{\diamond}}}
\newcommand{\pifib}   {\pi_{\mathrm{fib}}}
\newcommand{\mdiamond} {\mathbin{\tilde{\diamond}}}
\newcommand{\mtau}     {\tilde{\tau}}
\newcommand{\mstar}    {\mathbin{\tilde{*}}}
\newcommand{\tauEW}    {\tau^{\mathrm{EW}}}
\newcommand{\DEW}      {D^{\mathrm{EW}}}
\newcommand{\mmu}      {\tilde{\mu}}
\newcommand{\Fdiamond} {\mathbin{\diamond}_{\mathrm{F}}}
\newcommand{\Ftau}     {\tau^{\mathrm{F}}}
\newcommand{\DF}       {D^{\mathrm{F}}}

\newtheorem{defi}{Definition}[section]

\newtheorem{lem}[defi]{Lemma}
\newtheorem{teo}[defi]{Theorem}
\newtheorem{cor}[defi]{Corollary}

\newtheorem{pred}[defi]{Proposition}
\newtheorem{prop}[defi]{Proposition}
\newtheorem{claim}[defi]{Claim}
\newtheorem{cond}[defi]{Condition}

\theorembodyfont{\rm}
\newtheorem{example}[defi]{Example}
\newtheorem{property}{P}
\newtheorem{remark}[defi]{Remark}

\makeatletter
\newcommand\qedsymbol{\hbox{$\Box$}}
\newcommand\qed{\relax\ifmmode\Box\else
  {\unskip\nobreak\hfil\penalty50\hskip1em\null\nobreak\hfil\qedsymbol
  \parfillskip=\z@\finalhyphendemerits=0\endgraf}\fi}

\newcommand\subqedsymbol{\hbox{$\triangledown$}}
\newcommand\subqed{\relax\ifmmode\triangledown\else
  {\unskip\nobreak\hfil\penalty50\hskip1em\null\nobreak\hfil\subqedsymbol
  \parfillskip=\z@\finalhyphendemerits=0\endgraf}\fi}
\makeatother

\newenvironment{proof}[1][{}]{\par\noindent Proof{#1}. }{\qed}
\newenvironment{subproof}[1][{}]{\par\noindent Proof{#1}. }{\subqed}

\begin{document}

%
% Front page
%

\maketitle

\begin{center}
    \textit{Belatedly to Giovanni Felder and Boris Tsygan on the
      occasion of their 50th birthday.}
\end{center}

\medskip

\begin{abstract}
This paper deals with two aspects of the theory of characteristic
classes of star products: first, on an arbitrary Poisson manifold,
we describe Morita equivalent star products in terms of their
Kontsevich classes; second, on symplectic manifolds, we describe the
relationship between Kontsevich's and Fedosov's characteristic
classes of star products.
\end{abstract}

\tableofcontents

\newpage

%
% Introduction
%

\section{Introduction}

%Kontsevich's $L_{\infty}$ quasi-isomorphism \cite{K}, \cite{K-alg}
%from the Lie algebra of
%polyvector fields to the DG Lie algebra of polydifferential operators
%on a smooth real manifold $M$ produces a map from the set of formal
%Poisson structures on $M$ to the set of star-products on $M$\,.
%Standard homological algebra arguments imply that this map
%induces a bijection between the corresponding sets of equivalence
%classes. Thus to every star-product one may assign an equivalence
%class of a formal Poisson structure. This class is called Kontsevich's
%characteristic class or simply Kontsevich's class of this star product.

Given a smooth real manifold $M$, consider the set $\FPois(M)$ of
equivalence classes of formal Poisson structures $\pi = \h\pi_1 +
\h^2\pi_2 + \ldots \in \Gamma(\wedge^2 TM)[[\h]]$ on $M$, and let
$\Def(M)$ denote the set of equivalence classes of star products $*$
on $M$. The celebrated Kontsevich's formality theorem \cite{K,K-alg}
provides a bijective correspondence
\begin{equation}\label{eq:Kmap}
\cK_*: \FPois(M) \longrightarrow \Def(M),
\end{equation}
in such a way that if $\pi=\h\pi_1+\ldots$ and $*$ are related by
$\cK_*$, then $*$ is a deformation quantization, in the sense of
\cite{Bayen},  of the ordinary Poisson structure $\pi_1$. In
particular, any star product on $M$ can be assigned to an
equivalence class of a formal Poisson structure via \eqref{eq:Kmap},
called the \emp{Kontsevich characteristic class} or simply the
\emp{Kontsevich class} of the star product.

For a given star product, finding the associated Kontsevich class is
often a hard problem, to which no effective solution is currently
available. There are two main approaches to tackle this problem: the
first one makes use of algebraic index theorems, see e.g.
\cite{CFW-index,CD,Pindex,Fedosov-index,FFS,NT}, while the second is
based on homological algebra arguments, such as the duality between
Hochschild cohomology and homology
\cite{CFW-duality,Modular,EG,F-Sh,Vitya,VB}.

The first goal of the present paper is to describe the Kontsevich
classes of Morita equivalent star-product algebras on a smooth real
manifold $M$ (in this paper, star products are defined on the
algebra of \textit{complex-valued} smooth functions on $M$). As
shown in \cite{B,BW}, two star products on $M$ are Morita equivalent
if and only if they lie in the same orbit of a canonical action of
the group $\Dif(M)\ltimes \Pic(M)$ on the moduli space $\Def(M)$ of
equivalence classes of star products; here $\Dif(M)$ denotes the
group of diffeomorphisms of $M$, and $\Pic(M)\cong H^2(M, {\mathbb
Z})$ is the \emp{Picard group}, i.e., the group of isomorphism
classes of complex line bundles over $M$. The action of $\Dif(M)$ on
star products is the natural one by pull-back, while the action of
$\Pic(M)$ on $\Def(M)$ is defined in a less obvious way \cite{B}.
Hence the problem of expressing Morita equivalent star products in
terms of their Kontsevich classes amounts to describing the action
of $\Dif(M)\ltimes \Pic(M)$ on the moduli space of formal Poisson
structures $\FPois(M)$ making the map $\cK_*$ in \eqref{eq:Kmap}
equivariant.

The group $\Dif(M)$ naturally acts on formal Poisson structures, and
it follows from \cite{CEFT,K-alg} that the map \eqref{eq:Kmap} is
$\Dif(M)$-equivariant; so in order to describe Morita equivalent
star products one only needs to focus on the action of the Picard
group $\Pic(M)$. A key observation is that the set of formal Poisson
structures on $M$ carries a natural action of the abelian group of
closed ($\bbC[[\h]]$-valued) 2-forms, defined by a formal version of the
\textit{gauge transformations} of \cite[Sec.~3]{SeveraWeinstein}
(also known as \emp{B-field transforms} in the context of
generalized complex geometry \cite{Gualtthesis, Hitchin1});
moreover, we prove that this action naturally descends to an action
of the abelian group $H^2(M, \bbC)[[\h]]$ on the moduli space $\FPois(M)$.
Our first main result is Theorem \ref{ona}, which asserts that two star
products are related by the action of a line bundle $L$,
representing an element in $\Pic(M)$, if and only if their classes
in $\FPois(M)$ are connected by the action of the element
$2\pi i c_1(L)$, where $c_1(L)$ is the Chern class of $L$.
For a further discussion relating this result to Morita equivalence of
Poisson manifolds, we refer to \cite{BWe}. 

%For further connections with Poisson-geometric Morita equivalence
%we would like to refer the reader to
%In other
%words, upon an integrality condition, gauge equivalent formal
%Poisson structures quantize to Morita equivalent star-product
%algebras via $\cK_*$

Morita equivalent star products have been also considered in the
physics literature in the context of noncommutative gauge theory
\cite{JSW,JSW1,SW}. The essence of the statement of our Theorem
\ref{ona} may be found in these works, as well as ideas concerning
its proof; here we provide a complete proof of this result based on
the explicit globalization of Kontsevich's formality
quasi-isomorphism constructed in \cite{thesis,CEFT,K-alg} and some general
facts about formal differential equations.

On a symplectic manifold $(M,\omega)$,  equivalence classes of star
products quantizing the associated non-degenerate Poisson bracket
are classified by their \textit{Fedosov classes}, which are elements
in
\begin{equation}
\label{space-of-Fed-classes} \frac{1}{\h} [\om] +  H^2(M,
\bbC)[[\h]]\,,
\end{equation}
see, e.g., \cite{BCG,DeWilde,Deligne,F,Fedosov-index,NT,OMY}. Hence to
each star product $*$ on $(M,\omega)$ one may assign either its
Kontsevich class, defined by the class of a formal Poisson structure
$\pi=\h\pi_1 + \dots$, where $\pi_1$ is the Poisson bivector field
defined by the symplectic form $\omega$, or its Fedosov class in
\eqref{space-of-Fed-classes}. The nondegeneracy of $\pi_1$ implies
that the series $\pi=\h\pi_1 + \dots$ can be formally inverted to a
series of closed 2-forms, defining an element in
\eqref{space-of-Fed-classes}, and the fact that this element agrees
with the Fedosov class of the star product $*$ has been conjectured
by A. Chervov and L. Rybnikov in \cite[Conjecture~4]{CR}. In this
paper, we prove this conjecture, which allows us to recover the
description of Morita equivalent star products on symplectic
manifolds of \cite{BW1} as a particular case of our Theorem
\ref{ona}. The proof of this conjecture about the relationship
between Fedosov's and Kontsevich's classes in
Theorem~\ref{F-versus-K} also partially closes the project mentioned
in item 1) of \cite[Section~0.2]{K}.

We remark that the construction of the map $\cK_*$ in
\eqref{eq:Kmap} involves choices. We prove in Theorem~\ref{nezalezh}
that the definition of the Kontsevich classes of star products does
not depend on the choices made in the globalization procedure of
\cite{CEFT}. This definition may depend, however, on the specific
choice of formality quasi-isomorphism between polyvector fields and
polydifferential operators on $\bbR^d$\,. In this paper,  we tacitly
assume that the formality quasi-isomorphism on $\bbR^d$ is the one
constructed by M. Kontsevich in \cite{K} with the angle function
defined via hyperbolic geometry of the Lobachevsky plane.

Finally, we point out that there is an alternative construction of
global star products in \cite{CFT-global}, which we believe can also
be used to study Morita equivalence as well as to relate Fedosov's
and Kontsevich's classes of star products on symplectic manifolds.

Let us briefly describe the organization of the paper.

In Section \ref{sec:prelim}, we recall Fedosov's resolutions and the
globalization of Kontsevich's formality quasi-isomorphism
\cite{CEFT,thesis}, which we use to define the Kontsevich classes of
star products. We verify in Theorem \ref{nezalezh} (whose proof is
deferred to Appendix~\ref{App-C}) that the definition of the
Kontsevich classes is independent of the choices made in the
globalization procedure.

Section \ref{sec:Morita} is devoted to the description of the
Kontsevich classes of Morita equivalent star products. The key step
consists in verifying that the map $\cK_*$ in \eqref{eq:Kmap}
satisfies an equivariance property with respect to appropriate
actions of the Picard group, as explained in Theorem \ref{ona}.

In Section \ref{sec:fedosov} we describe the relationship between
Fedosov's and Kontsevich's classes of star products on symplectic
manifolds. The main result of this section is formulated in Theorem
\ref{F-versus-K},  and its proof is divided into several parts.
First, we introduce a modification of Fedosov's construction
\cite{F}. Second, we describe a version of the Emmrich-Weinstein
connection \cite{EW}, which is then used to show that Kontsevich
star products are equivalent to the star products constructed using
the modified Fedosov construction. Finally, we show that the
original Fedosov star products coincide with the star products
obtained using the modified Fedosov construction.

In the end of the paper, Appendix~\ref{App-A} collects the necessary
facts about formal differential equations, Appendix~\ref{App-B} recalls some key
facts about DGLAs, Maurer-Cartan elements and $L_\infty$-morphisms.
Finally, Appendix~\ref{App-C} contains the somewhat technical proof of
Theorem~\ref{nezalezh}.

\noindent

\textbf{Acknowledgements:} H.B. thanks CNPq and Faperj for financial
support. The work of V.D. was partially supported by the NSF grant
DMS 0856196, Regent's Faculty Fellowship, and the Grant for Support
of Scientific Schools NSh-8065.2006.2. V.D.'s part of the work was
done when his mother-in-law Anna Vainer was baby sitting little
Nathan. V.D. would like to thank his mother-in-law for her help. The
authors thank P. Etingof, M. Gualtieri and V. Rubtsov for useful
conversations, and G. Dito and P. Schapira for the invitation to the
meeting on Algebraic Analysis and Deformation Quantization in
Scalea, which facilitated our collaboration. We also thank the
institutions that hosted us at various stages of this project:
Freiburg University (H.B.), U. C. Riverside (H.B.), Northwestern
University (V.D.), the University of Chicago (V.D.), and IMPA
(S.W.).

%
% Notation and conventions
%

\subsection{Notation and conventions}
\label{subsec:notation}

Throughout this paper $M$ is a smooth real manifold, $\cO_M$ is the
sheaf of smooth \emp{complex-valued} functions on $M$, and $\cO(M)$
is the algebra of its global sections. The algebra of smooth
complex-valued polyvector fields is denoted by $\cX^\bul(M)$; it is
equipped with the Schouten-Nijenhuis bracket $[\, , \,]_{SN}$, which
is a degree $0$ Lie bracket for the shifted grading $\cX^{\bullet +
1}(M)$. The space of complex-valued differentiable forms is denoted
by $\Omb(M)$. For a sheaf of $\cO_M$-modules $\cG$ and an open
subset $U \subset M$, we denote by $\G(U, \cG)$ the vector space of
global sections of $\cG$ and by $\cO(U)$ the algebra of smooth
complex-valued functions on $U$\,. Furthermore, we denote by
$\Omb(U, \cG)$ the graded vector space of exterior forms on $U$ with
values in $\cG$.

For a vector $v$ in a graded vector space or a cochain complex $V$,
its degree is denoted by $|v|$. By the \emp{suspension} $\bs V$
of a graded vector space (or a cochain complex) $V$ we mean $\ve
\otimes V$, where $\ve$ is a one-dimensional vector space placed in
degree $+1$. The \emp{desuspension} $\bs^{-1} V$ is the inverse
operation.  Throughout this paper we use the Koszul rule of signs.
If $V$ is a graded vector space, we denote its symmetric algebra by
$S(V)$, whereas $S^k(V)$ is the $k$-th component of this algebra.

For a unital associative algebra $\cA$, we denote by $\Cbu(\cA)$ the
(normalized) Hochschild cochain complex of $\cA$ with coefficients
in $\cA$,
\begin{equation}
    \label{rule}
    \Cbu(\cA) = \Hom(\big(\cA/\bbC \, 1\big)^{\bul},\cA).
\end{equation}
The coboundary operator $\pa^{\Hoch}$ on (\ref{rule}) is given by
\begin{equation}
\label{pa-Hoch}
\begin{aligned}
(\pa^{\Hoch} P)(a_0, a_1, \dots, a_k) = & a_0 P(a_1, \dots, a_k)
- P(a_0 a_1, \dots, a_k) +\\
&  P(a_0, a_1 a_2, a_3, \dots , a_k) - \dots + (-1)^{k} P(a_0,
\dots , a_{k-2}, a_{k-1} a_k) +\\
&  (-1)^{k+1} P(a_0, \dots , a_{k-2}, a_{k-1}) a_k ,
\end{aligned}
\end{equation}
where $P\in C^k(\cA)$ and $a_i \in \cA$\,. In particular, for a
degree-zero cochain $P \in C^0(\cA) = \cA$, we have
\begin{equation}
\label{pa-Hoch-0}
(\pa^{\Hoch} P )(a_0) = a_0\,  P -  P\, a_0\,.
\end{equation}

The Hochschild cochain complex with the shifted grading $\sCbu(\cA)$
carries the structure of a differential graded Lie algebra (or
\textit{DGLA} for short). The differential is exactly the Hochschild
coboundary operator $\pa^{\Hoch}$ (\ref{pa-Hoch}) and the Lie
bracket is the well-known Gerstenhaber bracket\footnote{Note that
our sign convention for the Gerstenhaber bracket differs from the
standard one.} \cite{Ger},
\begin{equation}
\label{Gerst}
[Q_1, Q_2]_{G} =
 \sum_{i=0}^{k_1}(-1)^{(i+k_1) k_2}
Q_1(a_0,\,\dots , Q_2 (a_i,\,\dots,a_{i+k_2}),\, \dots,
a_{k_1+k_2}) -
(-1)^{k_1 k_2} (1 \leftrightarrow 2)\,,
\end{equation}
where $Q_i \in  C^{k_i+1}(\cA)$, and $a_j \in \cA$.

As usual in this subject, we use adapted versions of Hochschild
(co)chains for the algebra $\cO(M)$; we denote by $\Cbu(\cO_M)$ the
proper subcomplex of polydifferential operators in the full
Hochschild cochain complex of $\cO(M)$.

In this paper every DGLA $(\cL, d_{\cL}, [\,,\,]_{\cL})$ is equipped
with a complete descending filtration
\begin{equation}
    \label{Filt-cL}
    \dots \supset \cF^{-2} \cL
    \supset \cF^{-1} \cL
    \supset \cF^0 \cL \supset \cF^1 \cL \supset \dots\,, \qquad \cL =
    \lim_{n} \cL/\cF^n\cL\,.
\end{equation}
In most cases this filtration will be bounded from the left. We will
often use  a formal deformation parameter $\h$ to obtain a complete
descending filtration on $\cL$. For example, extending the field of
scalars $\bbC$ to the ring $\bbC[[\h]]$ of formal power series, we
obtain from the DGLA $\sCbu(\cA)$ (resp. the graded Lie algebra
$\cX^{\bul+1}(M)$) the DGLA $\sCbu(\cA)[[\h]]$ (resp. the graded Lie
algebra $\cX^{\bul+1}(M)[[\h]]$); the descending filtrations are
\[
\cF^k \sCbu(\cA)[[\h]] = \h^k \sCbu(\cA)[[\h]]\,, \qquad \cF^k
\cX^{\bul+1}(M)[[\h]] = \h^k \cX^{\bul+1}(M)[[\h]].
\]
We assume that every morphism $\ka: \cL \to \tcL $ of two such
DGLAs is compatible with the filtrations. In addition, every
quasi-isomorphism $\ka: \cL \stackrel{\sim}{\to} \tcL $ is assumed to
satisfy the following:
\begin{cond}
    \label{condition}
    The restriction of a quasi-isomorphism $\ka$ to each filtration
    subcomplex $\cF^m \cL^{\bul}$
    \[
    \ka\,  \Big|_{\cF^m \cL^{\bul}} \, : \,
    \cF^m \cL^{\bul} \to \cF^m \tcL^{\bul}
    \]
    is a quasi-isomorphism.
\end{cond}
We will also need $\Linf$ morphisms and $\Linf$ quasi-isomorphisms of
DGLAs. We recall them in Appendix \ref{App-B}; see Definitions
\ref{Linf-morph} and \ref{Linf-q-iso}. For $\Linf$ morphisms or
$\Linf$ quasi-isomorphisms between DGLAs, we reserve the arrow
$\brarrow$.

%
% Formality and star products
%

\section{Global formality and star products}
\label{sec:prelim}

%
% Fedosov's resolutions
%

\subsection{Fedosov's resolutions}
\label{subsec-Fedosov-resolutions}

We now briefly recall Fedosov's resolutions (see
\cite[Chapter~4]{thesis}) of polyvector fields, and Hochschild
cochains of $\cO(M)$. This construction has various incarnations,
and it is referred to as the Gelfand-Fuchs trick \cite{GF}, or
formal geometry \cite{G-Kazh} in the sense of Gelfand and Kazhdan,
or mixed resolutions \cite{Ye} of Yekutieli.

We denote by $x^i$ local coordinates on $M$ and by $y^i$ fiber
coordinates in the tangent bundle $TM$. We denote by $\SM$ the
formally completed symmetric algebra of the cotangent bundle $T^*M$.
We regard $\SM$ as a sheaf of algebras over $\cO_M$, whose sections
can be viewed as formal power series in tangent coordinates $y^i$\,.
In particular, $\Cbu(\SM)$ is the \emp{sheaf} of normalized
Hochschild cochains of $\SM$ over $\cO_M$. Namely, sections of
$C^k(\SM)$ over an open subset $U\subset M$ are $\cO(U)$-linear
maps\footnote{The sheaf $\Cbu(\SM)$ is the $y$-adic completion of
the
  sheaf $\cD^{\bul}_{poly}$ of fiberwise polydifferential operators.
  It is the latter sheaf that was used in \cite{thesis} (see
  Definition 12 on page 60), and it is not hard to see
  that the sheaf $\cD^{\bul}_{poly}$ can be replaced by its
  completion $\Cbu(\SM)$ in all the constructions of \cite{thesis}.}
\begin{equation}
    \label{oni}
    P : \G(U,\SM)^{\otimes \, k } \to \G(U,\SM),
\end{equation}
which are continuous in the $y$-adic topology on $\G(U, \SM)$ and
satisfy the normalization condition
\begin{equation}
\label{normalization}
P(\dots, 1, \dots) = 0\,.
\end{equation}
We let $\cTp$ be the sheaf of fiberwise polyvector fields, which is
the cohomology of the complex of sheaves $\Cbu(\SM)$ (see
\cite[page~60]{thesis}). The grading convention for 
$\cTp$ coincides with the one for $\X^{\bul}(M)$.

It is shown in \cite[Theorem~4]{thesis} that the algebra $\Omb(M,
\SM)$ can be equipped with a differential of the form
\begin{equation}
    \label{DDD}
    D = \n - \de + A,
\end{equation}
where
\begin{equation}
    \label{nabla-prelim}
    \n = d x^i \frac{\pa}{\pa x^i} - d x^i
    \G^k_{ij}(x) y^j \frac{\pa}{\pa y^k}
\end{equation}
is a torsion free connection with Christoffel symbols $\G^k_{ij}(x)$,
\begin{equation}
    \label{delta}
    \de = d x^i \frac{\pa}{\pa y^i}\,,
\end{equation}
and
\begin{equation}
    \label{A}
    A=\sum_{p=2}^{\infty}d x^k A^j_{ki_1\dots i_p}(x) y^{i_1}
    \dots y^{i_p}\frac{\pa}{\pa y^j} \in \Om^1(\cT^1_{poly})\,.
\end{equation}
\begin{remark}
    \label{remark:FedosovDifferential}
    Although the construction of the differential (\ref{DDD}) is very
    similar to the construction of the Fedosov differential in
    \cite{F}, they are not be confused; we refer to the differential (\ref{DDD})
    as a \emp{geometric Fedosov differential} while
    the original Fedosov differential is referred to as a \emp{quantum
      Fedosov differential}.
\end{remark}

Note that $\de$ in (\ref{delta}) is also a differential on $\Omb(M,
\SM)$, and (\ref{DDD}) can be viewed as a deformation of $\de$ via
the connection $\n$. Let us recall from \cite{thesis} the following
operator on $\Omb(M, \SM)$:
\begin{equation}
    \de^{-1}(a) =
    \begin{cases}
        \begin{array}{cc}
            \displaystyle y^k \frac {\vec{\partial}} {\partial (d x^k)}
            \int\limits_0^1 a(x, t y, t d x)\frac{d t} t, & {\rm if}~
            a \in \Om^{>0}(M, \SM)\,,\\
            0, & {\rm otherwise}\,.
        \end{array}
    \end{cases}
    \label{del-1}
\end{equation}
The arrow over $\pa$ in \eqref{del-1} means that
  we use the left derivative with respect to the ``anti-com\-muting''
  variable $d x^k$.
The operator \eqref{del-1} satisfies the following properties:
\begin{align}
    \label{nilp}
&    \de^{-1} \circ  \de^{-1} =0,\\
    \label{Hodge}
  &  a=\si(a) +\delta \delta ^{-1}a + \delta ^{-1}\delta a,
    \qquad
    \forall\;
    a\in \Omb(M, \SM),
\end{align}
where
\begin{equation}
    \label{sigma}
    \si (a)= a\Big|_{y^i=d x^i=0}.
\end{equation}
\begin{remark}
    \label{remark-Hodge}
    Following \cite[Chapter~4]{thesis}, we extend the operators
    $\de$, $\de^{-1}$ and $\si$ to $\Omb(M, \cTp)$ and $\Omb(M,
    \Cbu(\SM))$ so that \eqref{Hodge} holds also for all
    $a\in \Omb(M, \cTp)$, and for all $a \in \Omb(M, \Cbu(\SM))$\,.
\end{remark}

According to \cite[Prop.~10]{thesis} the sheaves $\cTp$ and
$\Cbu(\SM)$ are equipped with a canonical action of the sheaf of Lie
algebras $\cT_{poly}^1$, and this action is compatible with the
corresponding DGLA structures. Using this action in
\cite[Chp.~4]{thesis}, the Fedosov differential (\ref{DDD}) is
extended to differentials on $\Omb(M, \cTp)$ and $\Omb(M,
\Cbu(\SM))$. Propositions 13 and 14 in \cite{thesis} provide us with
a quasi-isomorphism from the graded Lie algebra
$\mathcal{X}^{\bul+1} (M)$ to the DGLA
\begin{equation}
\label{fiber-polyvect} (\Omb(M, \scTp), D, [\,,\,]_{SN}),
\end{equation}
and a quasi-isomorphism from the DGLA
$\sCbu(\cO_M)$ to the DGLA
\begin{equation}
\label{fiber-cochains}
(\Omb(M, \sCbu(\SM)), D + \pa^{\Hoch} , [\,,\,]_{G})\,.
\end{equation}

To construct these quasi-isomorphisms
 we recall that the restriction of the map $\si$
(\ref{sigma}) to the subspace of $D$-flat sections $\G(M, \SM) \cap
\ker D$ gives a bijection
\begin{equation}
    \label{si-biject}
    \si : \G(M, \SM) \cap \ker D \to \cO(M)
\end{equation}
onto the algebra of functions $\cO(M)$. The inverse map
\begin{equation}
    \label{tau}
    \tau : \cO(M) \to \G(M, \SM) \cap \ker D
\end{equation}
is defined by iterating the equation
\begin{equation}
    \label{iter-tau}
    \tau(f)
    = f + \de^{-1}(\n \tau(f)+ A\cdot \tau(f)),
    \qquad
    f \in \cO(M),
\end{equation}
in degrees in the fiber coordinates $y$'s.

One may verify that the map $\tau$ satisfies
\begin{equation}
    \label{tau-property}
    \frac{\pa}{\pa y^{i_1}} \dots \frac{\pa}{\pa
      y^{i_p}} \, \tau (f) \Big|_{y=0} = \pa_{x^{i_1}} \dots \pa_{x^{i_p}}
    f(x)+ {\rm lower~order~ derivatives~ of}~f.
\end{equation}

To proceed further, we need the subspace
\begin{equation}
    \label{G-de}
    \G_{\de}(M, \sCbu(\SM)) = \G(M, \sCbu(\SM)) \cap \ker
    \de
\end{equation}
of $\de$-flat cochains of the sheaf $\SM$\,. This subspace consists
of $\cO_M$-linear polydifferential operators on $\SM$ whose
coefficients do not depend on the fiber coordinates $y$'s. The
graded vector space $\G_{\de}(M, \sCbu(\SM))$ is, in fact,
isomorphic to the subspace
$$
\G(M, \sCbu(\SM)) \cap \ker D
$$
of $D$-flat sections of $\sCbu(\SM)$. The corresponding isomorphism,
\begin{equation}
\label{vro} \vro : \G_{\de}(M, \sCbu(\SM))
\stackrel{\cong}{\longrightarrow} \G(M, \sCbu(\SM)) \cap \ker D,
\end{equation}
is defined by iterating the equation
\begin{equation}
\label{iter-vro}
\vro(P) = P + \de^{-1} (\n \vro(P) + [A, \vro(P)]_G)\,, \qquad
P \in \G_{\de}(M, \sCbu(\SM))
\end{equation}
in degrees in the fiber coordinates $y$'s.

On the other hand, using $\tau$ (\ref{tau}) we construct the map
\begin{align}
\label{nu-map} & \nu : \G_{\de}(M, \sCbu(\SM)) \longrightarrow \sCbu(\cO_M),\\
%\end{equation}
%\begin{equation}
%    \label{nu}
   & \nu(P)(a_0, a_1, \dots, a_k) = P(\tau(a_0), \tau(a_1),
    \dots, \tau(a_k))\Big|_{y=0}\,,\nonumber
\end{align}
for $a_i \in \cO(M)$ and $P \in \G_{\de}(M, C^{k+1}(\SM))$. Due to
property (\ref{tau-property}), the map $\nu$ is also an isomorphism
of graded vector spaces.

Composing $\vro$ with $\nu^{-1}$, we get the isomorphism
\begin{equation}
    \label{tau-iso}
    \tau_{\mathrm{ext}} = \vro \circ \nu^{-1} : \sCbu(\cO_M)
    \stackrel{\cong}{\longrightarrow} \G(M, \sCbu(\SM))\cap \ker D,
\end{equation}
as well as the following embedding (for which we keep the same
notation):
\begin{equation}
    \label{tau-need}
    \tau_{\mathrm{ext}} = \tau \circ \nu^{-1} : \sCbu(\cO_M)
    \hookrightarrow \Omb(M, \sCbu(\SM)).
\end{equation}
Restricting (\ref{tau-iso}) and (\ref{tau-need}) to the graded Lie algebra
$\mathcal{X}^{\bul+1}(M)$  of polyvector fields,  we get the  
isomorphism
\begin{equation}
\label{tau-iso1}
\tau_{\mathrm{ext}} : \mathcal{X}^{\bul+1} (M)
\stackrel{\cong}{\longrightarrow} \G(M, \scTp) \cap \ker D
\end{equation}
and the embedding (for the both maps we keep the same notation $\tau_{\mathrm{ext}}$)
\begin{equation}
\label{tau-need1}
\tau_{\mathrm{ext}} : \mathcal{X}^{\bul+1} (M)
\hookrightarrow \Omb(M, \scTp)\,,
\end{equation}
respectively.
According to \cite[Chapter~4]{thesis} both maps (\ref{tau-need}) and
(\ref{tau-need1}) are compatible with the corresponding DGLA
structures. Furthermore, the acyclicity of $D$ in positive exterior
degrees implies that the maps (\ref{tau-need}) and (\ref{tau-need1})
are quasi-isomorphisms.

To simplify our notation, we shall denote all three maps
\eqref{tau}, \eqref{tau-need}, and \eqref{tau-need1} simply by
$\tau$, avoiding the notation $\tau_{\mathrm{ext}}$ henceforth. This
simplification does not lead to confusion because the restriction of
$\tau_{\mathrm{ext}}$ (\ref{tau-need}) to
$$
\cO(M) = C^0(\cO_M) = \mathcal{X}^{0}(M)
$$
coincides with $\tau$.

%
% A sequence of quasi-isomorphisms between $\cX^{\bul+1}(M)$ and $\sCbu(\cO_M)$
%%%%%%%%%%%%%%%%%%%%%%%%%%%%%%%%%%%%%%%%%%%%%%%%%%%%%%%%%%%%%%%%%%%%%%%%%%%%%%%%%

\subsection{A sequence of quasi-isomorphisms between $\cX^{\bul+1}(M)$
  and $\sCbu(\cO_M)$}
\label{subsec:sequence}

We now outline the construction of a sequence of quasi-isomorphisms
between the DGLAs $\mathcal{X}^{\bul+1}(M)$ and $\sCbu(\cO_M)$. The
construction makes use of Kontsevich's
$L_{\infty}$-quasi-isomorphism \cite{K} $K$ from the DGLA
$\X^{\bul+1}(\bbR^d)$ of polyvector fields on $\bbR^d$ to the DGLA
$\sCbu(\cO_{\bbR^d})$ of Hochschild cochains of $\cO_{\bbR^d}$ (see
Remark \ref{rem:teich}). Let us list some key properties of the
structure maps
\begin{equation}
    \label{K-n}
    K_n : \Big( \X^{\bul+1}(\bbR^d) \Big)^{\otimes \,
      n} \to \sCbu(\cO_{\bbR^d})[1-n]\,, \qquad n \ge 1
\end{equation}
of this $L_{\infty}$-quasi-isomorphism $K$:

\begin{property}
    \label{P-symmetry}
    $K_n(\dots, \ga_1, \ga_2, \dots) = -(-1)^{|\ga_1||\ga_2|}
    K_n(\dots, \ga_2, \ga_1, \dots)$\,, where $|\ga_i|$ is the degree
    of $\ga_i$ in the vector space $\X^{\bul+1}(\bbR^d)$ with
    the shifted grading.
\end{property}

\begin{property}
    \label{P-K1}
    The map $K_1 : \mathcal{X}^{\bul+1}(\bbR^d)\to \sCbu(\cO_{\bbR^d})$ coincides
    with the canonical embedding of the space of polyvectors into the
    space of polydifferential operators.
\end{property}
\begin{property}
    \label{P-gl-inv}
    The maps $K_n$ are $\mgl_d$ equivariant.
\end{property}
\begin{property}
    \label{P-vect-linear}
    $K_n(v, \dots ) = 0$ if $n\ge 2$ and $v$ is a vector field which
    depends linearly on the coordinates of $\bbR^d$.
\end{property}
\begin{property}
    \label{P-arg-vect}
    If $v_1, v_2, \dots, v_n$ are vector fields and $n \ge 2$ then
    $K_n(v_1, \dots, v_n) = 0$.
\end{property}
\begin{property}
    \label{P-constant}
    For every $n\ge 2$ we have $K_n(\dots, c) = 0$ if $c$ is a
    constant viewed as a degree-zero polyvector field $c\in
    \mathcal{X}^{0}(\bbR^d) = \cO(\bbR^d)$.
\end{property}

Properties P~\ref{P-symmetry} -- P~\ref{P-arg-vect} allow us to
construct an $L_{\infty}$-quasi-isomorphism
\begin{equation}
    \label{K-tw}
    K^{tw} : (\Omb(\scTp), D, [\,,\,]_{SN}) \brarrow
    (\Omb(\sCbu(\SM)), D + \pa^{\Hoch}, [\,,\,]_{G})\,
\end{equation}
as follows. For every coordinate neighborhood $U$ the part
\begin{equation}
    \label{mu-D-U}
    \mu^D_U = - dx^i \G^k_{ij}(x) y^j \frac{\pa}{\pa y^k}
    - \de + A
\end{equation}
of the geometric Fedosov differential (\ref{DDD}) may be viewed as a
Maurer-Cartan element of the DGLA
\[
(\Omb(U, \scTp), d, [\,,\,]_{SN}),
\]
where $d$ is the de Rham differential.  Kontsevich's $L_{\infty}$
quasi-isomorphism \cite{K} for $\bbR^d$ may be viewed as an
$L_{\infty}$ quasi-isomorphism
\begin{equation}
    \label{K-U}
    K : (\Omb(U, \scTp), d, [\,,\,]_{SN}) \brarrow (\Omb(U,
    \sCbu(\SM)), d + \pa^{\Hoch}, [\,,\,]_{G})\,.
\end{equation}
Using this $L_{\infty}$ quasi-isomorphism and equation (\ref{MC-F-al})
from Appendix \ref{App-B}, we can
send the Maurer-Cartan element $\mu^D_U$ to the Maurer-Cartan
element
\[
\sum_{n=1}^{\infty}\frac{1}{n!} K_n (\mu^D_U, \mu^D_U, \dots,
\mu^D_U)
\]
of the DGLA $(\Omb(U, \sCbu(\SM)), d + \pa^{\Hoch}, [\,,\,]_{G})$\,.
Properties P~\ref{P-K1} and P~\ref{P-arg-vect} imply that
\[
\sum_{n=1}^{\infty}\frac{1}{n!} K_n (\mu^D_U, \mu^D_U, \dots,
\mu^D_U) = \mu^D_U\,.
\]
Therefore, twisting $K$ by the Maurer-Cartan element $\mu^D_U$ (see
Appendix \ref{App-B}) we get an $L_{\infty}$ quasi-isomorphism
\begin{equation}
    \label{K-U-tw}
    K^{\mu^D_U} : (\Omb(U, \scTp), D, [\,,\,]_{SN})
    \brarrow (\Omb(U, \sCbu(\SM)), D + \pa^{\Hoch}, [\,,\,]_{G})\,.
\end{equation}
Properties P \ref{P-gl-inv} and P \ref{P-vect-linear} imply that
$K^{\mu^D_U}$ does not depend on the choice of trivialization of the
tangent bundle $TM$ over $U$. Hence we get an $L_{\infty}$
quasi-isomorphism (\ref{K-tw}) by setting
\begin{equation}
    \label{define-K-tw}
    K^{tw}_n(\ga_1, \ga_2, \dots, \ga_n)\Big|_{U} =
    K^{\mu^D_U}_n(\ga_1, \ga_2, \dots, \ga_n)\,.
\end{equation}
Combining $K^{tw}$ with the maps (\ref{tau-need}) and
(\ref{tau-need1}), we obtain a sequence of quasi-isomorphisms

\begin{equation}
    \label{upper}
        \mathcal{X}^{\bul+1}(M)
         \stackrel{\tau}{\lrarrow}
        \Omb(M, \scTp)
         \stackrel{K^{tw}}{\brarrow}
        \Omb(M, C^{\bul+1}(\SM) )
        \stackrel{\tau}{\llarrow}  \sCbu(\cO_M).
\end{equation}
%\begin{equation}
%    \label{upper}
%    \begin{array}{ccccc}
%        \mathcal{X}^{\bul+1}(M)
%        & \stackrel{\tau}{\lrarrow}&
%        (\Omb(M, \scTp), D, [\,,\,]_{SN})
%        & \stackrel{K^{tw}}{\brarrow} & ~ \\[0.3cm]
%        ~ & \stackrel{K^{tw}}{\brarrow}
%        & (\Omb(M, C^{\bul+1}(\SM) ), D+\pa, [\,,\,]_G) &
%        \stackrel{\tau}{\llarrow} & \sCbu(\cO_M).
%    \end{array}
%\end{equation}
Using \cite[Lemma~1]{erratum}, we can reduce the sequence
(\ref{upper}) to a single $L_{\infty}$ quasi-isomorphism
\begin{equation}
    \label{cK}
    \cK : \mathcal{X}^{\bul+1}(M) \brarrow \sCbu(\cO_M)
\end{equation}
between the DGLAs $\mathcal{X}^{\bul+1}(M)$ and $\sCbu(\cO_M)$.
More precisely, in \cite[Chapter 4, Eq. (4.36)]{thesis} a chain
homotopy is constructed which contracts the complex $(\Omb(M,
C^{\bul+1}(\SM)), D + \pa^{\Hoch}) $ to its sub-DGLA
\begin{equation}
    \label{D-flat-subcomplex}
    (\G(M, C^{\bul+1}(\SM)) \cap \ker D  ,
    \pa^{\Hoch}, [\,,\,]_G)\,.
\end{equation}
Using this chain homotopy and \cite[Lemma~1]{erratum}, one
constructs an $\Linf$ quasi-isomorphism
\begin{equation}
    \label{tcK}
    \tcK : \mathcal{X}^{\bul+1}(M)
    \brarrow
    (\Omb(M, C^{\bul+1}(\SM) ), D+\pa^{\Hoch}, [\,,\,]_G)
\end{equation}
satisfying two properties: first, all the structure maps $\tcK_{n}$
of $\tcK$ take values in the sub-DGLA (\ref{D-flat-subcomplex});
second, $\tcK$ is homotopy equivalent to the composition
\[
K^{tw} \circ \tau :  \mathcal{X}^{\bul+1}(M) \brarrow (\Omb(M,
C^{\bul+1}(\SM) ), D+\pa^{\Hoch}, [\,,\,]_G)\,.
\]
Composing $\tcK$ with the inverse of the isomorphism \eqref{tau-iso}
we get the desired $\Linf$ quasi-isomorphism $\cK$ (\ref{cK}). Going
through details of this construction and using the fact that the
structure maps \eqref{K-n} land in normalized Hochschild cochains,
one verifies the following:
\begin{pred}
\label{cK-constant} For every constant $c$, viewed as a degree zero
polyvector field $c\in \mathcal{X}^{0}(M) = \cO(M)$, we have
    \[
    \cK_1(c) = c, \;\;\; \mbox{ and } \;\;\;
    \cK_n(\dots, c) = 0, \;\;\; \forall ~ n\ge 2.
    \]
\end{pred}
\begin{remark}
\label{remark:adding-hbar}
From now on we extend the field of scalars $\bbC$ to the ring
$\bbC[[\h]]$ in all our constructions. In other words, we replace the DGLAs
$\mathcal{X}^{\bul+1}(M)$, $\Omb(M, \scTp)$, $\Omb(M,
C^{\bul+1}(\SM))$, and $\sCbu(\cO_M)$ in (\ref{upper}) with
\begin{equation}
\label{DGLAs-hbar}
    \mathcal{X}^{\bul+1}(M)[[\h]],
    \quad
    \Omb(M, \scTp)[[\h]],
    \quad
    \Omb(M, C^{\bul+1}(\SM))[[\h]],
    \qquad \sCbu(\cO_M)[[\h]],
\end{equation} 
    respectively.
We also replace the connection form $\G$ in (\ref{nabla-prelim})
by a general formal Taylor power series in $\h$:
    \[
    \G_{\h} = \G_0 + \h \G_1 + \h^2 \G_2 + \dots,
    \]
    where $\G_0$ is an ordinary torsion free connection form and
    $\G_1$, $\G_2$, $\dots$ are global sections of $TM\otimes
S^{2}(T^*M)$\,. Finally, we allow $A$ (\ref{A}) to have
the more general form:
\begin{equation}
        \label{A-more-gen}
        A=\sum_{p=2, r=0}^{\infty} d x^k  \h^r A^j_{r; k
          i_1\dots i_p}(x) y^{i_1} \dots y^{i_p}\frac{\pa}{\pa y^j} \in
        \Om^1(M, \cT^1_{poly})[[\h]]\,.
\end{equation}
It is not hard to see that the constructions 
described in Subsection \ref{subsec-Fedosov-resolutions} 
and in this subsection can be generalized to this setting. 
Furthermore, the resulting $\Linf$ quasi-isomorphisms (\ref{upper}) 
connecting the DGLAs in (\ref{DGLAs-hbar}) agree with 
the $\h$-adic filtration and satisfy Condition \ref{condition}.
\end{remark}
\begin{remark}
\label{remark-filtr}
    We will need the following complete descending filtration on
    the sheaf of algebras $\SM[[\h]]$ (see Appendix \ref{App-C}):
    \begin{equation}
        \label{filtr-SM-h}
        \SM[[\h]] = \cF^0 \SM[[\h]] \supset \cF^1
        \SM[[\h]] \supset \cF^2 \SM[[\h]] \supset \dots ,
     \end{equation}
    where local sections of $\cF^m \SM[[\h]]$ are the series
    \[
    a = \sum_{2k+l \ge m} \h^k \, a_{k; i_1, \dots, i_l} (x)\, y^{i_1}
    y^{i_2} \dots y^{i_l}\,.
    \]
\end{remark}

%
% Star products and their equivalence classes
%%%%%%%%%%%%%%%%%%%%%%%%%%%%%%%%%%%%%%%%%%%%%%%%%%%%%%%%%%%%%%%%%%%%%%%%

\subsection{Star products and their equivalence classes}
\label{subsec:star-products}

A \emp{star product} \cite{Bayen,Ber} on a manifold $M$ is a
$\bbC[[\h]]$-linear associative product on $\cO(M)[[\h]]$ of the form
\begin{equation}
    \label{star-intro}
    f * g = fg + \sum_{k=1}^{\infty} \h^k \Pi_k(f, g),
\end{equation}
where $f, g \in \cO(M)[[\h]]$ and $\Pi_k \in C^2(\cO_M)$, i.e.,
$\Pi_k$ are normalized bidifferential operators. Because we deal
with normalized Hochschild cochains (\ref{rule}), star products
satisfy
\begin{equation}
\label{star-and-unit} f * 1 = 1 * f = f\,.
\end{equation}
Since $f * g = fg ~{\rm mod}~ \h$, star products should be viewed as
an associative (but not necessarily commutative) formal deformation
of the ordinary product of functions on $M$.

Two star products $*$ and $*'$ are \emp{equivalent} if there exist
(normalized) differential operators $T_i: \cO(M)\to \cO(M)$, $i=1,2,\ldots$, so
that the formal series $T = \id + \h T_1 + \h^2 T_2 + \dots$
intertwines the star products,
\begin{equation}
    \label{equiv-intro}
    T(f * g) = T(f) \mathbin{*'} T(g).
\end{equation}
We denote the set of equivalence classes of star products on $M$ by
$\Def(M)$.

The associativity property of a star product \eqref{star-intro} can
be equivalently expressed as the Maurer-Cartan equation
\[
\pa^{\Hoch} \Pi + \frac{1}{2}[\Pi,\Pi]_G = 0
\]
for the formal series of bidifferential operators $\Pi :=
\sum_{k=1}^{\infty} \h^k \Pi_k \in C^2(\cO_M)[[\h]]$.
Thus Maurer-Cartan elements of the DGLA $\sCbu(\cO_M)[[\h]]$ (see
Appendix \ref{App-B}) are exactly the star products on $M$.
Furthermore, one can verify that equivalent Maurer-Cartan elements
of $\sCbu(\cO_M)[[\h]]$ (in the sense of \eqref{action-A})
correspond to equivalent star products.

Maurer-Cartan elements in $\mathcal{X}^{\bul+1}(M)[[\h]]$ are formal
series of bivector fields in $\h$,
\begin{equation}
    \label{pi}
    \pi = \h \pi_1 + \h^2 \pi_2 + \dots
    \in \h\, \mathcal{X}^2(M)[[\h]]
\end{equation}
satisfying the equation
\begin{equation}
    \label{MC-pi}
    [\pi, \pi]_{SN} = 0\,.
\end{equation}
We refer to Maurer-Cartan elements in
$\mathcal{X}^{\bul+1}(M)[[\h]]$ as \emp{formal Poisson structures},
by analogy with the usual definition of Poisson structures in
geometry (cf. \cite{Dufour-Zung,Lich,Alan}); further properties of formal
Poisson structures are discussed in Section~\ref{subsec:gauge}.

As recalled in Appendix~\ref{App-B}, the prounipotent group of formal
diffeomorphisms,
\begin{equation}
    \label{group-pvect}
    \mG(\mathcal{X}^{\bul+1}(M)[[\h]])
    = \exp \left(\h \mathcal{X}^1(M)[[\h]]\right),
\end{equation}
acts on Maurer-Cartan elements of $\mathcal{X}^{\bul+1}(M)[[\h]]$
according to
\begin{equation}
    \label{action-pi}
   \pi^{\exp(X)} = \exp([\cdot, X]_{SN}) \pi,
\end{equation}
where $X \in \h \mathcal{X}^1(M)[[\h]]$. Two formal Poisson
structures $\pi$ and $\tpi$ are said to be \emp{equivalent} if they
lie in the same orbit of this action. We denote the set of
equivalence classes of formal Poisson structures by\footnote{In the
notation of Appendix \ref{App-B},
\begin{equation}
    \label{eq:Fpois}
    \Def(M) = \pi_0(\MC(\sCbu(\cO_M)[[\h]])), \;\;\;\; \FPois(M) = \pi_0(\MC(\mathcal{X}^{\bul+1}(M)[[\h]])).
\end{equation}} $\FPois(M)$.
The equivalence class of a star product $*$ and of a formal Poisson structure
$\pi$ will be denoted by
\[
[*] \in \Def(M)
\quad
\textrm{and}
\quad
[\pi]\in \FPois(M),
\]
respectively.

The $\Linf$ quasi-isomorphism $\cK$ in \eqref{cK} establishes,
according to \eqref{MC-F-al}, a correspondence between formal Poisson
structures and star products on $M$,
\begin{equation}
    \label{star-K}
    \pi \mapsto *_K,
    \;\;\mbox{ where }\;\;
    f *_{K} g =
    fg
    + \sum_{n=1}^{\infty} \cK_n(\pi, \pi, \dots, \pi)(f, g),
\end{equation}
for $f, g \in \cO(M)[[\h]]$. We refer to $*_K$ as the \emp{Kontsevich
  star product} associated with $\pi$. By
Proposition~\ref{MC-teo-Linf}, the correspondence \eqref{star-K}
induces a bijection
\begin{equation}
    \label{eq:bijec}
    \cK_*: \FPois\rightarrow \Def(M),
\end{equation}
associating to each equivalence class of formal Poisson structures 
an equivalence class of star products. 
We call the class $[\pi]=\cK_*^{-1}([*])$ in
$\FPois$ \emp{Kontsevich's class} of the star product $*$.

Regarding the choices involved in the definition of Kontsevich's
classes, we first observe that the bijection $\cK_*$ agrees with the
bijection induced by the sequence of $\Linf$ quasi-isomorphisms
\eqref{upper}; indeed, the fact that shortening the sequence
(\ref{upper}) using \cite[Lemma~1]{erratum} does not change the
correspondence between
    equivalence classes of Maurer-Cartan elements follows from Lemma~\ref{lemma-equiv} in Appendix~\ref{App-B}.
On the other hand, as discussed in Section~\ref{subsec:sequence},
the middle $\Linf$ quasi-isomorphism in the sequence (\ref{upper})
requires the choice of a Fedosov differential \eqref{DDD}. As shown
by the next result, this choice does not affect the Kontsevich class
of a star product.
\begin{teo}
    \label{nezalezh}
    The map $\cK_*$ in \eqref{eq:bijec} does not depend on the choice
    of the Fedosov differential.
    %Furthermore, it agrees with the bijection $\FPois\rightarrow \Def(M)$ induced from \eqref{upper}.
\end{teo}

The proof of Theorem~\ref{nezalezh} is in Appendix~\ref{App-C}.

\begin{remark}\label{rem:teich} We note that the constructions in Section~\ref{subsec:sequence} (and hence
the notion of Kontsevich's class) are based on Kontsevich's $\Linf$
quasi-isomorphism $K$ \cite{K} from $\mathcal{X}^{\bul+1}(\bbR^d)$
to $\sCbu(\cO_{\bbR^d})$, which is fixed throughout this paper. It is
  known \cite{K-conjecture} that there are other $\Linf$
  quasi-isomorphisms from $\mathcal{X}^{\bul+1}(\bbR^d)$ to
  $\sCbu(\cO_{\bbR^d})$ which are not homotopy equivalent to $K$; In
  fact, Tamarkin's proof \cite{DTT,Hinich,Dima-Proof} of Kontsevich's
  formality theorem indicates that homotopy equivalence classes of
  such $\Linf$ quasi-isomorphisms are acted upon by the
  Grothendieck-Teichm\"uller group introduced by Drinfeld in
  \cite{Drinfeld}.
\end{remark}

%
% The characteristic classes of Morita equivalent star products
%

\section{The characteristic classes of Morita equivalent star
  products}
\label{sec:Morita}

Two unital rings are called \emp{Morita equivalent} if they have
equivalent categories of modules \cite{Morita}. We view
star-product algebras on a manifold $M$ as unital algebras over the
ground ring $\mathbb{C}[[\h]]$ and consider the problem of describing
their Morita equivalence classes. Following \cite{B, BW}, this
classification is given by the orbits of a canonical action of
$\Dif(M)\ltimes \Pic(M)$ on the space $\Def(M)$. Here $\Dif(M)$ is the
group of diffeomorphisms of $M$, and $\Pic(M)\cong H^2(M,\mathbb{Z})$
is its \emp{Picard group}, i.e., the group of isomorphism classes
of complex line bundles over $M$; $\Dif(M)\ltimes \Pic(M)$ is the
semi-direct product group with respect to the action of $\Dif(M)$ on
$\Pic(M)$ by pull-back. We will briefly recall how this action is
defined, and then give its explicit description in terms of
Kontsevich's classes.

%
% An action of the Picard group on star products
%

\subsection{An action of the Picard group on star products}
\label{subsec:pic}

Given a diffeomorphism $\varphi: M \to M$ and a star product $*$, we
obtain a new star product $*_\varphi$,
\[
f *_\varphi g
= (\varphi^{-1})^*(\varphi^* f * \varphi^* g),
\]
where $f, g \in \cO(M)$, and this induces an action
\begin{equation}
    \label{eq:difact}
    \Dif(M)\times \Def(M)  \to \Def(M),\;\; (\varphi,[*]) \mapsto
    [*_\varphi].
\end{equation}
It is known \cite{Gutt} that every isomorphism between two star-product
algebras on $M$ is a composition of an equivalence
\eqref{equiv-intro} with an element in $\Dif(M)$ (viewed as an
automorphism of $\cO(M)$ via pull-back).
 This gives a simple
interpretation of the action \eqref{eq:difact}: the classes of $*$ and
$*'$ in $\Def(M)$ are in the same $\Dif(M)$-orbit if and only if the
two star-product algebras are isomorphic.

The space $\Def(M)$ also carries a natural action of $\Pic(M)$
\cite{B}. Given a complex line bundle $L\to M$, we view $\Gamma(M, L)$ as
a right module over $\cO(M)$. We denote by
$\End(\Gamma(M, L))=\Gamma(M, \End(L))$ the algebra of endomorphisms of this
module, noticing that there is a canonical identification
\begin{equation}
    \label{eq:canonical}
    \End(\Gamma(M, L)) \cong \cO(M).
\end{equation}
As shown in \cite{BW0}, for a given star product $*$ on $M$, there is
a unique way (up to equivalence) of deforming this module structure to
make $\Gamma(M, L)[[\hbar]]$ a right module over the star-product algebra
$(\cO(M)[[\h]], *)$.
For $s\in \Gamma(M, L)$, $f \in \cO(M)$ we denote the deformed
module structure by
\[
s \bullet f = s f  \mod \h,
\]
and write $\mathrm{End}(\Gamma(M, L)[[\hbar]],\bullet)$ for the algebra
of endomorphisms of this module. One can always find an identification
of $\cO(M)[[\h]]$ with $\mathrm{End}(\Gamma(M, L)[[\hbar]],\bullet)$ as
$\mathbb{C}[[h]]$-modules, in such a way that for $\h = 0$ one
recovers the natural identification \eqref{eq:canonical}. As a
consequence, we obtain a star product $*'$ on $M$ for which
\begin{equation}
    \label{eq:ident}
    (\cO(M)[[\h]],*') \cong
    \mathrm{End}(\Gamma(M, L)[[\hbar]],\bullet)
\end{equation}
is an isomorphism of $\mathbb{C}[[\h]]$-algebras deforming
\eqref{eq:canonical}. This construction defines an action
\begin{equation}
    \label{eq:actionL}
    \Phi: \Pic(M)\times \Def(M) \to \Def(M), \;\;\; (L,[*]) \mapsto
    \Phi_L([*]) = [*'],
\end{equation}
where $*'$ is characterized, up to equivalence, by \eqref{eq:ident}.

The actions \eqref{eq:difact} and \eqref{eq:actionL} provide a
convenient characterization of Morita equivalence for star products
\cite{B}: Two star products $*$ and $*'$ on $M$ are Morita
equivalent if and only if their classes in $\Def(M)$ are related by
the actions \eqref{eq:difact} and \eqref{eq:actionL}, i.e., $[*']=
\Phi_L([*_\varphi])$ for some line bundle $L$ and diffeomorphism
$\varphi$.

It will be useful to describe the action \eqref{eq:actionL} in terms
of transition functions. Let us consider a star product $*$, a line
bundle $L\to M$, and a deformed right-module structure $\bullet$ on
$\Gamma(M, L)[[\h]]$ over $(\cO(M)[[\h]], *)$.  Let $\{U_\alpha\}$
be a cover of $M$ by contractible open subsets.  We can define local
$\cO(U_\alpha)$-linear trivialization isomorphisms $\psi_\alpha :
\Gamma(U_{\al}, L)\to \cO(U_\alpha)$ and transition functions
$g_{\alpha\beta}\in \cO(U_\alpha\cap U_\beta)$ such that
$\psi_\alpha(\psi_\beta^{-1})(f) (x)=g_{\alpha\beta}(x) f(x)$, which
satisfy $g_{\alpha\beta}^{-1}=g_{\beta\alpha}$ and, on triple
intersections, the cocycle condition
\[
g_{\alpha\beta}g_{\beta\gamma}g_{\gamma\alpha}=1.
\]
As shown in \cite{BW1}, one can always find
$\mathbb{C}[[\hbar]]$-linear deformed trivialization isomorphisms
$\Psi_\alpha= \psi_\alpha \mbox{ mod } \h:
\Gamma(U_{\al}, L)[[\hbar]]\to \cO(U_\alpha)[[\hbar]]$ satisfying
\[
\Psi_\alpha(s\bullet f)=\Psi_\alpha(s) * f,
\]
and define deformed transition functions $G_{\alpha\beta}=
g_{\alpha\beta} + \mbox{ mod } \h \in \cO(U_\alpha \cap
U_\beta)[[\hbar]]$ such that
\[
\Psi_\alpha(s)= G_{\alpha\beta}*\Psi_{\beta}(s).
\]
Since $1$ is the unit for $*$ (see (\ref{star-and-unit})), it is
clear that
\begin{equation}
    \label{eq:defcocyc}
    G_{\alpha\alpha}=1,\;\;\; G_{\alpha\beta}*G_{\beta\alpha} = 1, \;\;
    G_{\alpha\beta}* G_{\beta\gamma} * G_{\gamma\alpha}=1.
\end{equation}

Let us consider a $\mathbb{C}[[\h]]$-algebra isomorphism
\[
T:(\cO(M)[[\h]],*')\to \End(\Gamma(M, L)[[\hbar]],\bullet),
\]
coinciding with \eqref{eq:canonical} at the classical limit $\h=0$.
The isomorphism $T$ is totally determined by a collection of local equivalences
$T_\alpha: (\cO(U_\alpha)[[\h]],*')\to (\cO(U_\alpha)[[\h]],*)$
satisfying
\begin{equation}
    \label{eq:localT-old}
    T_\alpha\circ T_\beta^{-1} (f)
    =  G_{\alpha \beta} * f * G_{\beta \alpha},
\end{equation}
for $f \in \cO(U_\alpha\cap U_\beta)$. One recovers $T$ from the
collection $\{T_\alpha\}$ by
\begin{equation}
    \label{eq:Talpha}
    T(f) (s) = \Psi_\alpha^{-1}(T_\alpha(f) * \Psi_\alpha(s)),
    \;\;\;
    s\in\Gamma(U_{\al}, L).
\end{equation}
\begin{prop}
    \label{prop:crit}
    Let $*$ and $*'$ be star products on $M$. The following are
    equivalent:
    \begin{itemize}
    \item[(i)] The star products $*$ and $*'$ are related by
        \eqref{eq:actionL}, i.e., there exists a line bundle $L\to M$ for
        which $\Phi_L([*])=[*']$;
    \item[(ii)] There exists an open cover $\{U_\alpha\}$ of $M$, with
        trivialization maps $\psi_\alpha$ and transition functions
        $g_{\alpha\beta}$ for $L\to M$, as well as deformed transition
        functions $G_{\alpha\beta}=g_{\alpha\beta} \mbox{ mod } \h \in
        \cO(U_\alpha\cap U_\beta)[[\h]]$ satisfying the cocycle
        conditions \eqref{eq:defcocyc} and a collection of
        equivalences $T_\alpha: (\cO(U_\alpha)[[\h]],*') \to
        (\cO(U_\alpha)[[\h]],*)$ for which the compatibility
        \eqref{eq:localT-old} holds.
    \end{itemize}
\end{prop}
\begin{proof}
    The implication $(i) \implies (ii)$ was already discussed. We
    explain how $(ii)$ implies $(i)$.  We first show that we can find
    local deformed trivializations $\Psi_\alpha=\psi_\alpha \mbox{ mod
    } \h$ such that
    \begin{equation}
        \label{eq:compat}
        \Psi_\alpha \Psi_\beta^{-1}(f)=G_{\alpha\beta}* f
    \end{equation}
    for all $f \in \cO(U_\alpha \cap U_\beta)[[\h]]$. We know from
    \cite{BW1} that we can find deformed trivializations
    $\Psi_\alpha'=\psi_\alpha \mbox{ mod } \h$; let us define
    $G_{\alpha\beta}'$ by $G_{\alpha\beta}' * f =
    \Psi_\alpha'(\Psi_\beta')^{-1}(f)$. We now modify $\Psi'_\alpha$
    to obtain $\Psi_\alpha$ satisfying \eqref{eq:compat}.  Let
    $\{\chi_\alpha\}$ be a partition of unity on $M$ subordinated to
    $\{ U_\alpha\}$. Consider
    \[
    S_\alpha = \sum_\gamma G_{\alpha\gamma}'*\chi_\gamma *
    G_{\gamma\alpha},
    \]
    viewed as an element in $\cO(U_\alpha)[[\h]]$ (note that each
    summand has a natural extension from $\cO(U_\alpha\cap
    U_\beta)[[\h]]$ to $\cO(U_\alpha)[[\h]]$). Note also that
    $S_\alpha$ is invertible with respect to $*$, since $S_\alpha = 1
    \mbox{ mod } \h$. Finally note that, using the cocycle conditions
    for $G_{\alpha\beta}'$ and $G_{\alpha\beta}$, we have
    \[
    S_\alpha * G_{\alpha\beta}
    = \sum_\gamma
    G_{\alpha\gamma}'*\chi_\gamma * G_{\gamma\alpha} * G_{\alpha\beta}
    = \sum_\gamma G'_{\alpha\beta}*G'_{\beta \gamma} * \chi_\gamma *
    G_{\gamma\beta}
    = G'_{\alpha\beta}* S_\beta.
    \]
    In other words, $G_{\alpha\beta} = S_\alpha^{-1} *
    G_{\alpha\beta}' * S_\beta$. Let us now define $\Psi_{\alpha}$ by
    $\Psi_{\alpha}(s)= S_{\alpha}^{-1} * \Psi_\alpha'(s)$. Then
    \[
    \Psi_{\alpha}\Psi_{\beta}^{-1}(f)
    = S_\alpha^{-1} * G'_{\alpha\beta} * S_\beta * f
    = G_{\alpha\beta} * f,
    \]
    as desired. We now use the local equivalences $T_\alpha$ and
    $\Psi_{\alpha}$ to define an isomorphism $T: (\cO(M),*')\to
    \End(\Gamma(M, L)[[\h]],\bullet)$ via \eqref{eq:Talpha}.
\end{proof}

%
% An action of closed 2-forms on formal Poisson structures
%

\subsection{An action of closed 2-forms on formal Poisson structures}
\label{subsec:gauge}

The description of how formal Poisson structures are acted upon by
closed 2-forms is a simple adaptation of the discussion of
\emp{gauge transformations} in \cite[Sec.~3]{SeveraWeinstein}; in
the context of generalized complex geometry, the same operation
appears under the name of \emp{B-field transform}, see e.g.
\cite[Sec.~3]{Gualtthesis}. We start by recalling standard facts and
alternative views of formal Poisson structures.

Given a formal bivector field $\pi = \sum_{k=1}^\infty \h^k \pi_k \in
\h\cX^2(M)[[\h]]$, we consider the $\mathbb{C}$-bilinear brackets
\[
\{\cdot,\cdot\}_k: \cO(M)\times \cO(M) \to \cO(M), \; \;
\{f, g\}_k = \pi_k(df, dg),
\]
and the induced $\mathbb{C}[[\h]]$-bilinear bracket
$\{\cdot,\cdot\}_\pi: \cO(M)[[\h]] \times \cO(M)[[\h]] \to
\cO(M)[[\h]]$, uniquely determined by
\begin{equation}
    \label{eq:brack}
    \{f, g \}_\pi
    = \pi(df, dg)
    = \sum_{k=1}^\infty\h^k \{f, g\}_k, \qquad
    f, g \in \cO(M).
\end{equation}
Let $\mathrm{Jac}_\pi: \cO(M)[[\h]]\times \cO(M)[[\h]] \times
\cO(M)[[\h]] \to \cO(M)[[\h]]$ be given by
\begin{equation}
    \label{eq:jacpi}
    \mathrm{Jac}_\pi(f, g, h)
    = \{f, \{g, h\}_\pi\}_\pi
    + \{h, \{f, g\}_\pi\}_\pi
    + \{g, \{h, f\}_\pi\}_\pi.
\end{equation}
We also consider the $\cO(M)[[\h]]$-linear map
\begin{equation}
    \label{eq:pisharp}
    \pi^\sharp: \Omega^1(M)[[\h]] \to \h \cX^1(M)[[\h]],
    \;\;\;\;
    \pi^\sharp(\xi)
    =
    \sum_{k=1}^\infty \h^k \pi_k^\sharp(\xi),
\end{equation}
where $\pi^\sharp_k(\xi) = i_\xi \pi_k$ for $\xi \in \Omega^1(M)$, and
its unique extension (as an algebra homomorphism)
\begin{equation}
    \label{eq:pisharpext}
    \pi^\sharp: \Omega^\bullet(M)[[\h]] \to \h \cX^\bullet(M)[[\h]].
\end{equation}
\begin{prop}
    \label{prop:poissonprop}
    Let $\pi \in \h\cX^2(M)[[\h]]$ and $\partial_\pi =
    [\pi,\cdot]_{SN}$. Then
    \begin{equation}
        \label{eq:jac}
        \frac{1}{2}[\pi,\pi]_{SN}(df, dg, dh)
        = \mathrm{Jac}_\pi(f, g, h)
        = \partial_\pi^2(f)(dg, dh),
    \end{equation}
    for $f, g, h \in \cO(M)[[\h]]$.
\end{prop}
\begin{proof}
    It suffices to verify \eqref{eq:jac} for $f, g, h \in \cO(M)$.
    Given bivector fields $\pi_k, \pi_l \in \cX^2(M)$, let
    $\mathrm{Jac}_{k,l}: \cO(M)\times \cO(M)\times \cO(M)\to \cO(M)$
    be defined by
    \[
    \mathrm{Jac}_{k,l}(f, g, h)
    = \{f, \{g, h\}_k \}_l
    + \{h, \{f, g\}_k \}_l
    + \{g, \{h, f\}_k \}_l.
    \]
    The Schouten bracket satisfies (see e.g. \cite{Dufour-Zung})
    \[
    [\pi_k,\pi_l]_{SN}(df, dg, dh)
    = \mathrm{Jac}_{k,l}(f, g, h) + \mathrm{Jac}_{l,k}(f, g, h).
    \]
    As a result, the $n^{th}$-order term in $\h$ of
    $\frac{1}{2}[\pi,\pi](df, dg, dh)$ is
    \begin{equation}
        \label{eq:norder}
        \sum_{i=1}^{n-1} \mathrm{Jac}_{i,n-i}(f, g, h),
    \end{equation}
    which agrees with the $n^{th}$-order term in $\h$ of $\{f, \{g,
    h\}\} + \{h, \{f, g\}\} + \{g, \{h, f\}\}$, proving the first
    equality in \eqref{eq:jac}.  For the second equality, recall that
    the Schouten bracket satisfies
    \begin{equation}
        \label{eq:schouten}
        [\pi_l, f]_{SN} = - \pi^\sharp_l(df),
        \;\;\;\; [\pi_l, X]_{SN} = -\Lie_X\pi_l,
    \end{equation}
    for $f \in \cO(M)$, $X \in \cX^1(M)$. A direct computation shows
    that
    \[
    [\pi_k, [\pi_l, f]_{SN}]_{SN}(dg, dh)
    = \{f, \{g, h\}_k\}_l
    + \{h, \{f, g\}_{l}\}_{k}
    + \{g, \{h, f\}_{l}\}_{k},
    \]
    and, as a consequence, the $n^{th}$-order term in $\h$ of
    $\partial_\pi^2(f)(dg, dh)$ coincides with \eqref{eq:jac}.
\end{proof}
\begin{cor}
    \label{cor:chain}
    If $[\pi,\pi]_{SN} = 0$, then the map $\pi^\sharp$ in
    \eqref{eq:pisharpext} satisfies $\pi^\sharp \circ d = -
    \partial_\pi \circ \pi^\sharp$.
\end{cor}
\begin{proof}
    It suffices to verify that $\pi^\sharp \circ d = - \partial_\pi
    \circ \pi^\sharp$ holds on elements $f$ and $df$, for $f \in
    \cO(M)$. The fact that $\pi^\sharp (df) = - \partial_\pi f$
    directly follows from the first equation in \eqref{eq:schouten}.
    On the other hand, since $[\pi,\pi]_{SN} = 0$, we have
    \[
    - \partial_\pi(\pi^\sharp(df))
    = \partial_\pi([\pi, f]_{SN})
    = \partial_\pi^2 (f) = 0,
    \]
    which agrees with $\pi^\sharp(d^2 f) =0$.
\end{proof}

To describe the action by closed 2-forms, it is convenient to have an
alternative viewpoint to formal Poisson structures, in the spirit of
Dirac geometry \cite{Cour}. We consider the bundle
\[
E := TM\oplus T^*M,
\]
equipped with the symmetric $\cO(M)$-bilinear pairing
$\SP{\cdot,\cdot}: \Gamma(M, E)\times \Gamma(M, E)\to \cO(M)$,
\begin{equation}
    \label{eq:pairing}
    \SP{(X, \xi), (Y, \eta)} = \eta(X) + \xi(Y),
\end{equation}
and the $\mathbb{C}$-bilinear bracket $\Cour{\cdot,\cdot}:
\Gamma(M, E)\times\Gamma(M, E)\to \Gamma(M, E)$,
\begin{equation}
    \label{eq:courantbk}
    \Cour{(X, \xi), (Y, \eta)}
    = ([X,Y], \Lie_X\eta - i_Y d\xi),
\end{equation}
known as the \emp{Courant bracket}. Here $X, Y \in \cX^1(M)$ and
$\xi, \eta \in \Omega^1(M)$. Using $\h$-linearity, we extend these
operations to
\begin{equation}
    \label{eq:operat}
   \begin{aligned}
    \SP{\cdot,\cdot}:
    \Gamma(M,E)[[\h]] \times \Gamma(M, E)[[\h]]
    \to
    \cO(M)[[\h]], \\[0.3cm]
    \Cour{\cdot,\cdot}:
    \Gamma(M, E)[[\h]] \times \Gamma(M, E)[[\h]]
    \to \Gamma(M, E)[[\h]].
    \end{aligned}
\end{equation}
The following is a straightforward observation.
\begin{lem}
    \label{lem:skew}
    A $\cO(M)[[\h]]$-linear map $T: \Omega^1(M)[[\h]]\to
    \cX^1(M)[[\h]]$ is of the form $\pi^\sharp$ for a formal bivector
    field $\pi \in \cX^2(M)[[\h]]$ if and only if
    \[
    \SP{(T(\xi), \xi), (T(\eta), \eta)} = 0, \;\;\;
    \forall \;\; \xi, \eta \in \Omega^1(M).
    \]
\end{lem}

We can characterize formal Poisson structures using \eqref{eq:operat}
as follows.
\begin{lem}
    \label{lem:integ}
    Given a formal bivector field $\pi\in \h\cX^2(M)[[\h]]$, we have
    \begin{equation}
        \label{eq:integ}
        \mathrm{Jac}_\pi(f, g, h) =
        \SP{
          \Cour{(\pi^\sharp(df), df), (\pi^\sharp(dg), dg)},
          (\pi^\sharp(dh), dh)
        },
    \end{equation}
    for all $f, g, h \in \cO(M)[[\h]]$.
\end{lem}
\begin{proof}
    It suffices to verify the lemma for $f, g, h \in \cO(M)$. It is
    clear from \eqref{eq:brack} that $\Lie_{\pi^\sharp(df)} g =
    \{f, g\}_\pi$, and it immediately follows from the definitions
    \eqref{eq:pairing} and \eqref{eq:courantbk} (extended to formal
    power series) that the right-hand side of \eqref{eq:integ} is
    \[
    dh([\pi^\sharp(df), \pi^\sharp(dg)]) +
    d(\Lie_{\pi^\sharp(df)}g)(\pi^\sharp(dh))
    =
    \left(
        \Lie_{\pi^\sharp(df)}\Lie_{\pi^\sharp(dg)}
        -
        \Lie_{\pi^\sharp(dg)}\Lie_{\pi^\sharp(df)}
    \right)h
    +
    \Lie_{\pi^\sharp(dh)}\Lie_{\pi^\sharp(df)}g,
    \]
    which is $\mathrm{Jac}_\pi(f, g, h)$.
\end{proof}

Given any $B = B_0 + \h B_1 + \dots \in \Omega^2(M)[[\h]]$, there is
an associated automorphism of $\cO(M)[[\h]]$-modules given by
\begin{equation}
    \label{eq:la}
    \la_B: \Gamma(M, E)[[\h]] \to \Gamma(M,E)[[\h]], \;\;\; \;
    \la_B(X, \xi) = (X, \xi + i_XB).
\end{equation}
The following properties of $\la_B$ are proven analogously as e.g.
in \cite{Gualtthesis}.
\begin{lem}
    \label{lem:propB}
    The following holds:
    \begin{enumerate}
    \item $\SP{\la_B(X, \xi), \la_B(Y, \eta)} = \SP{(X, \xi), (Y,
          \eta)}$ for all $(X, \xi), (Y, \eta) \in \Gamma(M, E)[[\h]]$.
    \item $\Cour{\la_B(X, \xi),\la_B(Y, \eta)} = \la_B(\Cour{(X,
          \xi), (Y, \eta)})$ for all $(X, \xi), (Y, \eta) \in
        \Gamma(M,E)[[\h]]$ if and only if $dB = 0$.
    \end{enumerate}
\end{lem}
% \begin{proof}
%It suffices to check both statements for all $(X,\alpha), (Y,\beta)
%\in \Gamma(M, E)$. The first is an immediate consequence of the
%skew-symmetry of $B$, $B(X,Y)+B(Y,X)=0$. For the second statement, a
%direct computation shows that
%$$
%\la_B(\Cour{(X,\alpha),(Y,\beta)}) =
%\Cour{\la_B(X,\alpha),\la_B(Y,\beta)} - i_Yi_XdB.
%$$
%\end{proof}

Let us consider the $\cO(M)[[\h]]$-linear map
\[
B^\sharp: \cX^1(M)[[\h]]\to \Omega^1(M)[[\h]], \;\;\;
B^\sharp(X) = i_X B,
\]
associated with $B \in \Omega^2(M)[[\h]]$. For any formal bivector
field $\pi\in \h\cX^1(M)[[\h]]$, the operator
\[
\id + B^\sharp\pi^\sharp :\Omega^1(M)[[\h]]\to \Omega^1(M)[[\h]]
\]
is necessarily invertible; its inverse is given by
\[
\left(\id + B^\sharp\pi^\sharp\right)^{-1}
= \sum_{n=0}^\infty (-1)^n \left(B^\sharp \pi^\sharp\right)^n,
\]
which gives a well-defined formal series in $\h$ since $\pi = 0 \mod
\h$.
\begin{prop}
    \label{prop:tau}
    Let $\pi\in \h\X^2(M)[[\h]]$ and $B\in \Omega^2(M)[[\h]]$. Then:
    \begin{enumerate}
    \item There exists a unique $\ma(B,\pi)\in \h\X^2(M)[[\h]]$ such that
        $\ma(B,\pi)^\sharp=\pi^\sharp\circ (\id +
        B^\sharp\pi^\sharp)^{-1}$.
    \item If $dB=0$ and $[\pi,\pi]_{SN}=0$, then
        $[\ma(B,\pi),\ma(B,\pi)]_{SN}=0$.
    \end{enumerate}
\end{prop}
\begin{proof}
    Let $T= \pi^\sharp\circ (\id + B^\sharp\pi^\sharp)^{-1}$ and set
    $\xi'= (\id + B^\sharp\pi^\sharp)^{-1}(\xi)$, for
    $\xi \in \Omega^1(M)$. Then
    \[
    (T(\xi), \xi)
    = (\pi^\sharp(\xi'), \xi' + i_{\pi^\sharp(\xi')}B)
    = \la_B(\pi^\sharp(\xi'), \xi').
    \]
    Lemma~\ref{lem:propB} implies that
    \[
    \SP{(T(\xi), \xi), (T(\eta), \eta)} =
    \SP{(\pi^\sharp(\xi'), \xi'), (\pi^\sharp(\eta'), \eta')}
    = 0,
    \]
    so the first statement follows from Lemma~\ref{lem:skew}.  If $dB
    = 0$, we can use Lemmas~\ref{lem:integ} and \ref{lem:propB} to
    conclude that $\mathrm{Jac}_\pi(f, g, h) = \mathrm{Jac}_{\ma(B,\pi)}(f,
    g, h)$ for all $f, g, h \in \cO(M)$. The second statement easily
    follows from Prop.~\ref{prop:poissonprop}.
\end{proof}

Since $\la_{B+B'}=\la_B(\la_{B'})$, an immediate consequence of
Prop.~\ref{prop:tau} is that the operation
\[
\pi \mapsto \ma(B,\pi),
\]
for $\pi\in \h\cX^2(M)[[\h]]$ and $B\in \Omega^2(M)[[\h]]$, defines
an action of the abelian group of closed formal 2-forms
$\Omega^2_{cl}(M)[[\h]]$ on formal Poisson structures. We will now
see that this action descends to an action of $H^2(M, \bbC)[[\h]]$ on the
set $\FPois(M)$. For that, it will be convenient to view $\ma(B,
\pi)$ as a solution of a formal differential equation.

Let us consider the space $\h (\cX^\bullet(M)[t])[[\h]]$ of formal
power series in $\h$ with coefficients being polynomials in $t$. An
element $\pi_t \in \h (\cX^2(M)[t])[[\h]]$ defines, as in
\eqref{eq:pisharpext}, a map $\pi_t^\sharp: \Omega^\bullet(M)[[\h]]
\to \h (\cX^\bullet(M)[t])[[\h]]$.
\begin{lem}
    \label{urav}
    Given a formal bivector field $\pi \in \h\cX^2(M)[[\h]]$ and a
    $2$-form $B\in \Om^2(M)[[\h]]$, then $\pi_t = \ma(t B, \pi) \in \h
    (\cX^2(M)[t])[[\h]]$ is the unique solution to the formal
    differential equation
    \begin{equation}
        \label{ono}
        \frac{d}{d t}\, \pi_t =  \pi^{\sh}_t (B),
        \;\;\;
        \pi_t \Big|_{t=0} = \pi,
    \end{equation}
    In particular, $\ma(B, \pi) = \pi_t|_{t=1}$, where $\pi_t$ is the unique
    solution to \eqref{ono}.
\end{lem}
\begin{proof}
    The fact that \eqref{ono} admits a unique solution follows from
    Prop.~\ref{proposition:ODEGeneral} in Appendix~\ref{App-A}.  Note
    that $\pi_t$ is a solution to \eqref{ono} if and only if
    $\pi_t^\sharp$ satisfies
    \[
    \frac{d}{d t} \pi_t^\sharp
    =
    -\pi_t^\sharp\circ B^\sharp\circ \pi_t^\sharp,
    \]
    with initial condition $\pi_t^\sharp\Big|_{t=0} = \pi^\sharp$. A direct
    computation shows that
    \[
    \frac{d}{dt} \pi^\sharp(\id+ tB^\sharp\pi^\sharp)^{-1}
    =
    - \pi^\sharp(\id+ tB^\sharp\pi^\sharp)^{-1}
    B^\sharp \pi^\sharp
    (\id + tB^\sharp\pi^\sharp)^{-1},
    \]
    so the result follows.
\end{proof}

Let $\pi \in \h\cX^2(M)[[\h]]$ be a formal Poisson structure, let
$X(t)\in \h (\cX^1(M)[t])[[\h]]$, and consider the equation
\begin{equation}
    \label{diff-eq}
    \frac{d}{d t}\pi(t) = [\pi(t), X(t)]_{SN},
    \;\;\;
    \pi(0) = \pi
\end{equation}
in $\h (\cX^2(M)[t])[[\h]]$. The following result is proven (in more
generality) in Section~\ref{subsect-example} of Appendix~\ref{App-B}.
\begin{lem}
    \label{claim}
    If $\pi(t) \in \h (\cX^2(M)[t])[[\h]]$ is the solution to
    \eqref{diff-eq}, then $\pi(1)$ satisfies $[\pi(1),\pi(1)]_{SN} =
    0$, and $\pi$ and $\pi(1)$ are equivalent formal Poisson
    structures (i.e., they lie in the same orbit of
    \eqref{action-pi}).
\end{lem}

We can now prove the main result of this section.
\begin{pred}
    \label{dejstv}
    The action of closed 2-forms $\Om^2_{cl}(M)[[\h]]$ on the space of
    formal Poisson structures, $(B,\pi) \mapsto \ma(B, \pi)$, descends to an
    action
    \begin{equation}
        \label{eq:ma-action}
        H^2(M, \bbC)[[\h]]\times
        \FPois(M) \to \FPois(M), \;\;\;
        ([B], [\pi]) \mapsto
        [B] \cdot [\pi]:=[\ma(B,\pi)],
    \end{equation}
    on the set of equivalence classes of formal Poisson structures.
\end{pred}
\begin{proof}
    Let us first show that for cohomologous $2$-forms $B, B' \in
    \Omega_{cl}(M)[[\h]]$, the formal Poisson structures $\ma(B,\pi)$ and
    $\ma(B', \pi)$ are equivalent. It suffices to show that if $B = d\xi$
    is exact, then $\ma(B, \pi)$ is equivalent to $\pi$.

    Suppose that $\xi \in \Omega^1(M)[[\h]]$, and let $B = d\xi$.
    We know that $\ma(t B, \pi)$ is a path of formal Poisson structures
    connecting $\ma(B, \pi)$ and $\pi$, and that it satisfies \eqref{ono}.
    By Corollary~\ref{cor:chain},
    \[
    \ma(tB, \pi)^\sharp(d\xi)
    =
    -[\ma(tB, \pi),\ma(tB, \pi)^\sharp(\xi)]_{SN}.
    \]
    So, in this case, equation \eqref{ono} can be rewritten as
    \[
    \frac{d}{dt}\,\ma(tB, \pi)  = -[\ma(tB, \pi), \ma(tB, \pi)^\sharp(\xi)]_{SN}.
    \]
    Now Lemma~\ref{claim} implies that $\pi$ and $\ma(B, \pi)$ are
    equivalent.

    Next, we should prove that if $B\in \Omega^2_{cl}(M)[[\h]]$ and
    the Poisson structures $\pi$ and $\tpi$ are equivalent, then so
    are the formal Poisson structures $\ma(B, \pi)$ and $\ma(B, \tpi)$. Let us
    assume that
    \begin{equation}
        \label{tpi}
        \tpi = \exp([\cdot, X]_{SN}) \pi,
    \end{equation}
    for $X \in \h\cX^1(M)[[\h]]$. Since $\pi_t = \ma(tB, \pi)$ is the
    solution to \eqref{ono}, $\exp([\cdot, X]_{SN}) \pi_t$ satisfies the
    differential equation
    \begin{equation}
        \label{diff-tpi}
        \frac{d}{d t} \exp([\cdot, X]_{SN}) \pi_t
        =
        \exp([\cdot, X]_{SN})\frac{d}{dt}\pi_t
        =
        (\exp([\cdot, X]_{SN}) \pi_t)^{\sh} (\exp(-\cL_X)\,B).
    \end{equation}
    This equation implies that
    \[
    \ma(\exp(-\cL_X)\,B , \tpi) = \exp([\cdot, X]_{SN}) \ma(B, \pi).
    \]
    On the other hand, the $2$-form $\exp(-\cL_X)B$ is always
    cohomologous to $B$, as a consequence of the Cartan-Weil formula
    for the Lie derivative,
    \[
    \exp(-\cL_X)B
    =
    \sum_{k=0}^{\infty}\frac{1}{k!} (-d i_X)^k B
    =
    B - d
    \left(
        \sum_{k=1}^\infty \frac{1}{k!} i_X (-d i_X)^{k-1} B
    \right).
    \]
    Hence $\ma(\exp(\cL_X) B, \tpi)$ is equivalent to both $\ma(B, \tpi)$ and
    $\ma(B, \pi)$. This concludes the proof of the proposition.
\end{proof}

%
% Kontsevich's classes of Morita equivalent star products
%

\subsection{Kontsevich's classes of Morita equivalent star products}
\label{subsec:class}

As discussed in Section~\ref{subsec:pic}, the groups $\Dif(M)$ and
$\Pic(M)$ naturally act on the space $\Def(M)$ of equivalence classes
of star products on a manifold $M$, and their orbits characterize
Morita equivalent star products. We now describe the corresponding
actions on the moduli space of formal Poisson structure $\FPois(M)$,
making the bijection
\[
\mathcal{K}_*: \FPois(M)\to \Def(M)
\]
equivariant. In other words, we will describe the equivalence relation
in $\FPois(M)$ which is quantized to Morita equivalence under the
Kontsevich map $\cK_*$.

The group $\Dif(M)$ acts on formal Poisson structures in a natural way
by push-forward,
\[
(\varphi, \pi) \mapsto \varphi_*\pi = \h \varphi_*\pi_1 +
\h^2\varphi_*\pi_2 +\ldots,
\]
and it descends to an action of $\Dif(M)$ on $\FPois(M)$. As a result
of \cite[Thm.~1]{thesis}, the map $\cK_*$ respects this action, i.e.,
\[
\cK_*([\varphi_*])=[*_\varphi].
\]
So we focus on the description of the action $\Phi$ of $\Pic(M)$ on
$\Def(M)$ \eqref{eq:actionL} in terms of Kontsevich's classes, which
is given by the next result.
\begin{teo}
    \label{ona}
    Let $L$ be a line bundle over $M$ representing an element in
    $\Pic(M)$, and suppose that $[*]=\cK_*([\pi])$. The action $\Phi:
    \Pic(M)\times \Def(M)\to \Def(M)$ satisfies
    \begin{equation}
        \label{Phi-om}
        \Phi_{L} ([*]) = \cK_*([\ma(B,\pi)])
    \end{equation}
    where $B\in \Omega^2(M)$ is a curvature 2-form of $L$ (i.e., $B$
    represents $ 2 \pi i c_1(L)$, where $c_1(L)$ is the Chern
    class of $L$).
\end{teo}

This result extends the semi-classical description of $\Phi$ in
\cite{B} and, as we will see in Section~\ref{sec:fedosov}, agrees
with \cite{BW1} in the case of symplectic star products.

%\comment{comment on integral gauge transformations being Poisson analogue of algebraic Morita equivalence...}

Before moving to the proof, we need an auxiliary technical statement.
Let us consider an open subset $U \subset M$ for which
$$
B\Big|_{U} = d \theta\,, \qquad \quad \theta\in \Omega^1(U)\,.
$$
Due to Corollary \ref{cor:chain}, the restriction of $\pi_t = \ma(t
B, \pi)$ to $U$ is the unique solution to
\begin{equation}
    \label{eq:pit}
    \frac{d}{dt}\pi_{t} = [\pi_t, v^t]_{SN}, \;\;\; \pi_t\Big|_{t=0}=\pi,
\end{equation}
where $v^t = -\pi_t^\sharp(\theta) \in \h (\cX^1(U)[t])[[\h]]$. We use
the formality $\cK$ to quantize $v^t$ to a series of differential
operators $V^t \in \h(C^1(\cO_U)[t])[[\h]]$
given by
\begin{equation}
    \label{eq:Vt}
    V^t =
    \sum_{n=0}^\infty\frac{1}{n!}\cK_{n+1}(\pi_t,\ldots,\pi_t,v^t).
\end{equation}
Let us define the family of transformations $T^t: \cO(U)[[\h]]\to
\cO(U)[[\h]]$ as the solution to the differential equation (see
Proposition~\ref{proposition:ODEGeneral} in Appendix~\ref{App-A})
\begin{equation}
    \label{eq:localT}
    \frac{d}{dt}T^t(f) = T^t(V^t(f)),
    \;\;\;
    T^t\Big|_{t=0} = \id.
\end{equation}
It is not hard to see that
$$
T^t \in \id \,+\, \h(C^1(\cO_U)[t])[[\h]]
$$
and, in particular,
\begin{equation}
        \label{eq:Tclass}
        T^t\Big|_{\h=0} = \id.
\end{equation}

\begin{lem}
    \label{lem:Bexact}
    Let $*$ and $*_t$ be the Kontsevich star products associated with
    $\pi$ and $\pi_t$, as in \eqref{star-K}. If $T^t$ is the solution of
    the initial value problem (\ref{eq:localT}) then $T^t$ is an
    equivalence between the star-product algebras
    $(\cO(U)[[\h]], *_t)$ and $(\cO(U)[[\h]], *)$, i.e.,
    $T^t = \id \mod \h$ and
    \[
    T^t(f *_t g) = T^t(f) * T^t(g),
    \quad
    \textrm{for all}
    \quad
    f, g \in \cO(U)[[\h]].
    \]
\end{lem}
\begin{proof}
    By definition, $f *_t g = fg + \Pi_t(f, g)$, where
    \begin{equation}
        \label{pi-t}
        \Pi_t
        =
        \sum_{n=1}^{\infty}\frac{1}{n!}
        \cK_n(\pi_t, \pi_t, \ldots, \pi_t).
    \end{equation}
    Using \eqref{eq:pit}, we have
    \begin{equation}
        \label{d-Pi-t}
        \frac{d}{dt}\Pi_t
        =
        \sum_{n=0}^\infty \frac{1}{n!}
        \cK_{n+1}(\pi_t,\ldots,\pi_t,[\pi_t,v^t]_{SN})
        = \pa^{\Hoch}_{*_t} (V^t),
    \end{equation}
    where the last equality follows from the identity
    \eqref{F-1-al-inter} in Appendix~\ref{App-B}; here
    $\pa^{\Hoch}_{*_t}$ is the Hochschild coboundary operator
    corresponding to $*_t$. It follows from \eqref{d-Pi-t} that, for
    all $f, g \in \cO(U)[[\h]]$, we have
    \begin{equation}
        \label{d-star-t}
        \frac{d}{d t} (f *_t  g)
        =
        \pa^{\Hoch}_{*_t} (V^t)(f, g)
        = V^t(f) *_t g + f *_t V^t(g) - V^t(f *_t g).
    \end{equation}
    Combining this equation with \eqref{eq:localT}, we get
    \begin{equation}
        \label{dt}
        \frac{d}{d t}\, T^t(f *_t  g)
        = T^t(V^t(f) *_t g) + T^t(f *_t V^t(g)).
    \end{equation}
    Therefore the cochain $D^t \in (C^2(\cO_U)[t])[[\h]]$,
    \begin{equation}
        \label{C-t-al}
        D^t(f, g) = T^t(f *_t  g) - T^t(f) * T^t(g),
    \end{equation}
    satisfies the following differential equation:
    \[
    \frac{d}{d t} D^t = D^t (V^t \otimes \id + \id \otimes V^t).
    \]
    Taking into account the initial condition $D^t |_{t=0} = 0$ we
    deduce that $D^t$ is identically zero. This completes the proof of
    the lemma.
\end{proof}

~\\

\begin{proof}[ of Theorem~\ref{ona}]
    Let $*$ be a Kontsevich star product on $M$ associated with the
    formal Poisson structure $\pi$. We consider a complex line bundle
    $L\to M$ equipped with a connection $\nabla^L$, and let $B\in
    \Omega^2(M)$ be the curvature of $\nabla^L$. We denote by $*_t$ the
    Kontsevich star product of $\pi_t = \ma(t B,\pi)$. We must show that
    \[
    \Phi_L([*])=[*_1],
    \]
    and for that we will use the local criterium proved in
    Proposition~\ref{prop:crit}.

    Let us consider a cover $\{U_\alpha\}$ of $M$ by contractible open
    subsets with contractible intersections $U_\alpha\cap U_\beta$. We
    fix a set of local trivializations of $L$, defining transition
    functions $g_{\alpha\beta}\in \cO(U_\alpha\cap U_\beta)$.  Then
    $\nabla^L$ is described by a collection of connection 1-forms
    $\theta_\alpha\in \Omega^1(U_\alpha)$, satisfying
    \begin{equation}
        \label{eq:connection1}
        \theta_\beta - \theta_\alpha
        =
        g_{\alpha\beta}^{-1} d g_{\alpha\beta},
        \qquad
        d\theta_\alpha = B|_{U_\alpha}.
    \end{equation}
    By Lemma~\ref{lem:Bexact}, we know that over each $U_\alpha$ there
    is an equivalence
    \[
    T^t_\alpha (f *_t g) = T^t_\alpha(f) * T^t_\alpha(g)
    \]
    for $f, g \in \cO(U_\alpha)[[\h]]$.  By \eqref{eq:connection1},
    the difference between the local vector fields
    $v_\alpha^t = -\pi^\sharp_t(\theta_\alpha)$ and
    $v_\beta^t = -\pi^\sharp_t(\theta_\beta)$ is hamiltonian:
    \[
    v_\alpha^t - v_\beta^t
    = \pi^\sharp_t(d\log(g_{\alpha\beta}))
    = -[\pi_t, \log(g_{\alpha\beta})]_{SN}.
    \]
    Note that $\log(g_{\alpha\beta})$ is well-defined since
    $U_\alpha\cap U_\beta$ is contractible. It follows that the
    corresponding operators \eqref{eq:Vt} satisfy
    \[
    V_\alpha^t - V_\beta^t
    =
    - \sum_{n=0}^\infty \frac{1}{n!}
    \cK_{n+1}(\pi_t,\dots,\pi_t,[\pi_t, \log(g_{\alpha\beta})]_{SN}).
    \]
    By \eqref{F-1-al-inter}, we can write
    \begin{equation}
        \label{V-al-beta}
        V^t_{\al} - V^t_{\beta}
        = - \pa^{\Hoch}_{*_t} h_{\al\beta}(t)
        = \ad_{*_t}(h_{\alpha\beta}(t)),
    \end{equation}
    where $\pa^{\Hoch}_{*_t}$ denotes the Hochschild coboundary
    operator (\ref{pa-Hoch}) of
    $*_t$, $\ad_{*_t}(f) (g) := f *_t g - g *_t f$, and
    \begin{equation}
        \label{h-al-beta}
        h_{\al\beta}(t)
        =
        \sum_{n=0}^{\infty} \frac{1}{n!}
        \cK_{n+1}(\pi_t, \dots, \pi_t, \log(g_{\alpha\beta})).
    \end{equation}
    We will be interested in the operators
    \[
    T_\alpha = T^t_\alpha|_{t=1},
    \]
    which are local equivalences between $*_1$ and $*$. According to
    Proposition~\ref{prop:crit}, to prove the theorem it suffices to
    define deformed transition functions $G_{\alpha\beta} =
    g_{\alpha\beta} \mod \h$, satisfying the cocycle conditions
    \eqref{eq:defcocyc} as well as
    \begin{equation}
        \label{eq:compat2}
        T_\alpha(T_\beta)^{-1}(f)
        =
        G_{\alpha\beta} * f * G_{\alpha\beta}^{-1}
        \quad
        \textrm{for all}
        \quad
        f \in \cO(U_\alpha \cap U_\beta)[[\h]].
    \end{equation}
    We will do that by constructing a family of functions
    $G_{\alpha\beta}(t)$ that will satisfy the desired properties for
    $t=1$.
    \begin{claim}
        \label{claim:TT}
        The operator $T_\alpha^t(T_\beta^t)^{-1}$ is a self-equivalence of $*$
        satisfying
        \[
        \frac{d}{dt}T_\alpha^t(T_\beta^t)^{-1}
        =
        \ad_*(T^t_\alpha h_{\alpha\beta}) T_\alpha^t(T_\beta^t)^{-1},
        \;\;\;\;
        T_\alpha^t(T_\beta^t)^{-1} |_{t=0} = \id.
        \]
    \end{claim}
    \begin{subproof}
        It is clear that $T_\alpha^t(T_\beta^t)^{-1} = \id \mod \h$,
        that it is an automorphism of
        $(\cO(U_{\al}\cap U_{\beta})[[\h]], *)$, and that it equals
        $\id$ when $t=0$.  Since $T_\alpha^t$ satisfies
        \eqref{eq:localT}, differentiating the identity
        $T^t_\alpha(T^t_\alpha)^{-1} = \id$ implies that
        \[
        \frac{d}{dt}(T^t_\alpha)^{-1}
        = -V_\alpha^t(T^t_\alpha)^{-1}.
        \]
        Using \eqref{V-al-beta}, we obtain
        \[
        \frac{d}{dt}T_\alpha^t(T_\beta^t)^{-1}
        =
        T_\alpha^t V_\alpha^t (T^t_\beta)^{-1}
        -
        T^t_\alpha V_\beta^t(T^t_\beta)^{-1}
        =
        T^t_\alpha (\ad_{*_t}(h_{\alpha\beta}))(T^t_\beta)^{-1}.
        \]
        Since $T_\alpha^t$ is an algebra homomorphism from
        $(\cO(U_{\al})[[\h]], *_t)$ to $(\cO(U_{\al})[[\h]], *)$, we
        have the identity $T_\alpha^t \ad_{*_t}(a) = \ad_*(T^t_\alpha
        a) T^t_\alpha$, from which the claim follows.
    \end{subproof}

    Using Prop.~\ref{proposition:ODEExponential} in Appendix
    \ref{App-A}, we can define the family of functions
    $G_{\alpha\beta}(t)$ as the unique solution to the differential
    equation
    \begin{equation}
        \label{eq:Galbe}
        \frac{d}{dt}G_{\alpha\beta}(t)= (T_\alpha^t
        h_{\alpha\beta}(t))*G_{\alpha\beta}(t),
        \;\;\;\;
        G_{\alpha\beta}(0)=1.
    \end{equation}
    The solution of \eqref{eq:Galbe} is $*$-invertible, and its
    $*$-inverse satisfies
    \begin{equation}
        \label{eq:Galbe2}
        \frac{d}{dt}(G_{\alpha\beta}(t))^{-1}=
        -(G_{\alpha\beta}(t))^{-1}*(T_\alpha^t h_{\alpha\beta}(t)).
    \end{equation}
    From \eqref{h-al-beta}, we see that $h_{\alpha\beta}(t)=
    c_{\alpha\beta}\mod \h$, where
    $c_{\alpha\beta}=\log(g_{\alpha\beta})$. When $\h=0$, the
    differential equation \eqref{eq:Galbe} becomes
    \[
    \frac{d}{dt}G_{\alpha\beta}^{(0)}(t) = c_{\alpha\beta}
    G_{\alpha\beta}^{(0)}(t),
    \]
where
\[
G_{\alpha\beta}^{(0)}(t) = G_{\al\beta}(t) \Big|_{\h=0}\,.
\]
Hence $G_{\alpha\beta}^{0}(t) = e^{t c_{\alpha\beta}}$ and, in
particular
    \begin{equation}
        \label{eq:classG}
        G_{\alpha\beta}(1) = g_{\alpha\beta} \mod \h.
    \end{equation}
    \begin{claim}
        If $G_{\alpha\beta}(t)$ is a solution to \eqref{eq:Galbe} then
        \begin{equation}
        \label{eq:compat-t}
        T^t_\alpha (T^t_\beta)^{-1}(f)
        =
        G_{\alpha\beta}(t) * f * G_{\alpha\beta}^{-1}(t)
        \quad
        \textrm{for all}
        \quad
        f \in \cO(U_\alpha \cap U_\beta)[[\h]].
    \end{equation}
    \end{claim}
    \begin{subproof}
        Let $\mathrm{Ad}_*(G_{\alpha\beta}(t))$ be the conjugation
        operator with respect to $*$,
        \[
        \mathrm{Ad}_*(G_{\alpha\beta}(t))(f)
        =
        G_{\alpha\beta}(t)* f * G_{\alpha\beta}(t)^{-1},
        \]
        for $f \in \cO(U_\alpha \cap U_\beta)[[\h]]$.  Then from
        \eqref{eq:Galbe} and \eqref{eq:Galbe2} we see that
        \begin{align*}
            \frac{d}{dt} \mathrm{Ad}_*(G_{\alpha\beta}(t))(a)
            &=
            (T_\alpha^t h_{\alpha\beta}(t))
            * G_{\alpha\beta}(t) * f * G_{\alpha\beta}(t)^{-1}
            -
            G_{\alpha\beta}(t) * f * (G_{\alpha\beta}(t))^{-1}
            * (T_\alpha^t h_{\alpha\beta}(t))\\
            &=
            \ad_*(T_\alpha^t h_{\alpha\beta}(t))
            \mathrm{Ad}_{*}(G_{\alpha\beta}(t))(f)
        \end{align*}
        Since $\mathrm{Ad}_*(G_{\alpha\beta}(t))|_{t=0} = \id$, we
        conclude from Claim~\ref{claim:TT} that
        $T^t_\alpha(T^t_\beta)^{-1} = \mathrm{Ad}_*(G_{\alpha\beta}(t))$.
    \end{subproof}

    The next claim implies that the functions
$$
G_{\alpha\beta}: = G_{\al\beta}(1)
$$
    satisfy the
    desired cocycle conditions.
    \begin{claim}
        The following identities hold: $G_{\alpha\alpha} = 1$,
        $G_{\alpha\beta}*G_{\beta\alpha} = 1$, $G_{\alpha\beta} *
        G_{\beta\gamma} * G_{\gamma\alpha} = 1$.
    \end{claim}
    \begin{subproof}
        Since $h_{\alpha\alpha}(t) = 0$, it is clear from
        \eqref{eq:Galbe} that $G_{\alpha\alpha} = 1$. For the second
        identity, we have
        \begin{eqnarray*}
            \frac{d}{dt}(G_{\alpha\beta}(t) * G_{\beta\alpha}(t))
            &=&
            (T_\alpha^t h_{\alpha\beta}(t))
            * G_{\alpha\beta}(t) * G_{\beta\alpha}(t) + \\
            &&
            G_{\alpha\beta}(t) * (T_\beta^t h_{\beta\alpha}(t))
            * G_{\beta\alpha}(t)\\
            &=&
            (T_\alpha^t h_{\alpha\beta}(t))
            * G_{\alpha\beta}(t) * G_{\beta\alpha}(t) + \\
            &&
            G_{\alpha\beta}(t) * G_{\beta\alpha}(t)
            *
            (T_\alpha^th_{\beta\alpha}(t))
            *
            G_{\beta\alpha}^{-1}(t) * G_{\beta\alpha}(t)\\
            &=&
            \ad_*(T_{\alpha}^t h_{\alpha\beta}(t))
            (G_{\alpha\beta}(t)*G_{\beta\alpha}(t)),
        \end{eqnarray*}
        where we have used  \eqref{eq:Galbe}, \eqref{eq:compat-t} and
        $h_{\alpha\beta}(t) = -h_{\beta\alpha}(t)$. Note that
        $G_{\alpha\beta}(t) * G_{\beta\alpha}(t) = 1$ is the unique
        solution with initial condition $G_{\alpha\beta}(0) *
        G_{\beta\alpha}(0) = 1$.

        We proceed similarly to prove that $G_{\alpha\beta} *
        G_{\beta\gamma} * G_{\gamma\alpha} = 1$. Using
         \eqref{eq:Galbe} and \eqref{eq:compat-t}, we obtain:
        \begin{eqnarray}
            \frac{d}{dt}
            G_{\alpha\beta}(t) * G_{\beta\gamma}(t) *
            G_{\gamma\alpha}(t)
            &=&
            (T_\alpha^t h_{\alpha\beta}(t))
            * G_{\alpha\beta}(t)
            * G_{\beta\gamma}(t)
            * G_{\gamma\alpha}(t) \label{eq:cocyc}\\
            &&
            +
            G_{\alpha\beta}(t)
            * (T_\beta^t h_{\beta\gamma}(t))
            * G_{\beta\gamma}(t)
            * G_{\gamma\alpha}(t)\nonumber\\
            &&
            +
            G_{\alpha\beta}(t)
            * G_{\beta\gamma}(t)
            * (T_\gamma^t h_{\gamma\alpha}(t))
            * G_{\gamma\alpha}(t)\nonumber\\
            &=&
            (T_\alpha^t h_{\alpha\beta}(t))
            * G_{\alpha\beta}(t)
            * G_{\beta\gamma}(t)
            * G_{\gamma\alpha}(t)\nonumber\\
            &&
            +
            G_{\alpha\beta}(t)
            * G_{\alpha\beta}^{-1}(t)
            * (T_\alpha^t h_{\beta\gamma}(t))
            * G_{\alpha\beta}(t)
            * G_{\beta\gamma}(t)
            * G_{\gamma\alpha}(t)\nonumber\\
            &&
            +
            G_{\alpha\beta}(t)
            * G_{\beta\gamma}(t)
            * G_{\beta\gamma}^{-1}(t)
            * (T_\beta^t h_{\gamma\alpha}(t))
            * G_{\beta\gamma}(t)
            * G_{\gamma\alpha}(t)\nonumber\\
            &=&
            (T_\alpha^t h_{\alpha\beta}(t))
            * G_{\alpha\beta}(t)
            * G_{\beta\gamma}(t)
            * G_{\gamma\alpha}(t)\nonumber\\
            &&
            +
            (T_\alpha^t h_{\beta\gamma}(t))
            * G_{\alpha\beta}(t)
            * G_{\beta\gamma}(t)
            * G_{\gamma\alpha}(t)\nonumber\\
            &&
            +
            (T_\alpha^t h_{\gamma\alpha}(t))
            * G_{\alpha\beta}(t)
            * G_{\beta\gamma}(t)
            * G_{\gamma\alpha}(t)\nonumber\\
            &=&
            \left(
                T_\alpha^t
                \left(
                    h_{\alpha\beta}(t)
                    + h_{\beta\gamma}(t)
                    + h_{\gamma\alpha}(t)
                \right)
            \right)
            *G_{\alpha\beta}(t)*G_{\beta\gamma}(t)*G_{\gamma\alpha}(t)
            \nonumber
        \end{eqnarray}
        Using Proposition \ref{cK-constant}, we see that
        \[
        h_{\alpha\beta}(t)
        + h_{\beta\gamma}(t)
        + h_{\gamma\alpha}(t)
        =
        \sum_{n=0}^\infty\frac{1}{n!}
        \mathcal{K}_{n+1}\left(
            \pi_t, \ldots, \pi_t,
            c_{\alpha\beta} + c_{\beta\gamma} + c_{\gamma\alpha}
        \right)
        = c_{\alpha\beta} + c_{\beta\gamma} + c_{\gamma\alpha},
        \]
        where $c_{\alpha\beta} = \log(g_{\alpha\beta})$. Since
        $g_{\alpha\beta}g_{\beta\gamma}g_{\gamma\alpha} = 1$, we have
        that $c_{\alpha\beta} + c_{\beta\gamma} + c_{\gamma\alpha}=
        2\pi i n_{\alpha\beta\gamma}$, for $n_{\alpha\beta\gamma} \in
        \mathbb{Z}$ (note that $n_{\alpha\beta\gamma}$ is the \v{C}ech
        cocycle in $\check{H}^2(M,\mathbb{Z})$ representing the line
        bundle $L$). Hence the unique solution of \eqref{eq:cocyc}
        with initial condition $1$ is
        \begin{equation}
            \label{eq:TheFinalTransitionFunctions}
            G_{\alpha\beta}(t)
            * G_{\beta\gamma}(t)
            * G_{\gamma\alpha}(t)
            =
            e^{2\pi i n_{\alpha\beta\gamma} t}.
        \end{equation}
        In particular, for $t=1$ we have
        $G_{\alpha\beta} * G_{\beta\gamma} * G_{\gamma\alpha} = 1$.
    \end{subproof}
    This finishes the proof of Theorem \ref{ona}.
\end{proof}

%
% Fedosov's classes versus Kontsevich's classes
%%%%%%%%%%%%%%%%%%%%%%%%%%%%%%%%%%%%%%%%%%%%%%%%%%%%%%%%%%%%%%%%%%%%%%%%

\section{Fedosov's classes versus Kontsevich's classes}
\label{sec:fedosov}

In this section we focus on formal Poisson structures $\pi=\h\pi_1+
\h^2\pi_2+\ldots$ on a manifold $M$ for which the leading term
$\pi_1\in \cX^2(M)$ is a nondegenerate bivector field, i.e., we
assume that the associated vector-bundle map $\pi_1^\sharp:T^*M\to
TM$, $\pi^\sharp(\xi)=i_\xi\pi$, is an isomorphism. In this case,
$\pi_1$ corresponds to a symplectic form $\omega_{-1}\in
\Omega^2(M)$, uniquely defined by
$$
i_{\pi_1^\sharp(\xi)}\omega_{-1}=\xi,\;\;\;\; \forall ~ \xi \in
\Omega^1(M),
$$
and the Kontsevich star product (\ref{star-K}) defines a deformation
quantization of the  symplectic manifold $(M, \omega_{-1})$.

The fact that $\pi_1$ is nondegenerate implies, more generally, that
the $\cO(M)[[\h]]$-linear map $\pi^\sharp: \Omega^1(M)[[\h]]\to
\h\X^1(M)[[h]]$ (see \eqref{eq:pisharp}) is an isomorphism, and its
inverse uniquely defines a formal series of 2-forms
\begin{equation}
    \label{F-class}
    \omega
    =
    \frac{1}{\h} \omega_{-1} + \omega_0 + \h \omega_1 +
    \h^2 \omega_2 + \cdots.
\end{equation}
The integrability condition $[\pi,\pi]_{SN}=0$ is equivalent to
$d\omega_j=0$, $\forall ~ j=-1,0,1,\ldots$. This gives us a 1-1
correspondence between formal Poisson structures $\pi=\h\pi_1
+\ldots $ for which $\pi_1$ is nondegenerate and series of closed
2-forms $\omega$ as in \eqref{F-class} for which $\omega_{-1}$ is
symplectic (cf. \cite[Sec.~3]{Gutt}). Furthermore, under this
correspondence, the action
 of the group \eqref{group-pvect} boils down to the action
\[
\omega \mapsto \omega + d\, \ve,
\]
where $\ve= \ve_0 + \h \ve_1+\ldots \in \Omega^1(M)[[\h]]$ is an
arbitrary formal power series of $1$-forms. In particular, for a
fixed nondegenerate Poisson structure $\pi_1$, the set of
equivalence classes $[\pi]\in \FPois(M)$ such that
$\pi=\h\pi_1+\ldots$ is in bijective correspondence with
$H^2(M, \bbC)[[\h]]$ (cf. \cite[Prop.~13]{Gutt}).

On the other hand, Fedosov's construction \cite{F} leads to a
parametrization of the set of equivalence classes of star products
on a given symplectic manifold $(M,\omega)$ by elements
$\sum_{j=0}^\infty \h^j[\omega_j] \in H^2(M, \bbC)[[\h]]$ (see e.g.
\cite{BCG,Deligne,NT}); the elements $\frac{1}{\h}[\omega] +
\sum_{j=0}^\infty \h^j[\omega_j]$ are known as \textit{Fedosov
classes}.

\begin{teo}
    \label{F-versus-K} Let $\pi=\h\pi_1+\h^2\pi_2+\ldots$ be a
    formal Poisson structure such that $\pi_1$ is a
    nondegenerate bivector field, and let $\omega$ be the associated
    formal series of closed 2-forms as in \eqref{F-class}. Then the Fedosov class of the Kontsevich star
    product (\ref{star-K}) of $\pi$ is represented by $\omega$.
\end{teo}
%\begin{remark}
%    \label{remark:FedosovKontsevichConjectured}
%    The statement of this theorem was formulated in \cite{CR} as a
%    conjecture. (See Conjecture~4 in \cite{CR}.)  Combining this
%    theorem with Deligne's paper \cite{Deligne} in which Fedosov's
%    approach is compared to that of De Wilde and Lecomte
%    \cite{DeWilde} we close partially the project mentioned in item 1)
%    of Section~0.2 in \cite{K}.
%\end{remark} ----- This is already at the introduction
%\begin{remark}
 %   \label{remark:ReproducesSymplecticMoritaStuff}

    It directly follows from this result that the description of Morita
    equivalent star products in terms of Kontsevich's classes of
    Theorem~\ref{ona} reduces, in the symplectic case, to the description of \cite[Theorem~3.1]{BW1}
    in terms of Fedosov's classes.

%\end{remark}

The plan of the proof is depicted on the following diagram:
\begin{equation}
    \label{plan}
    \begin{array}{ccccc}
        \begin{array}{c}
            \mbox{Kontsevich's} \\
            \mbox{star product}~ *_K
        \end{array}
        & ~\sim~  &
        \begin{array}{c}
            \mbox{modified Fedosov's} \\
            \mbox{star product}~ \mstar
        \end{array}
        & ~ = ~ &
        \begin{array}{c}
            \mbox{original Fedosov's} \\
            \mbox{star product}~ *_F
        \end{array}
    \end{array}
\end{equation}

This diagram will be turned into a proof of Theorem~\ref{F-versus-K}
in this section. In Subsection~\ref{modified} we construct the
modified Fedosov's star product $\mstar$. The difference between the
constructions of $\mstar$ and the original Fedosov's star product
$*_F$ is that for $\mstar$ we use the fiberwise multiplication
(\ref{t-bullet}), which involves the whole series $\pi$, whereas to
construct $*_F$ we use only the first term $\h \pi_1$. In
Subsection~\ref{subsec:DEW} we introduce a version of the
Emmrich-Weinstein differential. Using this differential in
Subsection~\ref{K-equiv-m}, we show that Kontsevich's star product
$*_K$ (\ref{star-K}) is equivalent to $\mstar$. Finally, in
Subsection~\ref{m-equiv-F}, we prove that $\mstar$ coincides with
the original Fedosov's star product $*_F$ whose equivalence class is
represented by $\omega$ (\ref{F-class}). In following these steps,
it will be important to recall (see Subsection
\ref{subsec:star-products}) that the bijective correspondence
between equivalence classes of Maurer-Cartan elements induced from
the direct $L_{\infty}$ quasi-isomorphism $\cK$ (\ref{cK}) and the
sequence of quasi-isomorphisms in (\ref{upper}) coincide and are
independent of the choice of the connection/Fedosov's differential.

%\comment{
%We should remark that in this paper we use two different constructions
%of Kontsevich's star product $*_K$.  In the first construction we use
%the direct $L_{\infty}$ quasi-isomorphism $\cK$ (\ref{cK}) and obtain
%desired Kontsevich's star product $*_K$ using equation (\ref{star-K}).
%In the second construction we use the sequence of quasi-isomorphisms
%in (\ref{upper}) which also gives us a bijection between moduli spaces
%of Maurer-Cartan elements of the DGLAs $\mathcal{X}^{\bul+1}M$ and
%$\sCbu(\OM)$. Both constructions involve the choice of the
%connection/Fedosov's differential.

%According to Theorem~\ref{nezalezh} the bijection between the set of
%the equivalence classes of star products and the moduli space of
%Maurer-Cartan elements of the DGLA $\mathcal{X}^{\bul+1}M$ depends
%neither on the choice of the construction nor on the choice of the
%connection/Fedosov's differential.}
%
% Modified Fedosov's construction
%%%%%%%%%%%%%%%%%%%%%%%%%%%%%%%%%%%%%%%%%%%%%%%%%%%%%%%%%%%%%%%%%%%%%%%

\subsection{The modified Fedosov construction}
\label{modified}

In this subsection we consider a modification of Fedosov's
construction based on the following associative product on the sheaf
$\SM[[\h]]$ of $\OM[[\h]]$-modules:
\begin{equation}
    \label{t-bullet}
    a_1 \mdiamond a_2
    =
    a_1
    \exp \left(
        \pi^{ij}(x)
        \frac{\overleftarrow{\pa}}{\pa y^i} \,
        \frac{\overrightarrow{\pa}}{\pa y^i}
    \right)
    a_2,
\end{equation}
where $\pi=\h\pi_1+\ldots$ is a formal Poisson structure (in
particular, the coefficients $\pi^{ij}$ are series in $\h$) and
 $\pi_1$ is nondegenerate.  Recall that
the sheaf $\SM[[\h]]$ is equipped with the descending filtration of
Remark~\ref{remark-filtr}, and one can check that the product
(\ref{t-bullet}) is compatible with this filtration.  We may view
the product (\ref{t-bullet}) as a quantization of the fiberwise
Poisson structure
\begin{equation}
    \label{pi-y}
    \pifib = \pi^{ij}(x) \frac{\pa}{\pa y^i} \wedge
    \frac{\pa}{\pa y^j}.
\end{equation}
Since $\pi_1$ is a non-degenerate Poisson bivector field, there
exists a torsion-free connection form
\[
d x^j (\G_{\h})^i_{jk}(x)
=
d x^j \G^i_{jk}(x)
+ d x^j \h (\G_1)^i_{jk}(x)
+ d x^j \h^2 (\G_2)^i_{jk}(x)
+ \cdots,
\]
satisfying the compatibility condition
\begin{equation}
    \label{G-h-pi}
    \n \{a_1, a_2 \}_{\pifib}
    = \{\n a_1 , a_2 \}_{\pifib} + \{a_1 , \n a_2 \}_{\pifib},
\end{equation}
where
\begin{equation}
    \label{nabla}
    \n
    = d x^i \frac{\pa}{\pa x^i}
    - d x^i (\G_{\h})^k_{ij}(x) y^j \frac{\pa}{\pa y^k},
\end{equation}
and, for $a_1, a_2$ local sections of $\SM[[\h]]$,
\begin{equation}
    \label{brack-pi}
    \{a_1 , a_2 \}_{\pifib}
    = \pi^{ij}(x) \frac{\pa a_1}{\pa y^i} \frac{\pa a_2}{\pa y^j}
\end{equation}
is the fiberwise Poisson bracket on $\SM[[\h]]$ coming from the
fiberwise Poisson structure $\pifib$. Note that (\ref{G-h-pi})
implies that the connection $\n$ is also compatible with the product
(\ref{t-bullet}), i.e.,
\begin{equation}
    \label{n-tbul}
    \n (a_1 \mdiamond a_2)
    = (\n a_1) \mdiamond a_2 + a_1 \mdiamond (\n a_2).
\end{equation}

In general, the connection $\n$ is not flat. In fact,
\[
\n^2 = \frac{1}{\h}[R, \,]_{\mdiamond},
\]
where
\begin{equation}
    \label{R-h}
    R
    =
    \frac{1}{2}d x^i\, d x^j R_{ij\,\,kl}(x) y^k y^l,
\end{equation}
$R_{ij\,\,kl}(x) =
    \h\omega_{km}(x)R_{ij\,\,\,l}^{\,\,\,\,\, m}(x)$,
and $R_{ij\,\,\,l}^{\,\,\,\,\, m}(x)$ are the components (possibly
depending on $\h$) of the curvature tensor.  Even though $R$ is not
vanishing in general, we can modify $\n$ to the following flat
connection:
\begin{equation}
    \label{D-F-h}
    D^F_{\h} = \n -\de + \frac{1}{\h}[r, \cdot]_{\mdiamond},
\end{equation}
where $r$ is an element of $\Om^1(M, \cF^3 \SM[[\h]])$ obtained by
iterating the equation
\begin{equation}
    \label{r-iter}
    r =
    \de^{-1} R + \de^{-1} \left(
        \n r + \frac{1}{2\h}
        [r,r]_{\mdiamond}
    \right).
\end{equation}
It can be shown that by iterating (\ref{r-iter}) we get an element
$r\in \Om^1(M, \cF^3 \SM[[\h]])$ satisfying
\begin{equation}
    \label{r-eq}
    R + \n r -\de r + \frac{1}{2\h} [r,r]_{\mdiamond} = 0,
\end{equation}
and this equation implies that $(D^F_{\h})^2 =0$.  Notice that the
derivation $\de$ (\ref{delta}) of the algebra $\Omb(M, \SM[[\h]])$
is inner. More precisely,
\begin{equation}
    \label{de-inner-tbul}
    \de = [d x^i \omega_{ij}(x,\h) y^j, \cdot]_{\mdiamond}.
\end{equation}
As a consequence, the differential (\ref{D-F-h}) can be rewritten as
\begin{equation}
    \label{D-F-h1}
    D^F_{\h} = \n + \frac{1}{\h}[b, \cdot ]_{\mdiamond},
\end{equation}
where
\begin{equation}
    \label{b}
    b = r - \h\, d x^i \omega_{ij}(x, \h) y^j.
\end{equation}
It follows from (\ref{r-eq}) that the element $b$ satisfies
\begin{equation}
    \label{Fed-class}
    \frac{1}{\h} R + \frac{1}{\h} \n b +
    \frac{1}{2\h^2} [b, b]_{\mdiamond} = - \omega.
\end{equation}

As already used in Section~\ref{sec:prelim}, we have the obvious map
\begin{equation}
    \label{sigma-h}
\sigma:
    \G(M, \SM)[[\h]] \cap \ker D^F_{\h}
    \longrightarrow
    \cO(M)[[\h]],\;\;\; \si(c)= c \Big|_{y^i = 0},
\end{equation}
from the $\bbC[[\h]]$-module of $D^F_{\h}$-flat sections of
$\SM[[\h]]$ to the $\bbC[[\h]]$-module $\cO(M)[[\h]]$.  The map
(\ref{sigma-h}) turns out to be an isomorphism, and the inverse map
\begin{equation}
    \label{lift}
    \mtau :
    \cO(M)[[\h]]
    \longrightarrow
    \G(M, \SM)[[\h]] \cap \ker D^F_{\h}.
\end{equation}
is defined by the following iterative procedure:
\begin{equation}
    \label{iter-tau-h}
    \mtau(f)
    =
    f + \de^{-1}(\n \mtau (f) + [r, \mtau(f)]_{\mdiamond}),
\end{equation}
where $f\in \cO(M)[[\h]]$ and the iteration in (\ref{iter-tau-h}) goes
with respect to the filtration (\ref{filtr-SM-h}).

Using the isomorphism (\ref{lift}), we obtain the \emp{modified
Fedosov
  star product} $\mstar$:
\begin{equation}
    \label{star-m}
    f_1 \mstar f_2
    = \sigma (\mtau(f_1) \mdiamond \mtau(f_2)),
\end{equation}
where $f_1, f_2 \in \cO(M)[[\h]]$.  In Subsection~\ref{m-equiv-F} we
will show that $\mstar$ coincides with the original Fedosov's star
product whose equivalence class is represented by $\omega$ (\ref{F-class}).

%%%%%%%%%%%%%%%%%%%%%%%%%%%%%%%%%%%%%%%%%%%%%%%%%%%%%%%%%%%%%%%%%%%%%%%%%
% The Emmrich-Weinstein differential
%

\subsection{The Emmrich-Weinstein differential}
\label{subsec:DEW}
The compatibility between the ``deformed'' connection $\n$
(\ref{nabla}) and the fiberwise Poisson bracket $\{\cdot,
\cdot\}_{\pifib}$ (\ref{brack-pi}) allows us to construct the
Emmrich-Weinstein differential \cite{EW}
\begin{equation}
    \label{D-EW}
    \DEW = \n - \de + \frac{1}{\h}\{r^{cl}, \cdot \}_{\pifib},
\end{equation}
where $r^{cl}$ is the element of $\Om^1(M, \cF^3 \SM[[\h]])$
obtained by iterating the equation
\begin{equation}
    \label{r-cl-iter}
    r^{cl}
    =
    \de^{-1} R + \de^{-1} \left(
        \n r^{cl} + \frac{1}{2\h} \{r^{cl},r^{cl}\}_{\pifib}
    \right),
\end{equation}
where $R$ is defined in \eqref{R-h}.  By iterating
(\ref{r-cl-iter}), we obtain an element $r^{cl}\in \Om^1(M, \cF^3
\SM[[\h]])$ satisfying the equation
\begin{equation}
    \label{r-cl-eq}
    R + \n r^{cl} -\de r^{cl} + \frac{1}{2\h}
    \{r^{cl},r^{cl} \}_{\pifib}
    = 0,
\end{equation}
and this equation implies that $(\DEW)^2 =0$.  Similarly to equation
(\ref{de-inner-tbul}), we have
\[
\de = \{d x^i \omega_{ij}(x, \h) y^j, \cdot \}_{\pifib}.
\]
Therefore we can rewrite (\ref{D-EW}) as
\begin{equation}
    \label{D-EW-with-b-cl}
    \DEW
    =
    \n + \frac{1}{\h}\{b^{cl}, \cdot \}_{\pifib},
\end{equation}
where
\begin{equation}
    \label{b-cl-r-cl}
    b^{cl} = - \h d x^i \omega_{ij}(x, \h) y^j + r^{cl}.
\end{equation}
Equation (\ref{r-cl-eq}) implies that
\begin{equation}
    \label{b-cl-eq}
    \frac{1}{\h} R + \frac{1}{\h} \n b^{cl}
    +\frac{1}{2\h^2} \{b^{cl}, b^{cl}\}_{\pifib} = - \omega\,.
\end{equation}

We remark that the differential $\DEW$ (\ref{D-EW}) differs from the
original one introduced by Emmrich and Weinstein in \cite[Sect.
8]{EW}. The fiberwise Poisson bracket considered in \cite{EW} does
not involve $\h$, whereas $\pifib$ (\ref{pi-y}) is a series in
$\h$\,. So, even though the recursion for $r^{cl}$ looks
``classical'',  the element $r^{cl}$ does contain higher orders of
$\hbar$\,. In particular, $r^{cl}$ is not just obtained by setting
$\hbar = 0$ in the element $r$ (\ref{r-iter}). However we have the
following:
\begin{pred}
\label{prop:r-r-cl} Let $r$ and $r^{cl}$ be the elements of
$\Om^1(M, \cF^3 \SM[[\h]])$ defined by iterating equations
(\ref{r-iter}) and (\ref{r-cl-iter}), respectively. Then
    \begin{equation}
        \label{r-r-cl}
        r - r^{cl} = 0 \mod \h^2.
    \end{equation}
\end{pred}
\begin{proof}
Let $r_k$ (resp. $r^{cl}_k$) be the approximation of $r$ (resp.
    $r^{cl}$) which we obtain on the $k$-th step of the iterative
    procedure (\ref{r-iter}) (resp. (\ref{r-cl-iter})). Namely, $r_0 =
    r^{cl}_0 = \de^{-1} R$ and $r_k$ is related to $r_{k-1}$ via the
    equation
    \begin{equation}
        \label{r-iter-k}
        r_k
        =
        \de^{-1} R + \de^{-1}
        \left(
            \n r_{k-1} + \frac{1}{2\h} [r_{k-1}, r_{k-1}]_{\mdiamond}
        \right),
    \end{equation}
    while $r^{cl}_k$ is related to $r^{cl}_{k-1}$ via the equation
    \begin{equation}
        \label{r-cl-iter-k}
        r^{cl}_k
        =
        \de^{-1} R + \de^{-1}
        \left(
            \n
            r^{cl}_{k-1} + \frac{1}{2 \h}
            \{r^{cl}_{k-1}, r^{cl}_{k-1}\}_{\pifib}
        \right).
    \end{equation}
    Let us show by induction that
    \begin{equation}
        \label{r-r-cl-k}
        r_k - r^{cl}_k = 0 \mod \h^2
    \end{equation}
    for all $k$.  For $k=0$, this is obvious. To
    perform the inductive step, we observe that
    \begin{equation}
        \label{skobka-tbul}
        [a_1 , a_2]_{\mdiamond} - \{ a_1, a_2\}_{\pifib} = 0
        \mod \h^3
    \end{equation}
    for all $a_1, a_2 \in \G(M, \SM)[[\h]]$.  This observation
    and the inductive hypothesis imply that
    \begin{align*}
        r_k- r^{cl}_k
        &=
        \de^{-1} \left(
            \n (r_{k-1} - r^{cl}_{k-1})
            +
            \frac{1}{2\h} [r_{k-1}, r_{k-1}]_{\mdiamond}
            -
            \frac{1}{2\h} \{r^{cl}_{k-1}, r^{cl}_{k-1}\}_{\pifib}
        \right) \\
        &=
        \de^{-1} \left(
            \frac{1}{2\h} [r_{k-1}, r_{k-1}]_{\mdiamond}
            -
            \frac{1}{2\h} [r^{cl}_{k-1}, r^{cl}_{k-1}]_{\mdiamond}
        \right) \mod \h^2.
    \end{align*}
    On the other hand,
    \begin{align*}
        &\frac{1}{2\h} [r_{k-1}, r_{k-1}]_{\mdiamond}
        -
        \frac{1}{2\h} [r^{cl}_{k-1}, r^{cl}_{k-1}]_{\mdiamond}
        \\
        &\qquad=
        \frac{1}{2\h} [r_{k-1}, r_{k-1}]_{\mdiamond}
        -
        \frac{1}{2\h} [r_{k-1}, r^{cl}_{k-1}]_{\mdiamond}
        +
        \frac{1}{2\h} [r_{k-1}, r^{cl}_{k-1}]_{\mdiamond}
        -
        \frac{1}{2\h} [r^{cl}_{k-1}, r^{cl}_{k-1}]_{\mdiamond} \\
        &\qquad=
        \frac{1}{2\h}
        [r_{k-1}, (r_{k-1} - r^{cl}_{k-1})]_{\mdiamond}
        +
        \frac{1}{2\h}
        [(r_{k-1} - r^{cl}_{k-1}), r^{cl}_{k-1}]_{\mdiamond}
        = 0 \mod \h^2.
    \end{align*}
    Therefore Equation (\ref{r-r-cl-k}) holds for all $k$ and
    the proof is concluded.

\end{proof}

%
% Kontsevich's star product $*_K$ is equivalent to $\mstar$
%

\subsection{The Kontsevich star product $*_K$ is equivalent to $\mstar$}
\label{K-equiv-m}

Since the differential $\DEW$ (\ref{D-EW}) has the 
form\footnote{See Remark \ref{remark:adding-hbar}.} (\ref{DDD}),
we may use it to construct the Kontsevich star product as in
\cite{CEFT}. The class of the star product does not depend on this
particular choice of differential due to Theorem~\ref{nezalezh}. In
particular, we denote by $\tauEW$ the corresponding map
\eqref{tau-need}.

We will use the sequence of $\Linf$ quasi-isomorphisms (\ref{upper})
with $D = \DEW$ to the obtain the Kontsevich star product $*_K$
corresponding to $\pi$. Going through details of this construction,
we will produce an equivalence transformation between $*_K$ and
$\mstar$ (\ref{star-m}).

Let us consider the (fiberwise) Poisson-Lichnerowicz differential on
$\Omb(M, \cTp)[[\h]]$,
\begin{equation}
    \label{Lich-pa}
    \pa_{\pifib} = [\pifib, \cdot]_{SN},
\end{equation}
corresponding to the fiberwise Poisson structure $\pifib$
(\ref{pi-y}).  We can then rewrite $\DEW$ (\ref{D-EW-with-b-cl}) as
\begin{equation}
    \label{D-EW-with-Lich}
    \DEW
    =
    \n - \frac{1}{\h} \pa_{\pifib} (b^{cl}).
\end{equation}
Combining this last equation with the compatibility between $\n$ and
$\pifib$, we conclude that
\begin{equation}
    \label{D-EW-pi-y}
    \DEW \pifib = 0.
\end{equation}
In other words, the lift $\tauEW(\pi)$ of the formal Poisson
structure (\ref{pi}) to a $\DEW$-flat section of
$\cT^2_{poly}[[\h]]$ takes the following simple form:
\[
\tauEW(\pi)
=
\pi^{ij}(x) \pa_{y^i} \wedge \pa_{y^j}.
\]
An important consequence of this observation is that the components
of $\tauEW(\pi)$ do not depend on the fiber coordinates $y$'s.

Following Subsection~\ref{subsec:sequence}, we consider the ``tail''
of the differential $\DEW$,
\begin{equation}
    \label{tail}
    \mu^{EW}_U
    =
    - \G_{\h} - \frac{1}{\h} \, \pa_{\pifib} b^{cl},
\end{equation}
as a Maurer-Cartan element of the DGLA
\begin{equation}
    \label{DGLA-scTp-U}
    (\Omb(U, \scTp)[[\h]], d, [\cdot, \cdot]_{SN}),
\end{equation}
where $U$ is a coordinate open subset of $M$.  Then twisting the
$\Linf$ quasi-isomorphism $K$ (\ref{K-U}) by $\mu^{EW}_U$, we obtain
the $\Linf$ quasi-isomorphism
\begin{equation}
    \label{K-tw-mu-EW-U}
    K^{\mu^{EW}_U}:
    (\Omb(U, \scTp)[[\h]], \DEW,[\, , \,]_{SN})
    \brarrow
    (\Omb(U, \sCbu(\SM))[[\h]], \DEW + \pa^{\Hoch},[\, ,\,]_G).
\end{equation}
As explained in Subsection~\ref{subsec:sequence}, the $\Linf$
quasi-isomorphism $K^{\mu^{EW}_U}$ does not depend on the choice of
local coordinates on $U$. Hence we get a global $\Linf$
quasi-isomorphism
\begin{equation}
    \label{K-tw-mu-EW}
    K^{tw}:
    (\Omb(M, \scTp)[[\h]], \DEW,[\,,\,]_{SN})
    \brarrow
    (\Omb(M, \sCbu(\SM))[[\h]], \DEW + \pa^{\Hoch},[\, ,\,]_{G}).
\end{equation}

Let us denote by $\mu^K$ the Maurer-Cartan element of the DGLA
\begin{equation}
    \label{Cbu-D-EW}
    \left(
        \Omb(M, \sCbu(\SM))[[\h]],
        \DEW + \pa^{\Hoch}, [\cdot, \cdot]_{G}
    \right)
\end{equation}
obtained from $\tauEW(\pi) = \pifib$ via the $\Linf$
quasi-isomorphism $K^{tw}$:
\begin{equation}
    \label{mu-K}
    \mu^{K}
    =
    \sum_{n=1}^{\infty} \frac{1}{n!}
    K^{tw}_n (\tauEW(\pi), \tauEW(\pi), \ldots, \tauEW(\pi)).
\end{equation}
A simple degree bookkeeping shows that $\mu^K = \mu^K_0 + \mu^K_1 +
\mu^K_2$, where $\mu^K_0$ is a $0$-form with values in
$C^2(\SM)[[\h]]$, $\mu^K_1$ is a $1$-form with values in
$C^1(\SM)[[\h]]$, and $\mu^K_2$ is a $2$-form with values in
$C^0(\SM)[[\h]] = \SM[[\h]]$. More precisely,
\begin{align}
    \label{MC-Konts-0}
    &\mu^{K}_0 =
    \sum_{n=1}^{\infty}
    \frac{1}{n!} K_n (\pifib , \pifib, \ldots, \pifib),\\
    \label{MC-Konts-1}
&    \mu^{K}_1 =
    \sum_{n=1}^{\infty}
    \frac{1}{n!} K_{n+1}
    \left(
        \pifib , \pifib, \ldots, \pifib,
        -\G_{\h} - \frac{1}{\h} \pa_{\pifib} (b^{cl})
    \right),\\
    \label{MC-Konts-2}
&    \mu^{K}_2 =
    \sum_{n=1}^{\infty}
    \frac{1}{n! \, 2} K_{n+2}
    \left(
        \pifib , \pifib, \ldots, \pifib,
        -\G_{\h} - \frac{1}{\h} \pa_{\pifib} (b^{cl}),
        -\G_{\h} - \frac{1}{\h} \pa_{\pifib} (b^{cl})
    \right).
\end{align}

It is known that Kontsevich's star product corresponding to the
constant Poisson structure coincides with the Moyal star product,
see e.g. \cite{Zotov}. Therefore, since the components of $\pifib$
do not depend on the fiber coordinates $y$'s, we conclude
that\footnote{Due to
  (\ref{MC-Konts-0-1}), the product $\mdiamond$ can be written as $a_1
  \mdiamond a_2 = a_1 a_2 + \mu^K_0(a_1, a_2)$.}
\begin{equation}
    \label{MC-Konts-0-1}
    \mu^{K}_0 (a_1, a_2) =  a_1\,
    \sum_{n=1}^{\infty} \frac{1}{n!}
    \left(\pi^{ij}(x)
        \frac{\overleftarrow{\pa}}{\pa y^i} \,
        \frac{\overrightarrow{\pa}}{\pa y^j} \right)^n
    \, a_2
\end{equation}
where $a_1, a_2$ are local sections of $\SM[[\h]]$.

Using property P~\ref{P-vect-linear} from
Subsection~\ref{subsec:sequence}, we simplify $\mu^K_1$ and
$\mu^K_2$ as follows:
\begin{align}
    \label{MC-Konts-1-simpler}
&    \mu^{K}_1 = -
    \sum_{n=1}^{\infty}
    \frac{1}{n!} K_{n+1}
    \left(
        \pifib, \pifib, \ldots, \pifib,
        \frac{1}{\h} \pa_{\pifib} (b^{cl})
    \right),\\
    \label{MC-Konts-2-simpler}
&    \mu^{K}_2 =
    \sum_{n=1}^{\infty}
    \frac{1}{n! \, 2} K_{n+2}
    \left(
        \pifib, \pifib, \ldots, \pifib,
        \frac{1}{\h} \pa_{\pifib} (b^{cl}),
        \frac{1}{\h} \pa_{\pifib} (b^{cl})
    \right).
\end{align}
Due to property P~\ref{P-K1} from Subsection~\ref{subsec:sequence},
we can write
\begin{align*}
    -\frac{1}{\h} \pa_{\pifib} (b^{cl}) + \mu^K_1
    &=
    -\frac{1}{\h} K_1 (\pa_{\pifib} (b^{cl}) ) -
    \sum_{n=1}^{\infty}
    \frac{1}{n!} K_{n+1}
    \left(
        \pifib, \pifib, \ldots, \pifib,
        \frac{1}{\h} \pa_{\pifib} (b^{cl})
    \right) \\
    &=
    -\sum_{n=0}^{\infty} \frac{1}{\h n!} K_{n+1}
    \left(
        \pifib , \pifib, \ldots,
        \pifib, \pa_{\pifib} (b^{cl})
    \right).
\end{align*}
In other words,
\begin{equation}
    \label{b-cl-mu-K-1}
    -\frac{1}{\h}\pa_{\pifib} (b^{cl}) + \mu^K_1 =
    -\frac{1}{\h} K^{\pifib}_1 (\,\pa_{\pifib} (b^{cl}) \,)\,,
\end{equation}
where $K^{\pifib}$ is the $\Linf$ quasi-isomorphism obtained from
\[
K : (\Omb(M, \scTp)[[\h]], 0) \brarrow (\Omb(M, \sCbu(\SM))[[\h]],
\pa^{\Hoch})
\]
via twisting by $\pifib$ (\ref{pi-y}).  Using (\ref{F-1-al-inter})
from Appendix~\ref{App-B}, the map $K^{\pifib}_1$ intertwines the
Poisson-Lichnerowicz differential $\pa^{\pifib}$ (\ref{Lich-pa})
with the Hochschild differential $\pa^{\Hoch}_{\mdiamond}$
corresponding to the product (\ref{t-bullet}). Hence
(\ref{b-cl-mu-K-1}) can be rewritten as
\begin{equation}
    \label{b-cl-mu-K-1-simpler}
    - \frac{1}{\h} \pa_{\pifib} (b^{cl}) +
    \mu^K_1 = - \frac{1}{\h}
    \pa^{\Hoch}_{\mdiamond} K^{\pifib}_1 (b^{cl})
     = \frac{1}{\h} [K^{\pifib}_1 (b^{cl}), \cdot]_{\mdiamond}.
\end{equation}
Since the components of $\pifib$ do not depend on the fiberwise
coordinates $y$'s, it follows that, for every $n\ge 1$,
\[
K_{n+1} (\pifib, \pifib, \dots, \pifib, b^{cl}) = 0
\]
Hence
\begin{equation}
    \label{K-b-cl}
    K^{\pifib}_1 (b^{cl}) = b^{cl}
\end{equation}
and
\begin{equation}
    \label{b-cl-mu-K-1-simplest}
    - \frac{1}{\h} \pa_{\pifib} (b^{cl}) +
    \mu^K_1 = \frac{1}{\h} [b^{cl}, ~]_{\mdiamond}\,.
\end{equation}
Let us now find a simpler expression for $\mu^K_2$
(\ref{MC-Konts-2-simpler}).  Property P~\ref{P-arg-vect} from
Subsection~\ref{subsec:sequence} implies that
\[
K_2\left(
    \frac{1}{\h} \pa_{\pifib} (b^{cl}),
    \frac{1}{\h} \pa_{\pifib} (b^{cl})
\right)
= 0.
\]
Hence
\begin{align*}
    \mu^{K}_2
    &= K_2 \left(
        \frac{1}{\h} \pa_{\pifib} (b^{cl}),
        \frac{1}{\h} \pa_{\pifib} (b^{cl})
    \right)
    +
    \sum_{n=1}^{\infty}
    \frac{1}{n! \, 2} K_{n+2}
    \left(
        \pifib, \pifib, \ldots, \pifib,
        \frac{1}{\h} \pa_{\pifib} (b^{cl}),
        \frac{1}{\h} \pa_{\pifib} (b^{cl})
    \right) \\
    &=
    \sum_{n=0}^{\infty} \frac{1}{2 \h^2 n!} K_{n+2}
    \left(
        \pifib, \pifib, \ldots, \pifib,
        \pa_{\pifib} (b^{cl}), \pa_{\pifib} (b^{cl})
    \right).
\end{align*}
In other words,
\begin{equation}
    \label{mu-K-2}
    \mu^K_2
    =
    \frac{1}{2 \h^2} K^{\pifib}_2
    \left(
        \pa_{\pifib} (b^{cl}), \pa_{\pifib} (b^{cl})
    \right).
\end{equation}
Since $K^{\pifib}_{2}(b^{cl}, \pa_{\pifib} (b^{cl}))$ vanishes for
degree reasons, the component $\mu^K_2$ can be written as
\[
\mu^{K}_2
=
\frac{1}{2\h^2}\pa^{\Hoch}_{\mdiamond}
K^{\pifib}_{2}(b^{cl}, \pa_{\pifib} (b^{cl}))
+
\frac{1}{2\h^2} K^{\pifib}_{2}
\left(\pa_{\pifib} (b^{cl}), \pa_{\pifib} (b^{cl})\right)
+
\frac{1}{2\h^2} K^{\pifib}_{2}
\left(b^{cl}, (\pa_{\pifib})^2 (b^{cl})\right).
\]
Using Equation (\ref{F-2-and-F-1}) from Appendix~\ref{App-B}, we
obtain
\begin{equation}
    \label{MC-Konts-2-111}
    \mu^{K}_2
    =
    \frac{1}{2\h^2} \left(
        K^{\pifib}_{1}
        \left(
            \left[
                b^{cl}, \pa_{\pifib} (b^{cl})
            \right]_{SN}
        \right)
        -
        \left[
            K^{\pifib}_{1}(b^{cl}),
            K^{\pifib}_{1}(\pa_{\pifib} (b^{cl}))
        \right]_G
    \right).
\end{equation}
Thus, due to Equation (\ref{K-b-cl}),
\begin{equation}
    \label{MC-Konts-2-1111}
    \mu^{K}_2
    =
    \frac{1}{2\h^2} \{b^{cl}, b^{cl}\}_{\pifib}
    -\frac{1}{2\h^2} [b^{cl}, b^{cl}]_{\mdiamond}\,.
\end{equation}
Combining equations (\ref{Fed-class}) and (\ref{b-cl-eq}), we deduce
that
\[
\frac{1}{2 \h^2} [b,b]_{\tdiamond} - \frac{1}{2 \h^2} \{b^{cl},
b^{cl}\}_{\pifib} + \frac{1}{\h} \n (b - b^{cl})
= 0.
\]
Therefore, Equation~(\ref{MC-Konts-2-1111}) can be rewritten as
\begin{equation}
    \label{MC-Konts-2-5}
    \mu^{K}_2
    =
    \frac{1}{\h} \n (b - b^{cl})
    +
    \frac{1}{2 \h^2}[b, b]_{\mdiamond}
    - \frac{1}{2 \h^2} [b^{cl}, b^{cl}]_{\mdiamond}.
\end{equation}
Combining this equation with (\ref{b-cl-mu-K-1-simplest}), we obtain
\begin{equation}
    \label{mu-K-and}
    -\G_h - \frac{1}{\h} \pa_{\pifib}  b^{cl} + \mu^K
    =
    -\G_h + \frac{1}{\h}[b^{cl}, \cdot]_{\mdiamond}
    + \frac{1}{\h} \n (b - b^{cl})
    + \frac{1}{2 \h^2}[b, b]_{\mdiamond}
    - \frac{1}{2 \h^2} [b^{cl}, b^{cl}]_{\mdiamond} + \mu^K_0.
\end{equation}
The left-hand side of (\ref{mu-K-and}) is a Maurer-Cartan element of
the DGLA
\begin{equation}
    \label{DGLA-sCbu-U}
    (\Omb(U, \sCbu(\SM))[[\h]], d + \pa^{\Hoch}, [\cdot, \cdot]_G),
\end{equation}
where $U$ is a coordinate open subset of $M$.

The next step in the construction of the star product on $M$ is to
eliminate the components $\mu^K_2$ and $\mu^K_1$ via an equivalence
transformation. Our goal is to show that the component $\mu^K_2$ can
be eliminated in a way that gives us a Maurer-Cartan element which
combines both the defining part
\[
-\G_{\h} + \frac{1}{\h}[b, \cdot]_{\mdiamond}
\]
of the quantum Fedosov differential $D^F_{\h}$ (\ref{D-F-h1}) and
the defining part $\mu^K_0$ (\ref{MC-Konts-0-1}) of the fiberwise
product $\mdiamond$ (\ref{t-bullet}). We have the following result:
\begin{pred}
    \label{ubit'}
    There exists an element $\xi \in \h \Om^1(M,\SM)[[\h]]$ such that
    \begin{equation}
        \label{mu-K-mu-F}
         \left(
            -\G_{\h}
         - \frac{1}{\h} \pa_{\pifib} b^{cl}
            + \mu^K
        \right)^{\exp(\xi)}
        = -\G_{\h} + \frac{1}{\h}[b, \cdot]_{\mdiamond} + \mu^K_0
    \end{equation}
    in the DGLA (\ref{DGLA-sCbu-U}) for every coordinate open subset
    $U$ of $M$.
\end{pred}
\begin{proof}
    Let us denote by $\tmu^K$ the left-hand side of
    (\ref{mu-K-and}):
    \begin{equation}
        \label{tmu-K}
        \tmu^K
        =
        -\G_{\h} - \frac{1}{\h} \pa_{\pifib}  b^{cl} + \mu^K.
    \end{equation}
    As we have remarked, $\tmu^K$ is a Maurer-Cartan element of the
    DGLA (\ref{DGLA-sCbu-U}) for every open coordinate subset $U$.
    According to the formula in (\ref{action-A}) in
    Appendix~\ref{App-B}, we have
    \begin{equation}
        \label{action1}
         \big(\tmu^K \big)^{\exp(\xi)}
        =
        \tmu^K + \frac{\exp([\cdot ,\xi]_G) -1}{[\cdot ,\xi ]_G}
        \left(d \xi + \pa^{\Hoch} \xi + [\tmu^K, \xi]_G\right).
    \end{equation}
    Since the Gerstenhaber bracket is zero if both arguments take
    values in $C^0(\SM)$, we conclude that only the first two terms $1
    + \frac{1}{2}[\cdot ,\xi ]_G$ of the series $\frac{\exp([\cdot
      ,\xi ]_G) -1}{[\cdot ,\xi ]_G}$ contribute to the right-hand
    side of (\ref{action1}). Hence,
\begin{equation}
\label{with2terms}
\big(\tmu^K \big)^{\exp(\xi)}
= \tmu^K + \left( 1 + \frac{1}{2}[\cdot ,\xi ]_G
\right)
\left(d \xi + \pa^{\Hoch} \xi + [\tmu^K, \xi]_G\right).
\end{equation}
    Using (\ref{mu-K-and}), we write (\ref{with2terms}) as
    follows:
    \begin{align*}
        \big( \tmu^K \big)^{\exp(\xi)}
        &=
        - \G_h
        + \frac{1}{\h}[b^{cl}, \cdot]_{\mdiamond}
        + \frac{1}{\h} \n (b - b^{cl})
        + \frac{1}{2 \h^2}[b, b]_{\mdiamond}
        - \frac{1}{2 \h^2} [b^{cl}, b^{cl}]_{\mdiamond} + \mu^K_0 \\
        &\quad
        + \n \xi
        + \pa^{\Hoch} \xi
        + \frac{1}{\h}[b^{cl}, \xi]_{\mdiamond}
        + [\mu^K_0, \xi]_G
        + \frac{1}{2} \left[
            \pa^{\Hoch} \xi + [\mu^K_0, \xi]_G, \xi
        \right]_G.
    \end{align*}
    Using (\ref{pa-Hoch-0}) and
    (\ref{MC-Konts-0-1}), we combine $ \pa^{\Hoch}
    \xi + [\mu^K_0, \xi]_G$ into $-[\xi, \cdot]_{\mdiamond}$. Thus
    $\big(\tmu^K \big)^{\exp(\xi)}$ can be further simplified as
    \begin{align}
    \big( \tmu^K \big)^{\exp(\xi)}
        &=
        - \G_h
        + \frac{1}{\h}[b^{cl}, \cdot]_{\mdiamond}
        + \frac{1}{\h} \n (b - b^{cl})
        + \frac{1}{2 \h^2}[b, b]_{\mdiamond}
        - \frac{1}{2 \h^2} [b^{cl}, b^{cl}]_{\mdiamond}
        + \mu^K_0
        \nonumber \\
        &\quad
        \label{exp-xi-tmu}
        + \n \xi
        - [\xi, \cdot]_{\mdiamond}
        + \frac{1}{\h}[b^{cl}, \xi]_{\mdiamond}
        - \frac{1}{2}[\xi, \xi]_{\mdiamond}.
    \end{align}
    One can now show that, by plugging
    \begin{equation}
        \label{xi}
        \xi = \frac{1}{\h} (b^{cl} - b)
    \end{equation}
    into (\ref{exp-xi-tmu}), we obtain the desired identity
    \[
     \big( \tmu^K \big)^{\exp(\xi)}
    =
    - \G_{\h} + \frac{1}{\h}[b, \cdot]_{\mdiamond} + \mu^K_0.
    \]

Equations (\ref{b}) and (\ref{b-cl-r-cl}) imply that
$$
\xi = \frac{1}{\h} (r^{cl} - r)\,.
$$
By Proposition \ref{prop:r-r-cl}, it follows that $\xi = 0 ~ {\rm
mod}~ \h$, and this concludes the proof.
\end{proof}

Let us denote the right-hand side of (\ref{mu-K-mu-F}) by $\mmu_U$:
\begin{equation}
    \label{mu-Fed}
    \mmu_U
    =
    -\G_{\h} + \frac{1}{\h}[b, \, ]_{\mdiamond} + \mu^K_0,
\end{equation}
where $b$ is defined in (\ref{b}), and an expression for $\mu^K_0$
is given in (\ref{MC-Konts-0-1}).  As we remarked above, $\mmu_U$ is
a Maurer-Cartan element of the DGLA (\ref{DGLA-sCbu-U}), and
Proposition~\ref{ubit'} says that $\mmu_U$ is equivalent to the
Maurer-Cartan element $\tmu^K$.

Twisting the DGLA (\ref{DGLA-sCbu-U}) by $\mmu_U$ (\ref{mu-Fed}), we
get the DGLA $\Omb(U, \sCbu(\SM))[[\h]]$ with differential $D^F_{\h}
+ \pa^{\Hoch}_{\mdiamond}$, where $D^F_{\h}$ is given in
(\ref{D-F-h1}).  Since the differential $D^F_{\h} +
\pa^{\Hoch}_{\mdiamond}$ does not depend on the choice of
coordinates on $U$, twisting by $\mmu_U$ (\ref{mu-Fed}) gives us the
DGLA
\begin{equation}
    \label{DGLA-F}
    \left(
        \Omb(M, \sCbu(\SM))[[\h]],
        D^F_{\h} + \pa^{\Hoch}_{\mdiamond},
        [\cdot, \cdot]_G
    \right).
\end{equation}
To obtain Kontsevich's star product we need to further modify the
Maurer-Cartan element (\ref{mu-Fed}) by an equivalence
transformation of the form
\begin{equation}
    \label{equiv}
    T = \exp(\xi_1),
    \quad
    \textrm{with}
    \quad
    \xi_1\in \h\, \Om^0(M, C^1(\SM))[[\h]],
\end{equation}
to get the Maurer-Cartan element\footnote{Recall that the differential
  $D$ (\ref{DDD}) we use for the construction of Kontsevich's star
  product is $\DEW$ (\ref{D-EW}).}
\begin{equation}
    \label{mu-U}
    \mu_U = -\G_{\h} - \de +  \frac{1}{\h}\{r^{cl}, \,
    \}_{\pifib}  + \Pi^K,
\end{equation}
where $\Pi^K\in \Om^0(M, C^2(\SM))[[\h]]$\,.  Then the element $\Pi^K
\in \Om^0(M, C^2(\SM))[[\h]]$ is a Maurer-Cartan element of the DGLA
\begin{equation}
    \label{DGLA-K}
    \left(
        \Omb(M, \sCbu(\SM))[[\h]],
        \DEW + \pa^{\Hoch},
        [\cdot, \cdot]_G
    \right),
\end{equation}
which is, in turn, quasi-isomorphic to $\sCbu(\OM)[[\h]]$ via the
map $\tauEW$ in (\ref{tau-need}).  Since $\Pi^K$ has zero exterior
degree, the Maurer-Cartan equation for $\Pi^K$ is equivalent to two
equations:
\begin{align}
    \label{Pi-K-pa-Hoch}
    &\pa^{\Hoch} \Pi^K
    +
    \frac{1}{2} [\Pi^K, \Pi^K]_{G}
    = 0,\\
    \label{D-EW-Pi-K}
&    \DEW \Pi^K = 0.
\end{align}
The last equation implies that $\Pi^K$ lies in the image of $\tauEW$
(\ref{tau-iso}), i.e., $\Pi^K = \tauEW(\Pi)$ for a unique $\Pi \in
\h C^2(\cO_M)[[\h]]$.  Then (\ref{Pi-K-pa-Hoch}) implies that the
product
\begin{equation}
    \label{star-K1}
    f_1 * f_2 = f_1 f_2 + \Pi(f_1, f_2)
\end{equation}
on $\cO(M)[[\h]]$ is associative. This is exactly Kontsevich's star
product corresponding to the formal Poisson structure $\pi$
(\ref{pi}).

A more explicit way to get the star product (\ref{star-K1}) from the
element $\Pi^K$ is to use the isomorphism of $\bbC[[\h]]$-modules
$\tauEW : \cO(M)[[\h]]\to \G(M, \SM)[[\h]] \cap \ker \DEW.$ This
isomorphism is obtained by iterating the following equation in
degrees in the fiber coordinates $y$'s:
\begin{equation}
    \label{iter-tau-EW}
    \tauEW(f)
    =
    f + \de^{-1}
    \left(
        \n \tauEW(f)+ \{r^{cl}, \tauEW(f)\}_{\pifib}
    \right),
    \quad
    \textrm{for}
    \quad
    f \in \cO(M)[[\h]].
\end{equation}
Equation (\ref{Pi-K-pa-Hoch}) implies that the formula
\begin{equation}
    \label{bul-Pi-K}
    a_1 \diamond a_2
    =
    a_1 a_2 + \Pi^K (a_1,a_2),
    \quad
    \textrm{for}
    \quad
    a_1, a_2 \in \G(M, \SM)[[\h]],
\end{equation}
defines an associative product on $\SM[[\h]]$. Equation
(\ref{D-EW-Pi-K}), in turn, implies that the differential $\DEW$ is
a derivation of the product (\ref{bul-Pi-K}).  Thus the formula
\begin{equation}
    \label{star-K11}
    f_1 * f_2
    =
    \sigma\left(
        \tauEW(f_1) \diamond \tauEW(f_2)
    \right)
    \quad
    \textrm{for}
    \quad f_1, f_2 \in \cO(M)[[\h]]
\end{equation}
defines an associative product on $\cO(M)[[\h]]$. According to the
construction of the map $\tauEW$ (see
Subsection~\ref{subsec-Fedosov-resolutions}) the star product
(\ref{star-K11}) coincides with (\ref{star-K1}).

To construct an equivalence transformation between the star products
(\ref{star-m}) and (\ref{star-K11}), we recall that the
Maurer-Cartan elements $\mmu_U$ (\ref{mu-Fed}) and $\mu_U$
(\ref{mu-U}) are connected by the equivalence transformation
(\ref{equiv}):
\begin{equation}
    \label{mu-U-mu-m-U}
    \mu_U
    =
    \mmu_U
    + \frac{e^{[\cdot, \xi_1]_G} - 1}{[\cdot, \xi_1]_G}
    \left(d \xi_1 + \pa^{\Hoch} \xi_1 + [\mmu_U, \xi_1]_G\right).
\end{equation}
This equation can be rewritten as
\begin{equation}
    \label{mu-U-mu-m-U1}
    - \G_{\h} + \frac{1}{\h}\{b^{cl}, \cdot \}_{\pifib} + \Pi^K
    =
    - \G_{\h} + \frac{1}{\h}[b, \cdot]_{\mdiamond} + \mu^K_0
    + \frac{e^{[\cdot, \xi_1]_G} - 1}{[\cdot, \xi_1]_G}
    \left(D^F_{\h} \xi_1 + \pa^{\Hoch}_{\mdiamond} \xi_1\right).
\end{equation}
Since $\xi_1$ has zero exterior degree,  (\ref{mu-U-mu-m-U1}) is
equivalent to the pair of equations
\begin{align}
    \label{mu-U-mu-m-U-raz}
&    \frac{1}{\h}\{b^{cl}, \cdot\}_{\pifib} -\G_{\h}
    =
    -\G_{\h} + \frac{1}{\h}[b, \cdot]_{\mdiamond}
    + \frac{e^{[~, \xi_1]_G} - 1}{[\cdot, \xi_1]_G}
    \, D^F_{\h} \xi_1,\\
    \label{mu-U-mu-m-U-dva}
&    \Pi^K
    =
    \mu^K_0 + \frac{e^{[\cdot, \xi_1]_G} - 1}{[\cdot, \xi_1]_G}
    \, \pa^{\Hoch}_{\mdiamond} \xi_1.
\end{align}
Equation (\ref{mu-U-mu-m-U-raz}) says that the transformation
(\ref{equiv}) intertwines the differentials $\DEW$ and $D^F_{\h}$:
\begin{equation}
    \label{D-F-D-EW-xi1}
    D^F_{\h} e^{\xi_1}\, (a)
    =
    e^{\xi_1}\, \DEW (a),
    \quad
    \textrm{for}
    \quad
    a \in \G(M, \SM)[[\h]],
\end{equation}
and Equation (\ref{mu-U-mu-m-U-dva}) implies that the transformation
(\ref{equiv}) intertwines the fiberwise products $\mdiamond$ from
(\ref{t-bullet}) and $\diamond$ as in (\ref{bul-Pi-K}):
\begin{equation}
    \label{Pi-K-tbul-xi1}
    e^{\xi_1} (a_1) \mdiamond e^{\xi_1}(a_2)
    =
    e^{\xi_1} (a_1 \diamond a_2),
    \quad
    \textrm{with}
    \quad
    a_1,a_2 \in \G(M, \SM)[[\h]].
\end{equation}
Let us consider the map $E : \cO(M)[[\h]] \to \cO(M)[[\h]]$,
\begin{equation}
    \label{equiv-K-F}
    E(f) =
    \sigma \left(e^{\xi_1} \tauEW (f)\right)
    \quad
    \textrm{for}
    \quad
    f\in \cO(M)[[\h]]\,.
\end{equation}
Since $e^{\xi_1}$ intertwines the differentials $D^F_{\h}$ and
$\DEW$ we conclude that
\begin{equation}
    \label{E-D-F-D-EW}
    e^{\xi_1} \circ  \tauEW (f)
    =
    \mtau \circ E (f),
    \quad
    \textrm{for}
    \quad f\in \cO(M)[[\h]],
\end{equation}
where the map $\mtau$ is defined in (\ref{iter-tau-h}). Using the
definition of $\mstar$ (\ref{star-m}) and equations
(\ref{Pi-K-tbul-xi1}) and (\ref{E-D-F-D-EW}), we get the following
identities:
\begin{align*}
    E(f_1 * f_2)
    &=
    \sigma \left(
        e^{\xi_1} \tauEW (f_1 * f_2)
    \right)
    =
    \sigma \left(
        e^{\xi_1} \left(
            \tauEW(f_1) \diamond \tauEW(f_2)
        \right)
    \right) \\
    &=
    \sigma \left(
        e^{\xi_1} (\tauEW (f_1)) \mdiamond e^{\xi_1}(\tauEW (f_2))
    \right)
    =
    \sigma\left(
        \mtau (E (f_1)) \mdiamond \mtau (E(f_2))
    \right) \\
    &= E(f_1) \mstar E(f_2)
\end{align*}
for all $f_1, f_2 \in \cO(M)[[\h]]$. Since $E$ starts with the
identity in the zeroth order in $\h$, Kontsevich's star product (\ref{star-K11})
is indeed equivalent to the modified Fedosov star product (\ref{star-m}) via
$E$.

%
% The star product $\mstar$ coincides with the original Fedosov
% star product
%

\subsection{The star product $\mstar$ coincides with the original
Fedosov star product}
\label{m-equiv-F}

The construction of the original Fedosov star product $*_F$ is based
on a fiberwise product different from $\mdiamond$ (\ref{t-bullet}),
namely, one uses the following product on $\SM[[\h]]$:
\begin{equation}
    \label{bullet}
    a_1 \Fdiamond a_2
    =
    a_1 \exp \left(
        \h \pi^{ij}_1(x)
        \frac{\overleftarrow{\pa}}{\pa y^i} \,
        \frac{\overrightarrow{\pa}}{\pa y^i}
    \right)
    a_2.
\end{equation}
We will show that the product $\Fdiamond$ can be connected to the
product $\mdiamond$ (\ref{t-bullet}) by an equivalence
transformation of a specific form:
\begin{lem}
    \label{lemma:NiceEquivalence}
    The products $\mdiamond$ and $\Fdiamond$ are fiberwise
    equivalent,
    \begin{equation}
        \label{eq:FiberwiseEquivalence}
        P(a_1) \Fdiamond P(a_2) = P (a_1\, \mdiamond\, a_2),
    \end{equation}
    via an equivalence transformation of the form
    \begin{equation}
        \label{eq:NiceEquivalence}
        P = \exp(\chi)
        \quad
        \textrm{with}
        \quad
        \chi =
        \left(\hbar \chi_1 + \hbar^2 \chi_2 + \cdots\right)^i_j
        y^j \partial_{y^i},
    \end{equation}
    where $\chi_r \in \G(M, TM \otimes T^*M)$ for each $r$.
\end{lem}
\begin{proof}
    The statement is fiberwise so it suffices to consider the
    following situation on $\mathbb{R}^{2n}$. Let $\star$ be the
    ordinary Weyl-Moyal star product, i.e.,
    \[
    f_1 \star f_2
    = \mu_0 \circ \exp\left(
        \hbar \pi_1^{ij} \partial_{y^i}
        \otimes \partial_{y^j}
    \right) (f_1 \otimes f_2),
    \]
    where $\mu_0(f_1 \otimes f_2) = f_1 f_2$ denotes the ordinary
    commutative product on $\cO(\bbR^{2n})[[\h]]$, and $\pi_1^{ij}$ is
    a constant antisymmetric nondegenerate $2n \times 2n$-matrix (with
    complex entries).
    We must construct a specific equivalence transformation intertwining
    the product $\star$ with the star product
    \begin{equation}
        \label{tstarWeyl}
        f_1 \wstar f_2
        =
        \mu_0 \circ \exp\left(
            \pi^{ij} \partial_{y^i} \otimes \partial_{y^j}
        \right) (f_1 \otimes f_2),
    \end{equation}
    where $\pi$ is a formal power series of constant antisymmetric
    $2n\times 2n$-matrices (with complex entries) starting with $\h
    \pi_1$
    \[
    \pi^{ij} = \hbar \pi_1^{ij} + \hbar^2 \pi_2^{ij} + \hbar^3
    \pi_3^{ij}
    + \cdots.
    \]

 Let us consider the following sets of operators,
    \[
    \cB =
    \left\{
        B^{ij} \partial_{y^i} \otimes \partial_{y^j}
        \; \big| \;
        B^{ij} \in \hbar \mathbb{C}[[\hbar]]
    \right\},\;\;\;\;
    \cA =
    \left\{
        \chi \otimes \id + \id \otimes \chi
        \; \big | \;
        \chi = \chi^i_j y^j \partial_{y^i}
        \; \textrm{with} \;
        \chi^i_j \in \hbar \mathbb{C}[[\hbar]]
    \right\},
    \]
    acting on the tensor product of two copies of
    $\cO(\mathbb{R}^{2n})[[\hbar]]$.  Note that $\mathcal{A}$ is a
    subalgebra, while $\mathcal{B}$ is an abelian subalgebra of the Lie
    algebra of all endomorphisms of $\cO(\mathbb{R}^{2n})[[\hbar]]
    \otimes_{\bbC[[\h]]} \cO(\mathbb{R}^{2n})[[\hbar]]$. The property
    \[
    [\mathcal{A}, \mathcal{B}] \subseteq \mathcal{B}
    \]
    implies that
    \begin{equation}
        \label{exp-A-B}
        \exp(A) \circ \exp(B) \circ \exp(-A)
        =
        \exp(e^{[A, \cdot]}(B))
        =
        \exp(\tilde{B})
    \end{equation}
    with $\tilde{B} = B + [A, B] + \cdots \in \mathcal{B}$.  Let
    us suppose that the matrix $\pi^{ij}$ in (\ref{tstarWeyl}) has the
    form $(m \ge 2)$
    \[
    \pi^{ij}
    =
    \hbar \pi_1^{ij} + \hbar^m \pi_m^{ij} +
    \hbar^{m+1} \pi_{m+1}^{ij} + \hbar^{m+2} \pi_{m+2}^{ij}
    + \cdots.
    \]
    In other words, $\pi^{ij} - \h\pi^{ij}_1 = 0 \mod \h^m$.  Consider
    an equivalence transformation $P_m$ of the form
    \begin{equation}
        \label{P-m}
        P_m
        =
        \exp \left(\h^m (\chi_m)^i_j y^j \pa_{y^i} \right),
    \end{equation}
    where $(\chi_m)^i_j$ is a constant $2n\times 2n$ matrix.  Due to
    (\ref{exp-A-B}), we have
    \begin{align*}
        P_m (P^{-1}_m f_1 \wstar P_m^{-1} f_2)
        &=
        \mu_0
        \circ
        \exp(\chi \otimes \id + \id \otimes \chi)
        \circ
        \exp(\pi)
        \circ
        \exp(-\chi \otimes \id - \id \otimes \chi)
        (f_1 \otimes f_2) \\
        &=
        \mu_0
        \circ
        \exp(\pi')
        \circ
        (f_1 \otimes f_2),
    \end{align*}
    where
    \[
    (\pi')^{ij}
    =
    \h \pi_1^{ij}
    +
    \hbar^m \left(
        \pi_m^{ij} - \pi_1^{kj}(\chi_m)^i_k + \pi_1^{ki}(\chi_m)^j_k
    \right)
    +
    \hbar^{m+1} \pi_{m+1}^{ij}
    +
    \hbar^{m+2} \pi_{m+2}^{ij} + \cdots.
    \]
    Since the matrix $\pi_1^{ij}$ is nondegenerate, we can choose
    the matrix $(\chi_m)^i_j$ in such way that
    \[
    \pi_m^{ij} - \pi_1^{kj}(\chi_m)^i_k + \pi_1^{ki}(\chi_m)^j_k = 0.
    \]
    The new product $f_1, f_2 \mapsto P_m (P^{-1}_m f_1
    \wstar P_m^{-1} f_2)$ has the form
    \[
    P_m (P^{-1}_m f_1 \wstar P_m^{-1} f_2)
    =
    \mu_0 \circ
    \exp(\pi')
    \circ
    (f_1 \otimes f_2)\,,
    \]
    where $\pi' = \h\pi_1 \mod \h^{m+1}$.  Thus the desired
    equivalence transformation $P$ is obtained as an infinite
    product
    \[
    P = \dots P_4 \, P_3\, P_2,
    \]
    where the $m$-th transformation $P_m$ has the form (\ref{P-m}).
    This infinite product converges in the $\h$-adic
    topology, and it is clear that $P$ has the form
    \[
    P = \exp(\chi)
    \]
    where $\chi$ is a linear vector field. This concludes the proof.
\end{proof}

Since the exponent of $P$ is a vector field, it is also an
automorphism of the undeformed product:
\begin{equation}
    \label{P-prod}
    P(a_1)  P(a_2) = P (a_1 a_2),
\end{equation}
for all $a_1, a_2\in \G(M, \SM)[[\h]]$\,. Furthermore, $P$ preserves
the degree in $y$'s and transforms the differential $D^F_{\h}$
(\ref{D-F-h1}) to
%\[
%\DF = P \,D^F_{\h} \,P^{-1}
%\]
\begin{equation}
    \label{D-F}
    \DF =P \,D^F_{\h} \,P^{-1}= P\, \n\, P^{-1}  + \frac{1}{\h}[P(b), \cdot ]_{\Fdiamond}
\end{equation}
with the element $b$  defined in (\ref{b}).  Since $P$ has the form
\eqref{eq:NiceEquivalence}, the operator
\[
\n_0 =  P\, \n \, P^{-1}
\]
is again a connection (possibly with Christoffel symbols depending
on $\h$). Furthermore, $\n_0$ is a derivation of $\Fdiamond$ because
$\n$ is a derivation of $\mdiamond$.  As for the curvature form
(\ref{R-h}), we have
\[
(\n_0)^2
= P\, \n^2\, P^{-1}
= \frac{1}{\h} P\, [R, \cdot]_{\mdiamond}\, P^{-1}
= \frac{1}{\h} [P(R), \,]_{\Fdiamond}.
\]
Since the operator $P$ preserves the degree in $y$'s, we have
\begin{equation}
    \label{P-R}
    P(R) = \frac{1}{2}d x^i\, d x^j R^0_{ij\,\,kl}(x) y^k y^l,
\end{equation}
where $R^0_{ij\,\,kl}(x)$ are components (possibly depending on
$\h$) of the curvature tensor for $\n_0$.

Applying the operator $P$ to Equation (\ref{Fed-class}), and using
the fact that the components of $\omega$ (\ref{F-class}) do not
depend on the fiber coordinates $y$'s, we have that
\begin{equation}
    \label{Fed-class1}
    \frac{1}{\h} P(R) + \frac{1}{\h} \n_0 P(b) +
    \frac{1}{2\h^2} [P(b) , P(b)]_{\Fdiamond} = - \omega,
\end{equation}
Thus $\DF$ in (\ref{D-F}) is the quantum Fedosov differential from
the original construction in \cite{F}, and the Fedosov class of the
resulting star product is represented by $\omega$ (\ref{F-class}).

Let $\Ftau$ be the isomorphism
\begin{equation}
    \label{lift1}
    \Ftau:
    \cO(M)[[\h]] \to \G(M, \SM)[[\h]]\, \cap\, \ker \DF
\end{equation}
which lifts functions on $M$ to $\DF$-flat sections of the sheaf
$\SM[[\h]]$. Similarly to $\mtau$, the isomorphism $\Ftau$ is
defined by iterating the equation
\begin{equation}
    \label{iter-tau-DF}
    \Ftau(a)
    =
    a + \de^{-1}\left(
        \n_0 \Ftau (a) + [r^{\mathrm{F}}, \Ftau(a)]_{\Fdiamond}
    \right),
\end{equation}
where $a\in \G(M, \SM)[[\h]]$ and $r^{\mathrm{F}} = P(b)+ \h\, d x^i
\omega_{ij}(x, \h) y^j$.  The (original) Fedosov star product is defined
in terms of $\Ftau$ and the fiberwise product $\Fdiamond$ for $f_1,
f_2 \in \cO(M)[[\h]]$ as
\begin{equation}
    \label{star-F}
    f_1 *_F f_2
    =
    \sigma\left(
        \tau^F(f_1) \Fdiamond \tau^F(f_2)
    \right).
\end{equation}

Since the transformation $P$ (\ref{eq:NiceEquivalence}) intertwines
the differentials $\DF$ and $D^F_{\h}$, we conclude that $P\circ
\tau^{m} (f)$ is $\DF$-flat for every $f\in \cO(M)[[\h]]$.  On the
other hand, since the transformation $P$ preserves the degree in the
fiber coordinates $y$'s, we have that $ P \circ \mtau(f) \Big|_{y=0}
= f$. Hence
\begin{equation}
    \label{P-tau-m-tau-F}
    P \circ \mtau (f) = \Ftau(f).
\end{equation}

Combining the last equation with (\ref{eq:FiberwiseEquivalence}),
and using the fact that $P$ preserves the degree in fiber
coordinates $y$'s, we get the following series of identities: for
$f_1, f_2 \in \cO(M)[[\h]]$,
\begin{align*}
    f_1 *_F f_2
    &=
    \sigma\left(
        \Ftau(f_1) \Fdiamond \Ftau(f_2)
    \right)
    =
    \sigma\left(
        P \circ \mtau(f_1) \Fdiamond P \circ \mtau(f_2)
    \right) \\
    &=
    \sigma\left(
        P (\mtau(f_1) \mdiamond \mtau(f_2))
    \right)
    =
    \sigma\left(
        (\mtau(f_1) \mdiamond \mtau(f_2))
    \right) \\
    &=
    f_1 \mstar f_2.
\end{align*}
Thus the star product $\mstar$ (\ref{star-m}) coincides with the
original Fedosov product $*_F$ (\ref{star-F}), and
Theorem~\ref{F-versus-K} is proved.

%
% Here are the appendices
%

\appendix

%
% Formal differential equations
%

\section{Formal differential equations}
\label{App-A}

In this appendix we collect some results on differential equations
in $\mathbb{C}[[\hbar]]$-modules which are needed throughout Section
\ref{sec:Morita}. Most of the material is well-known or can be
easily reconstructed from well-known results, see e.g. the textbook
\cite[Sect.~3]{waldmann:2007a}.

Let us consider the following purely algebraic situation. We fix a
commutative ring $\ring{C}$ containing $\mathbb{Q}$, let $V$ be a
$\ring{C}$-module, and let $\mathcal{D} \subseteq
\End_{\ring{C}}(V)$ be a unital sub-algebra. In our case we usually
have $\ring{C} = \mathbb{C}$, $V = \cO(M)$ or $\Gamma(M, E)$ for
some vector bundle $E \longrightarrow M$, and $\mathcal{D}$ being
the differential operators on $V$. Let us also consider the
$\hbar$-adically complete $\ring{C}$-module $(V[t])[[\hbar]]$, i.e.,
in each order of $\hbar$ we have a polynomial in $t$ with
coefficients in $V$. Note that this is different from
$(V[[\hbar]])[t]$, which is a proper sub-module of
$(V[t])[[\hbar]]$. Let $D(t) \in \hbar(\mathcal{D}[t])[[\hbar]]$ and
$w(t)\in \h (V[t])[[\hbar]] $ be given, and consider the
differential equation
\begin{equation}
    \label{eq:vdotDv}
    \frac{d}{d t} v(t) = w(t) + D(t) v(t)
\end{equation}
with initial condition $v(0) = v_0 \in V[[\hbar]]$.
\begin{pred}
    \label{proposition:ODEGeneral}
    For each initial condition $v(0) = v_0$
    equation~\eqref{eq:vdotDv} has a unique solution $v(t) \in
    (V[t])[[\hbar]]$. Moreover, if $w = 0$ then the flow map $v_0
    \mapsto v(t)$ is a formal series $\id + \sum_{r=0}^\infty \hbar^r
    D_r(t)$ with $D_r(t) \in \mathcal{D}[t]$.
\end{pred}
\begin{proof}
    First we rewrite \eqref{eq:vdotDv} as the integral equation
    \begin{equation}
        \label{eq:iter-v}
        v(t) = v(0) + \int_0^t  (w(\tau) + D(\tau) v(\tau)) d \tau,
    \end{equation}
    incorporating the initial condition. Since in each
    order of $\hbar$, $w(t)$ and $D(\tau)v(\tau)$ are polynomials in
    $\tau$, the integral operator is a purely algebraic gadget defined
    by linear extension of $\int_0^t \tau^n d \tau = \frac{1}{n+1}
    \tau^{n+1}$ (this is also the reason why we require $\mathbb{Q}
    \subseteq \ring{C}$).  Since by assumption $D(t)$ and $w(t)$ are at
    least of order $\hbar$, the right-hand side of (\ref{eq:iter-v})
    is directly shown to be a \emp{contracting} endomorphism in the
    $\hbar$-adic topology of the complete module $(V[t])[[\hbar]]$.
    It follows from the usual fixed point argument that there is a unique
    solution of (\ref{eq:iter-v}) which is the unique solution of
    \eqref{eq:vdotDv} with correct initial condition, see
    e.g.~\cite[Sect.~6.2.1]{waldmann:2007a}. When $w = 0$, the
    iteration clearly produces a flow map of the specified type.
\end{proof}
\begin{example}
    Let $\mathcal{A}$ be a $\ring{C}$-algebra and let $\star$ be an
    associative deformation of $\mathcal{A}$, so that
    $\mathcal{A}[[\hbar]]$ is a $\ring{C}[[\hbar]]$-algebra with
    respect to $\star$. The product $\star$ extends to
    $(\mathcal{A}[t])[[\hbar]]$ in the obvious way, making it a
    $\ring{C}[[\hbar]]$-algebra. Let $d(t) \in
    \hbar(\mathcal{A}[t])[[\hbar]]$ be given. Then for every $a_0 \in
    \mathcal{A}[[\hbar]]$ the differential equation
    \begin{equation}
        \label{eq:ODEAlgebra}
        \frac{d}{d t} a(t) = d(t) \star a(t)
        \quad
        \textrm{with}
        \quad
        a(0) = a_0
    \end{equation}
    has a unique solution by Proposition~\ref{proposition:ODEGeneral}.
\end{example}
\begin{example}
    \label{example:ODEDiffOps}
    Let $\mathcal{A} = \DiffOp(\Gamma(M, E))$ be the differential
    operators on some vector bundle $E \longrightarrow M$ and let
    $\star$ be the undeformed multiplication of differential
    operators. Then for any $D(t) \in \hbar
    (\DiffOp(\Gamma(M, E))[t])[[\hbar]]$ the equation
    \begin{equation}
        \label{eq:ODEDiffOps}
        \frac{d}{d t} A(t) = D(t) \star A(t)
    \end{equation}
    has a unique solution for every initial condition $A(0) = A_0\in
    \DiffOp(\Gamma(M, E))[[\hbar]] $.  The important point here is that
    the solution is again in $(\DiffOp(\Gamma(M, E))[t])[[\hbar]]$. This
    is the situation which we encountered in Section~\ref{sec:Morita}
    frequently.
\end{example}

Another situation refers to $\ring{C} = \mathbb{C}$ and smooth
functions on a manifold only. Let $D(t) \in
\hbar(\DiffOp(M)[t])[[\hbar]]$ be a formal series of differential
operators depending polynomially on $t$ at each order of $\h$. Let
$d_0 \in \cO(M)$ be a function and consider the differential
equation
\begin{equation}
    \label{eq:ODEManifold}
    \frac{d}{d t} f(t) = (d_0 + D(t)) f(t)
    \quad
    \textrm{with}
    \quad
    f(0) = h
\end{equation}
with some \emp{invertible} $h \in \cO(M)[[\hbar]]$. Note that $h$ is
invertible iff the zeroth order $h_0$ is invertible.
Proposition~\ref{proposition:ODEGeneral} does not directly apply in
this case due to the non-trivial zeroth order contribution coming
from $d_0$. But the following holds.

\begin{pred}
    \label{proposition:ODEExponential}
    For any invertible $h \in \cO(M)[[\hbar]]$,
    \eqref{eq:ODEManifold} has a unique solution $f(t)$, for all $t \in
    \mathbb{R}$, of the form
    \begin{equation}
        \label{eq:AlmostAnExponential}
        f(t) = \E^{td_0} h_0 g(t),
    \end{equation}
    where $g(t) = 1 + \sum_{n=1}^\infty \hbar^n g_n(t)$ with $g_n(t)
    \in \cO(M)[t]$. In particular, $f(t)$ is invertible for all $t \in
    \mathbb{R}$.
\end{pred}
\begin{proof}
    We write $f(t) = \sum_{n=0}^\infty \hbar^n f_n(t)$. Then in order
    $\hbar^0$ \eqref{eq:ODEManifold} reads
    \[
    \frac{d}{d t} f_0(t) = d_0 f_0(t)
    \quad
    \textrm{with}
    \quad
    f_0(0) = h_0,
    \]
    so $f_0 = \E^{t d_0} h_0$ is the unique solution. Making the
    Ansatz $f(t) = \E^{td_0} h_0 g(t)$, we see that $f(t)$ is a solution of
    \eqref{eq:ODEManifold} if and only if $g(t)$ satisfies
    \[
    \frac{d}{d t} g(t)
    =
    \E^{-td_0} \frac{1}{h_0} D(t) \left(\E^{td_0} h_0 g(t)\right)
    =
    \tilde{D}(t) g(t),
    \]
    with initial condition $g(0) = \frac{h}{h_0}$. Note that $\tilde{D}
    \in \hbar(\DiffOp(M)[t])[[\hbar]]$ since every differentiation in
    $D(t)$ reproduces the exponential function, which in the end
    cancels. Thus only polynomials in $t$ remain. We can now apply
    Proposition~\ref{proposition:ODEGeneral} and obtain a unique
    solution  $g(t) \in (\cO(M)[t])[[\hbar]]$. Since the solution
    is obtained by iteration, we have
    $g_0(t) = 1$, for all $t$, in zeroth order. The invertibility of $f(t)$
    follows since its zeroth order is invertible.
\end{proof}
\begin{example}
    \label{example:StarExponential}
    Let $d_0 \in \cO(M)$ and $d_+(t) \in \hbar(\cO(M)[t])[[\hbar]]$ be
    given. Then we have a unique invertible solution $f(t)$ to the
    equation
    \begin{equation}
        \label{eq:AlmostStarExponential}
        \frac{d}{d t} f(t) = d(t) \star f(t)
        \quad
        \textrm{with}
        \quad
        f(0) = 1,
    \end{equation}
    where $d(t) = d_0 + d_+(t)$ and $\star$ is a star
    product on $M$. If $d(t) \equiv d$ is time independent then $f(t)$
    is the $\star$-exponential $\Exp_{\star}(t d)$ as in \cite{Bayen},
    \cite[App.~A]{BW1}, and \cite[Thm.~6.3.4]{waldmann:2007a}.
    Moreover, $f(t)$ is $\star$-invertible for all $t$ and the
    $\star$-inverse $f(t)^{-1}$ is determined by the equation
    \begin{equation}
        \label{eq:StarInverseAlmostExp}
        \frac{d}{d t} f(t)^{-1} = - f(t)^{-1} \star d(t)
        \quad
        \textrm{with}
        \quad
        f(0) = 1,
    \end{equation}
    so we also can apply
    Proposition~\ref{proposition:ODEExponential} to this situation.
\end{example}

%
% Maurer-Cartan elements and the twisting procedure
%

\section{Maurer-Cartan elements and the twisting procedure}
\label{App-B}

In this section we recall some general facts about Maurer-Cartan
elements and the twisting procedure. Further details can be found in
Sections 2.3 and 2.4 in \cite{thesis} and in \cite{erratum}.

%
% Maurer-Cartan elements in DGLAs
%

\subsection{Maurer-Cartan elements in DGLAs}
\label{subsec:MCinDGLA}

Recall that every DGLA $(\cL, d_{\cL}, [\cdot, \cdot]_{\cL})$ in
this paper is equipped with a complete descending filtration
\begin{equation}
    \label{Filt-cL1}
    \dots \supset \cF^{-2} \cL
    \supset \cF^{-1} \cL
    \supset \cF^0 \cL \supset \cF^1 \cL \supset \dots\,, \qquad \cL =
    \lim_{n} \cL/\cF^n\cL\,,
    \end{equation}
which means, in particular, that $\cF^1\cL$ is a projective limit of
nilpotent DGLAs.

By definition, $\al$ is a \textit{Maurer-Cartan element} of $\cL$ if
$\al\in \cF^1 \cL^1$ (i.e., $\al\in \cF^1 \cL$ and has degree $1$ in
$\cL$) and satisfies the equation
\begin{equation}
    \label{MC-def}
    d_{\cL} \al + \frac1{2} [\al, \al]_{\cL} = 0\,.
\end{equation}

Notice that $\mg(\cL) = \cF^1 \cL^0$ forms an ordinary (not graded)
Lie algebra which is the projective limit of nilpotent Lie algebras.
Hence $\mg(\cL)$ can be exponentiated to the group
\begin{equation}
    \label{mG}
    \mG(\cL) = \exp (\,\cF^1 \cL^0\,),
\end{equation}
and this group acts on Maurer-Cartan elements of $\cL$ via
\begin{equation}
    \label{action-A}
    \alpha \mapsto \al^{\exp(\xi)} = \exp([\cdot, \xi]_{\cL}) \al +
    \frac{\exp([\cdot, \xi]_{\cL})-1}{[\cdot, \xi]_{\cL}} \,
    d_{\cL}\xi\,,
\end{equation}
where $\xi \in \cF^1 \cL^0$\,, and the expression
\[
\frac{\exp([\cdot, \xi]_{\cL})-1}{[\cdot, \xi]_{\cL}}
\]
is defined via the Taylor expansion of the function $\frac{e^x -
1}{x}$ around the point $x=0$\,. Both terms on the right-hand side
of (\ref{action-A}) are well defined because the filtration on $\cL$
is complete. We remark that (\ref{action-A}) defines a right action,
i.e., for all $\xi, \eta\in \cF^1 \cL^0$ we have
\begin{equation}
    \label{action-is-right}
    \big(\al^{\exp(\xi)}\big)^{\exp(\eta)} =
    \al^{\exp(\CH(\xi, \eta))}\,,
\end{equation}
where $\CH(\xi, \eta)$ is the Campbell-Hausdorff series:
\begin{equation}
\label{eq:CH}
    \CH (\xi,\eta) = \log(e^{\xi} e^{\eta}) = \xi + \eta +
    \frac{1}{2}[\xi,\eta] + \dots \,.
\end{equation}

We let $\MC(\cL)$ denote the transformation groupoid  of the action
(\ref{action-A}), called the \textit{Goldman-Millson groupoid}
\cite{Goldman-M}: its objects are the Maurer-Cartan elements of
$\cL$ and morphisms between two Maurer-Cartan elements $\al_1$ and
$\al_2$ are elements of the group $\mG$ (\ref{mG}) which transform
$\al_1$ to $\al_2$. We call Maurer-Cartan elements
\textit{equivalent} if they are isomorphic in $\MC(\cL)$ and denote
by $\pi_0(\MC(\cL))$ the set of equivalence classes of Maurer-Cartan
elements.

Every morphism  $f: \cL \to \tcL$ of DGLAs defines a functor
\begin{equation}
    \label{f-star}
    f_* : \MC(\cL) \to \MC(\tcL).
\end{equation}
According to \cite{Ezra-Lie,Goldman-M,SS-MC}, we have the following
result.
\begin{teo}
    \label{MC-teo}
    If $f:\cL\to \tcL$ is a quasi-isomorphism of DGLAs,
    then the functor (\ref{f-star}) induces a
    bijection from $\pi_0 (\MC(\cL)) $ to $\pi_0(\MC(\tcL))$.
\end{teo}

Every Maurer-Cartan element $\al$ of $\cL$ can be used to modify the
DGLA structure on $\cL$. This modified structure is called the DGLA
structure \textit{twisted by the Maurer-Cartan $\al$}  \cite{Q}.
The Lie bracket of the twisted DGLA structure is unchanged, and the
differential is given by
\begin{equation}
    \label{d-al}
    d^{\al}_{\cL} = d_{\cL} + [\al, \cdot]_{\cL}\,.
\end{equation}
The DGLA resulting from twisting $\cL$ by $\al$ will denote by
$\cL^{\al}$.

%
% $\Linf$-morphisms of DGLAs
%

\subsection{$\Linf$-morphisms of DGLAs}
\label{subsec:LinftyMorphismsDGLA}

Two DGLAs $\cL$ and $\tcL$ are called \textit{quasi-isomorphic} if
there is a sequence of quasi-isomorphisms $f$, $f_1$, $f_2$,
$\dots$, $f_n$ connecting $\cL$ with $\tcL$:
\begin{equation}
    \label{cL-tcL}
    \cL \,\stackrel{f}{\rightarrow}\, \cL_1
    \,\stackrel{f_1}{\leftarrow}\, \cL_2
    \,\stackrel{f_2}{\rightarrow}\, \dots
    \,\stackrel{f_{n-1}}{\rightarrow}\,
    \cL_n \,\stackrel{f_n}{\leftarrow}\,  \tcL\,.
\end{equation}
It follows from Theorem \ref{MC-teo} that a sequence of
quasi-isomorphisms (\ref{cL-tcL}) between the DGLAs $\cL$ and $\tcL$
defines a bijection between the sets of equivalence classes of
Maurer-Cartan elements.

We will need to extend the class of morphisms between DGLAs to
$L_{\infty}$ morphisms. To this end, we need the Chevalley-Eilenberg
complex $C(\cL)$ of a DGLA $\cL$. As a graded vector space, $C(\cL)$
is the direct sum of all symmetric powers of the desuspension (see
Subsection~\ref{subsec:notation})
%\footnote{Recall that by the desuspension $\bs^{-1} V$ of $V$ we mean $\ve' \otimes V$, where $\ve'$ is a one-dimensional vector space placed in degree $-1$\,. }
$\bs^{-1} \cL$ of $\cL$:
\begin{equation}
    \label{CE}
    C(\cL) = \bigoplus_{k=1}^{\infty} S^k ( \bs^{-1}\cL )\,.
\end{equation}
The space $C(\cL)$ is equipped with the following cocommutative
comultiplication:
\begin{equation}
    \label{copro-C}
    \D\, :\, C(\cL)\longrightarrow C(\cL) \otimes C(\cL),
\end{equation}
defined by
\begin{align}
&\D(v_1)=0\,,\nonumber \\
    \label{copro-eq}
&    \D (v_1, v_2, \dots , v_n)  = \sum_{k=1}^{n-1}
    \sum_{\si \in {\rm Sh}(k,n-k)} \pm (v_{\si(1)}, \dots, v_{\si(k)})
    \otimes ( v_{\si(k+1)}, \dots , v_{\si(n)})\,,
\end{align}
where $v_1, \dots,\, v_n$ are homogeneous elements of $\bs^{-1}\cL$\,,
${\rm Sh}(k,n-k)$ is the set of $(k,n-k)$-shuffles in $S_n$\,, and the
signs are determined using the Koszul rule.

It can be shown that every coderivation $Q$ of the coalgebra
$C(\cL)$ is uniquely determined by its composition $p \circ Q$ with
the natural projection
\begin{equation}
    \label{p}
    p: C(\cL) \to \bs^{-1}\cL\,.
\end{equation}
This statement follows from the fact that $C(\cL)$ is a cofree
cocommutative coalgebra\footnote{In fact, $C(\cL)$ is a cofree
cocommutative  coalgebra without counit.}. A similar statement holds
for cofree coalgebras of other types, see \cite[Prop.~2.14]{GJ}.

We define the coboundary operator $Q$ of the complex $C(\cL)$ by
requiring that $Q$ is a coderivation of the coalgebra structure and
by setting
\[
p \circ Q(v) = - d_{\cL} v\,, \;\;\; p \circ Q(v_1, v_2) =
(-1)^{|v_1|+1} [v_1, v_2]_{\cL}\,,\;\;\; p \circ Q(v_1, v_2, \dots,
v_k) = 0\, \;\;\; (k
> 2),
\]
where $v, v_1, \dots, v_k$ are homogeneous elements of $\cL$\,.  The
equation $Q^2=0$ readily follows from the Leibniz rule and the Jacobi
identity.  Thus to every DGLA $(\cL, d_{\cL}, [\cdot, \cdot]_{\cL})$
we assign a DG cocommutative coalgebra
\[
(C(\cL), Q)
\]
without counit.
\begin{defi}
    \label{Linf-morph}
    An $\Linf$ morphism
    \[
    F:  \cL \brarrow \tcL
    \]
    from a DGLA $\cL$ to a DGLA $\tcL$ is a (degree zero) morphism of
    the corresponding DG cocommutative coalgebras:
    \[
    F: (C(\cL), Q) \to (C(\tcL), \tQ)\,.
    \]
\end{defi}
The compatibility of $F$ with the comultiplication $\D$
(\ref{copro-C}) implies that $F$ is uniquely determined by its
composition $p \circ F$ with the projection
\[
p : C(\tcL) \to \bs^{-1}\tcL\,.
\]
We denote by $F_n$ the following restriction of $p \circ F$:
\begin{equation}
    \label{F-n}
    F_n = p \circ F \Big|_{S^n( \bs^{-1}\cL ) } S^n (
    \bs^{-1}\cL ) \to \bs^{-1} \tcL\,.
\end{equation}
The maps $F_n$'s are the \textit{structure maps} of the $\Linf$
morphism $F$\,. The presence of the desuspensions in (\ref{F-n})
simply means that the map $F_n$ can be thought of as a map from
$\cL^{\otimes n}$ to $\tcL$ of degree $1-n$ with the following
symmetry in the arguments:
\[
F_n(\dots, \ga_1, \ga_2, \dots) = -(-1)^{|\ga_1||\ga_2|} F_n(\dots,
\ga_2, \ga_1, \dots)\,,
\]
where $|\ga_i|$ is the degree of $\ga_i$ in $\cL$\,. We tacitly use
this identification in our paper.

The compatibility of $F$ with the codifferentials $Q$ and $\tQ$ is
equivalent to a sequence of quadratic relations on $F_n$.
 The first of these relations
says that the map
$$
F_1 : \cL \to \tcL
$$
intertwines the differentials $d_{\cL}$ and $d_{\tcL}$:
\begin{equation}
\label{F-1-and-diff}
F_1 (d_{\cL} \ga) = d_{\tcL}\, F_1 (\ga)\,, \qquad \ga \in \cL\,.
\end{equation}
The second relation says that $F_1$ is compatible with the brackets
up to homotopy:
\begin{equation}
\label{F-2-and-F-1}
d_{\tcL}\, F_2(\ga_1, \ga_2) + F_2(d_{\cL} \ga_1, \ga_2)
+ (-1)^{|\ga_1|} F_2(\ga_1, d_{\cL} \ga_2)  =
F_1([\ga_1, \ga_2]_{\cL})
- [F_1(\ga_1), F_1(\ga_2)]_{\tcL}\,,
\end{equation}
where $ \ga_1, \ga_2 \in \cL$.

Condition (\ref{F-1-and-diff}) motivates the following definition:
\begin{defi}
    \label{Linf-q-iso}
    An $\Linf$ morphism $F: \cL \brarrow \tcL$ is an $\Linf$
    quasi-isomorphism if $F_1$ induces an isomorphism from
    $H^{\bul}(\cL, d_{\cL})$ to $H^{\bul}(\tcL, d_{\tcL})$\,.
\end{defi}

Just as ordinary quasi-isomorphisms, an $\Linf$ quasi-isomorphism
between DGLAs induces a bijection between the sets of equivalence
classes of Maurer-Cartan elements. More precisely, if $F$ is an
$\Linf$ morphism from $\cL$ to $\tcL$ then, for every Maurer-Cartan
element $\al$ of $\cL$,
\begin{equation}
    \label{MC-F-al}
    \beta = \sum_{n=1}^{\infty} \frac{1}{n!} F_n(\al,
    \al, \dots, \al)
\end{equation}
is a Maurer-Cartan element\footnote{The infinite series in
  (\ref{MC-F-al}) is well defined since $\al\in \cF^1\cL$ and $\tcL$ is complete with respect to its filtration.} of the DGLA
$\tcL$\,. Furthermore, if $\al$ is equivalent to $\al'$ in $\cL$,
then $\beta$ is equivalent to the Maurer-Cartan element
\[
\beta' = \sum_{n=1}^{\infty} \frac{1}{n!} F_n(\al', \al', \dots,
\al')
\]
of $\tcL$. As a result, the correspondence
\begin{equation}
    \label{corresp}
    \al \mapsto \beta= \sum_{n=1}^{\infty} \frac{1}{n!}
    F_n(\al, \al, \dots, \al)
\end{equation}
induces a map
\begin{equation}
    \label{F-star}
    F_{*} : \pi_0 (\MC(\cL)) \to  \pi_0 (\MC(\tcL))\,.
\end{equation}
Due to \cite[Prop.~4]{thesis} we have:
\begin{pred}
    \label{MC-teo-Linf}
    If $F: \cL \brarrow \tcL$ is an $\Linf$ quasi-isomorphism, then
    the map (\ref{F-star}) is a bijection.
\end{pred}

Given a Maurer-Cartan element $\al\in \cL$, any $\Linf$ morphism $F:
\cL \brarrow \tcL$ can be modified to an $\Linf$ morphism $F^{\al}$
between the twisted DGLAs $\cL^{\al}$ and $\tcL^{\beta}$, where
$\beta$ is as in (\ref{MC-F-al}), see \cite[Prop.~1]{thesis}. We say
that the $\Linf$ morphism $F^{\al}: \cL^\alpha \brarrow \tcL^\beta$
is \emp{twisted} by the Maurer-Cartan element $\al$; its structure
maps $F^{\al}_n$ are given by
\begin{equation}
    \label{F-n-al}
    F^{\al}_n (\ga_1, \ga_2, \dots, \ga_n) =
    \sum_{k=0}^{\infty} \frac{1}{k!} F_{k+n}(\al, \al, \dots, \al,
    \ga_1, \ga_2, \dots, \ga_n)\,,
\end{equation}
where $\ga_i \in \cL$.  In particular,  $F^{\al}_1$ intertwines the
differentials in $\cL^{\al}$ and $\tcL^{\beta}$:
\begin{equation}
    \label{F-1-al-inter}
    \sum_{k=0}^{\infty} \frac{1}{k!} F_{k+1}(\al,
    \al, \dots, \al, d_{\cL} \ga + [\al, \ga]_{\cL}) =
    \sum_{k=0}^{\infty} \frac{1}{k!} \,(d_{\tcL} + [\beta,
    \,\,]_{\tcL})\, F_{k+1}(\al, \al, \dots, \al, \ga )\,.
\end{equation}
According to \cite[Prop.~1]{thesis}, twisting an $\Linf$
quasi-isomorphism by a Maurer-Cartan element gives an $\Linf$
quasi-isomorphism.

One can identify $\Linf$ morphisms from a DGLA $\cL$ to a DGLA
$\tcL$ with Maurer-Cartan elements of another DGLA, denoted by
$\cH$\,. As a graded vector space,
\begin{equation}
    \label{Hom-cL-tcL}
    \cH = \Hom(C(\cL), \tcL)\,.
\end{equation}
The differential $d_{\cH}$ and the bracket $[\cdot, \cdot]_{\cH}$ are
given by the formulas:
\begin{align}
    \label{d-cH}
    &d_{\cH} \Psi = d_{\tcL} \Psi  - (-1)^{|\Psi|} \Psi Q\,,\\
    \label{brack-cH}
    &[\Psi, \Theta]_{\cH}(X) = \sum_i
    (-1)^{|\Theta|\,|X_i|} [\Psi(X_i), \Theta(X'_i)]_{\tcL}\,,
\end{align}
where $\D X = \sum_{i} X_i \otimes X'_i $, and $Q$ is the
codifferential on $C(\cL)$\,.  The DGLA $\cH$ is equipped with the
following descending filtration:
\begin{align}
&\cH = \cF^1 \cH \supset \cF^2 \cH \supset \dots \supset \cF^k \cH
\supset \dots\nonumber \\
    \label{filtr-cH}
 &   \cF^k \cH
    =
    \left\{
        f \in \Hom(C(\cL), \tcL)
        \; \big| \;
        f \big|_{S^{<k} (\bs^{-1}\cL )} = 0
    \right\}.
\end{align}
The DGLA structure defined by (\ref{d-cH}) and (\ref{brack-cH}) is
compatible with this filtration, and the DGLA $\cH$ is complete with
respect to this filtration.  Thus the group $\mG(\cH)$ is defined
for $\cH$ and acts on Maurer-Cartan elements of $\cH$ according to
(\ref{action-A}).

Following \cite{erratum,Shoikhet}, the correspondence
\begin{equation}
    \label{identif-MC}
    F \mapsto p \circ F\,.
\end{equation}
identifies an $\Linf$ morphism $F : \cL \brarrow \tcL$ with a
Maurer-Cartan element of the DGLA $\cH$. Moreover, for two $\Linf$
morphisms $F$ and $\tF$, the Maurer-Cartan elements $p \circ F$ and
$p \circ \tF$ are connected by the action (\ref{action-A}) of the
group $\mG(\cH)$, and the structure maps $F_1$ and $\tF_1$ are chain
homotopic. As a result, if the Maurer-Cartan elements $p \circ F$
and $p \circ \tF$ are equivalent and $F$ is an $\Linf$
quasi-isomorphism, then so is $\tF$\,. We say that two $\Linf$
morphisms $F$ and $\tF$ are \textit{homotopy equivalent} if the
corresponding Maurer-Cartan elements $p \circ F$ and $p \circ \tF$
are connected by the action (\ref{action-A}) of the group
$\mG(\cH)$\,.

It is natural to ask whether two homotopy equivalent $\Linf$
morphisms induce the same map from $\pi_0(\MC(\cL))$ to
$\pi_0(\MC(\tcL))$\,.  The following lemma gives a positive answer
to this question.
\begin{lem}
    \label{lemma-equiv}
    Let $\cL$ and $\tcL$ be DGLAs, and let $F$ and $\tF$ be two $\Linf$ morphisms from $\cL$ to $\tcL$\,.
    If the corresponding Maurer-Cartan elements $p
    \circ F$ and $p \circ \tF$ of the DGLA $\cH$ (\ref{Hom-cL-tcL})
    are equivalent, then $F$ and $\tF$ induce the same map from
    $\pi_0(\MC(\cL))$ to $\pi_0(\MC(\tcL))$\,.
\end{lem}
\begin{proof}
    We need to show that for every Maurer-Cartan element $\al$ of
    $\cL$, the Maurer-Cartan elements
    \[
    \beta = \sum_{n=1}^{\infty} \frac{1}{n!} F_n (\al, \al, \dots,
    \al),\;\;\;\mbox{ and }\;\;\;
    \tbeta = \sum_{n=1}^{\infty} \frac{1}{n!} \tF_n (\al, \al, \dots,
    \al)
    \]
    are connected by the action (\ref{action-A}) of the group
    $\mG(\tcL)$.  Let us denote by $f$ (resp. $\tf$) the composition
    $p \circ F$ (resp. $p \circ \tF$):
    \[
    f= p\circ F\,, \qquad \quad \tf= p\circ \tF\,.
    \]
    We know that $f$ and $\tf$ are equivalent Maurer-Cartan elements
    of the DGLA $\cH$. Hence there exists
    an element $\psi\in \cF^1 \cH^0$ such that
    \begin{equation}
        \label{tf-f-psi}
        \tf = \exp([\cdot, \psi]_{\cH}) f +
        \frac{\exp([\cdot, \psi]_{\cH})-1}{[\cdot, \psi]_{\cH}} \,
        d_{\cH}\psi.
    \end{equation}
    Let us consider the element
    \begin{equation}
        \label{cxp-al}
        \sum_{k=1}^{\infty}\frac{1}{k!} (\underbrace{\al,
          \al, \dots, \al}_{k})
    \end{equation}
    in the completion of the coalgebra $C(\cL)$ with respect to the
    natural filtration coming from $\cL$\,.  A direct computation
    shows that applying both sides of equation (\ref{tf-f-psi}) to the
    element (\ref{cxp-al}) and using the Maurer-Cartan equation
    (\ref{MC-def}), we obtain
    \[
    \tbeta = \exp([\cdot, \xi]_{\tcL}) \beta + \frac{\exp([\cdot,
      \xi]_{\tcL})-1}{[\cdot, \xi]_{\tcL}} \,  d_{\tcL} \xi \,,
    \]
    where the element $\xi \in \cF^1\tcL^0$ is defined by
    \[
    \xi = \sum_{k=1}\frac{1}{k!} \psi(\underbrace{\al, \al, \dots,
      \al}_{k})\,.
    \]
    It follows that $\tbeta$ is connected to the
    Maurer-Cartan element $\beta$ by the action (\ref{action-A}) of
    the group $\mG(\tcL)$, concluding the proof.
\end{proof}

%
% The case of $\h$-adic filtration
%

\subsection{The case of $\h$-adic filtration}
\label{subsect-example}

If $(\cL, d, [\cdot, \cdot])$ is a DGLA\footnote{In this subsection we
  omit the subscript $\cL$ for the differential $d_{\cL}$ and for the
  bracket $[\cdot, \cdot]_{\cL}$\,.} which is not equipped with a
descending filtration then, extending the differential $d$ and the
Lie bracket $[\cdot, \cdot]$ by $\bbC[[\h]]$-linearity, we get the
DGLA $\cL[[\h]]$ over the ring $\bbC[[\h]]$ with the obvious
descending filtration
\begin{equation}
    \label{h-filtr-cL}
    \cF^k\cL = \h^k \cL[[\h]].
\end{equation}
The new DGLA $\cL[[\h]]$ is clearly complete with respect to this
filtration.  This case is of central importance in our paper and,
here, we will give an alternative description of (iso)morphisms in
the Goldman-Millson groupoid $\MC(\cL[[\h]])$.

Let $\al \in \h \cL^1[[\h]]$ and $\xi \in \h (\cL^0[t])[[\h]]$. Due
to Proposition~\ref{proposition:ODEGeneral}, the differential
equation
\begin{equation}
    \label{diffura}
    \frac{d}{dt} \al(t) = d \xi + [\al(t), \xi]
\end{equation}
with initial condition
\begin{equation}
    \label{initial-cond}
    \al(0) = \al
\end{equation}
has a unique solution in $\h(\cL^1[t])[[\h]]$\,.  We claim that if
$\al$ satisfies the Maurer-Cartan equation
\[
d \al + \frac{1}{2}[\al, \al] = 0,
\]
then so does $\al(t)$\,.  Indeed, let
\[
\Psi(t) = d \al(t) + \frac{1}{2}[\al(t), \al(t)].
\]
Taking a derivative in $t$ and using (\ref{diffura}), we have
\[
\frac{d}{d t} (d \al(t) + \frac{1}{2}[\al(t), \al(t)] ) = [d \al(t)
+ \frac{1}{2}[\al(t), \al(t)], \xi],
\]
that is,
\[
\frac{d}{d t}\Psi(t) = [\Psi(t), \xi]\,.
\]
Note that $\Psi(0) =0$, since $\al$ satisfies the Maurer-Cartan
equation. Then Proposition~\ref{proposition:ODEGeneral} implies that
$\Psi(t) \equiv 0$, i.e., $\al(t)$ satisfies the Maurer-Cartan
equation for all $t$.

If $\xi$ does not depend on $t$ (that is $\xi \in \h \cL^0[[\h]]$ )
then the initial value problem (\ref{diffura}), (\ref{initial-cond})
can be solved explicitly. Indeed, in this case we have
$$
\al(t) =  \exp(t[\cdot, \xi]) \al +
    \frac{\exp(t[\cdot, \xi])-1}{[\cdot, \xi]_{\cL}} \,
    d \xi\,.
$$
In other words, if  $\xi$ does not depend on $t$, then the
evaluation of $\al(t)$ at $t=1$ is connected with $\al$ by the
action (\ref{action-A}) of the group $\mG(\cL[[\h]])$\,.

We will now show that, for an arbitrary element $\xi \in \h
(\cL^0[t])[[\h]]$, the evaluation $\al(1)$ is also connected with
$\al$ by the action of the group $\mG(\cL[[\h]])$. We need the
following technical statement:
\begin{lem}
    \label{nuzhna}
    Consider a Maurer-Cartan element $\al$ of $\cL[[\h]]$, let $\xi$ be
    an element of $\h(\cL^0[t])[[\h]]$, and $\al(t)$ be the unique
    solution of (\ref{diffura}) with initial condition
    (\ref{initial-cond}). Then for every $\eta\in \h \cL^0[[\h]]$ and
    every nonnegative integer $k$, the element
    \begin{equation}
        \label{la-t}
        \la(t) = \exp\left(\frac{t^{k+1}}{k+1} [\cdot,
            \eta]\right) \al(t) + \frac{\exp\left(\frac{t^{k+1}}{k+1} [\cdot,
              \eta] \right) - 1}{[\cdot, \eta]} ~ d \eta
    \end{equation}
    satisfies the differential equation
    \begin{equation}
        \label{diffura-la-t}
        \frac{d}{d t} \la(t) = d \txi + [\la(t),
        \txi ],
    \end{equation}
    where
    \[
    \txi = t^k \eta + \exp\left(\frac{t^{k+1}}{k+1}[\cdot, \eta]
    \right) \xi\,.
    \]
\end{lem}
\begin{proof}
    We compute the derivative explicitly and obtain
    \begin{align*}
        \frac{d}{d t} \la(t)
        &=
        t^k[\cdot, \eta]
        \exp\left(\frac{t^{k+1}}{k+1} [\cdot, \eta]\right) \al(t)
        +
        \exp\left(\frac{t^{k+1}}{k+1} [\cdot, \eta]\right)  d\xi
        \\
        &\quad+
        \exp\left(\frac{t^{k+1}}{k+1} [\cdot, \eta]\right)
        [\al(t), \xi]
        +
        t^k \exp\left(\frac{t^{k+1}}{k+1} [\cdot, \eta]\right)
        d \eta.
    \end{align*}
    The latter can be rewritten as
    \begin{align*}
        \frac{d}{d t} \la(t)
        &=
        \left[
            \exp\left(
                \frac{t^{k+1}}{k+1}
                [\cdot, \eta]
            \right) \al(t),
            t^k \eta
            +
            \exp\left(
                \frac{t^{k+1}}{k+1}
                [\cdot, \eta]
            \right) \xi
        \right] \\
        &\quad+
        d(t^k \eta)
        +
        [\cdot, \eta]
        \frac{\exp\left(\frac{t^{k+1}}{k+1} [\cdot, \eta]\right)-1}
        {[\cdot, \eta]} d(t^k \eta) \\
        &\quad+
        \exp\left(\frac{t^{k+1}}{k+1} [\cdot, \eta]\right) d
        \exp\left(-\frac{t^{k+1}}{k+1} [\cdot, \eta]\right)
        \exp\left(\frac{t^{k+1}}{k+1} [\cdot, \eta]\right) \xi \\
        &\quad-
        d \exp\left(\frac{t^{k+1}}{k+1} [\cdot, \eta]\right) \xi
        +
        d \exp\left(\frac{t^{k+1}}{k+1} [\cdot, \eta]\right) \xi.
    \end{align*}
    Hence
    \begin{align*}
        \frac{d}{d t} \la(t)
        &=
        d \left(
            t^k \eta
            +
            \exp\left(\frac{t^{k+1}}{k+1} [\cdot, \eta]\right) \xi
        \right) \\
        &\quad+
        \left[
            \la(t),
            t^k \eta
            +
            \exp\left(\frac{t^{k+1}}{k+1} [\cdot, \eta]\right) \xi
        \right] \\
        &\quad+
        \left(
            \exp\left(\frac{t^{k+1}}{k+1} [\cdot, \eta]\right)
            d
            \exp\left(-\frac{t^{k+1}}{k+1} [\cdot, \eta]\right)
            -
            d
        \right)
        \exp\left(\frac{t^{k+1}}{k+1} [\cdot, \eta]\right) \xi \\
        &\quad-
        \left[
            \frac{
              \exp\left(\frac{t^{k+1}}{k+1} [\cdot, \eta] \right) - 1
            }
            {[\cdot, \eta]}
            d \eta,
            \exp\left(\frac{t^{k+1}}{k+1} [\cdot, \eta]\right) \xi
        \right].
    \end{align*}
    Thus, in order to prove the proposition, we need to show that
    \begin{equation}
        \label{Exp-d-Exp}
        \exp\left(\frac{t^{k+1}}{k+1} [\cdot, \eta]\right)
        \left[
            d,
            \exp\left(-\frac{t^{k+1}}{k+1} [\cdot, \eta] \right)
        \right]
        =
        \left[
            \frac{
              \exp\left(\frac{t^{k+1}}{k+1} [\cdot, \eta] \right) - 1
            }
            {[\cdot, \eta]}
            d \eta, \cdot
        \right].
    \end{equation}
    One now verifies that both sides of (\ref{Exp-d-Exp}) satisfy
    the same differential equation:
    \[
    \frac{d}{d t} \Theta(t)
    =
    \left[
        t^k [\cdot, \eta], \Theta(t)
    \right]
    +
    t^k [d\eta, \cdot],
    \]
    with the same initial condition
    \[
    \Theta(0) =0.
    \]
    Therefore, by Proposition~\ref{proposition:ODEGeneral},
    (\ref{Exp-d-Exp}) holds and the result
    follows.
\end{proof}

We can now prove the main result of this subsection.
\begin{pred}
    \label{nuzhno}
    Let $\al$ be a Maurer-Cartan element of $\cL[[\h]]$\,, $\xi$ be
    an element of $\h(\cL^0[t])[[\h]]$\,, and $\al(t)$ be the unique
    solution of (\ref{diffura}) with the initial condition
    (\ref{initial-cond}). Then the Maurer-Cartan element $\al(1)$ is
    connected with $\al$ by the action (\ref{action-A}) of the group
    $\mG(\cL[[\h]])$\,.
\end{pred}

\begin{proof}
    Let us denote by $\rE(\xi, \al)$ the evaluation of $\al(t)$ at
    $t = 1$:
    \begin{equation}
        \label{E}
        \rE(\xi, \al) = \al(t) \Big|_{t=1}\,,
    \end{equation}
    where $\al(t)$ is the solution of the differential equation
    (\ref{diffura}) satisfying $\al(0)= \al$\,.  In
    general, we have $\xi\in \h^m (\cL^0[t])[[\h]]$, where $m$ is a
    positive integer, and
    \begin{equation}
        \label{xi-series}
        \xi = t^k  \h^m \xi_{m,k} +  t^{k+1} \h^m
        \xi_{m,k+1} +
        t^{k+2}  \h^m \xi_{m,k+2}+ \dots +  t^N \h^m \xi_{m,N}
        \qquad {\rm mod} \quad \h^{m+1}\,,
    \end{equation}
    $\xi_{m,j}\in \cL^0$, for some nonnegative integers $k$ and $N$ with $k\le N$\,.
    Lemma~\ref{nuzhna} implies that
\begin{equation}
\label{kill-Bill}
\big(\rE(\xi, \al)\big)^{\exp\left[-\frac{ \h^m \xi_{m,k}}{k+1}\right]} =
\rE(\xi', \al)\,,
\end{equation}
    where $\xi'\in \h^{m} (\cL^0[t])[[\h]]$ and
    \begin{equation}
        \label{xi-p-series}
        \xi' =  t^{k+1} \h^m \xi'_{m,k+1} +
        t^{k+2}  \h^m \xi'_{m,k+2}+ \dots +  t^N \h^m \xi'_{m,N}
        \qquad {\rm mod} \quad \h^{m+1},
    \end{equation}
    with
    $\xi'_{m,j}\in \cL^0\,. $

    Repeating this argument $N-k-1$ times, we see that there exists
    elements
   $ \txi_1\in \h^{m+1} (\cL^0[t])[[\h]]$
    and $\eta_m \in \h^m\cL^0[[\h]]$ such that
\begin{equation}
\label{kill-Bill-m}
\rE(\txi_1, \al) = \big(\rE(\xi, \al)\big)^{\exp[\eta_m]}\,.
\end{equation}
    Therefore we have an infinite series of elements $\eta_{m+n-1} \in
    \h^{m+n-1}\cL^0[[\h]]$ and elements
    \[
    \txi_{n}\in \h^{m+n} (\cL^0[t])[[\h]]\,, \qquad n \ge 1,
    \]
    such that
\begin{equation}
\label{kill-Bill-m-n}
\rE(\txi_n, \al) = \big(\rE(\xi, \al)\big)^{\La_{n,m}}\,,
\end{equation}
where
$$
\La_{n,m} =
\exp[\eta_{m}] \exp[\eta_{m+1}] \dots
 \exp[\eta_{m+n-1}]\,.
$$

Since for large $n$ the element $\eta_{m+n-1}$ lies in the deeper
filtration subalgebra $\h^{m+n-1}\cL^0[[\h]]$, the infinite product
\[
\La = \exp[\eta_{m}] \exp[\eta_{m+1}] \dots \exp[\eta_{m+n}] \dots
\]
is a well defined element of $\mG(\cL[[\h]])$\,.  Furthermore, due
to (\ref{kill-Bill-m-n}), we have
    \[
    \al = \big(\al(1)\big)^{\La}
    \]
    and the proposition follows.
\end{proof}

%%%%%%%%%%%%%%%%%%%%%%%%%%%%%%%%%%%%%%%%%%%%%%%%%%%%%%%%%%%%%%%%%%%%%%%%%%%%%%%%%%%%%%
\section{Independence of Fedosov's differential: proof of Theorem~\ref{nezalezh}}
\label{App-C}

    This section presents the proof of Theorem~\ref{nezalezh}, asserting that 
    the correspondence between equivalence
    classes of star products and equivalence classes of formal Poisson
    structures induced by the sequence of $\Linf$
    quasi-isomorphisms (\ref{upper}) does not depend on the choice of
    the connection/Fedosov differential. 

We would like to emphasize that we prove (and use) Theorem 
\ref{nezalezh} in the setting where the ground field $\bbC$
is replaced by the ground ring\footnote{See Remark \ref{remark:adding-hbar}.} 
$\bbC[[\h]]$. In particular, the connection form $\G$ in (\ref{nabla-prelim})
is replaced by a general formal Taylor power series in
  $\h$:
   \[
   \G_{\h} = \G_0 + \h \G_1 + \h^2 \G_2 + \dots,
   \]
and the element $A$ (\ref{A}) is allowed to have
the more general form:
\[
    A=\sum_{p=2, r=0}^{\infty} d x^k  \h^r A^j_{r; k
        i_1\dots i_p}(x) y^{i_1} \dots y^{i_p}\frac{\pa}{\pa y^j} \in
      \Om^1(M, \cT^1_{poly})[[\h]]\,.
\]

    Although the proof is long and technical, the
    general idea is simple. The key point is observing that
    changing the geometric Fedosov differential corresponds to
    twisting the DGLAs $\Omb(M, \scTp)[[\h]]$ and $\Omb(M,
    \sCbu(\SM))[[\h]]$ by a Maurer-Cartan element which is equivalent
    to zero. What makes the proof intricate is
    that one needs filtrations on these DGLAs which are more
    subtle than the $\h$-adic ones.

\begin{proof}[ of Theorem~\ref{nezalezh}]
Let us introduce the following 
descending filtrations on the DGLAs
$\Omb(M,\cT^{\bul+1}_{poly})[[\h]]$ and $\Omb(M,
    \sCbu(\SM))[[\h]]$: The $m$-th subspace $\mF^m
    \Omb(M, C^k(\SM))[[\h]]$ of the filtration on $\Omb(M,
    C^k(\SM))[[\h]]$ consists of the elements $P$ of $\Omb(M,
    C^k(\SM))[[\h]]$ satisfying
    \begin{equation}
        \label{mF-Cbu}
        P \,\big(\,\cF^{p_1} \SM[[\h]]  \otimes \cF^{p_2}
        \SM[[\h]] \otimes \dots \otimes  \cF^{p_k} \SM[[\h]] \, \big)\,
        \subset
        \bigoplus_{s+t = m + p_1 + p_2 + \dots + p_k}
        \Om^{s}(\cF^t \SM[[\h]])\,,
    \end{equation}
    where the filtration $\cF^{\bul}\SM[[\h]]$ is defined in
    Remark~\ref{remark-filtr}; the $m$-th subspace $\mF^m \Omb(M,
    \cT^{k}_{poly})[[\h]]$ of the filtration on $\Omb(M,
    \cT^{k}_{poly})[[\h]]$ is specified by the same condition:
    for $\ga\in \Omb(M, \cT^{k}_{poly})[[\h]]$, viewed as a
    element of $\Omb(M, C^k(\SM))[[\h]]$,
    \begin{equation}
        \label{mF-cT}
        \ga \,\big(\,\cF^{p_1} \SM[[\h]]  \otimes \cF^{p_2}
        \SM[[\h]] \otimes \dots \otimes  \cF^{p_k} \SM[[\h]] \, \big)\,
        \subset
        \bigoplus_{s+t = m + p_1 + p_2 + \dots + p_k}
        \Om^{s}(\cF^t \SM[[\h]]).
    \end{equation}
    The filtrations $\mF^{\bul} \Omb(M, \cT^{\bul+1}_{poly})[[\h]]$ and
    $\mF^{\bul} \Omb(M, \Cbu(\SM))[[\h]]$ assign to $y^i$, $d x^i$,
    $\pa_{y^i}$ and $\h$ the degrees $1$, $1$, $-1$, and $2$,
    respectively.  For the filtration $\mF^{\bul}$ on $\Omb(M,
    \cT^{\bul+1}_{poly})[[\h]]$ we have
    \begin{equation}
        \label{complete-cT}
        \Omb(M, \cT^{\bul+1}_{poly})[[\h]] = \lim_{m}
        \Omb(M, \cT^{\bul+1}_{poly})[[\h]] \, \big/ \, \mF^{m} \Omb(M,
        \cT^{\bul+1}_{poly})[[\h]]
    \end{equation}
    and
    \begin{equation}
        \label{mF-cT-bounded}
        \Omb(M, \cT^{\bul+1}_{poly})[[\h]] = \mF^{-d}
        \Omb(M, \cT^{\bul+1}_{poly})[[\h]]\,,
    \end{equation}
    where $d$ is the dimension of $M$.  Although the filtration
    $\mF^{\bul}$ on $\Omb(M, \sCbu(\SM))[[\h]]$ is unbounded in both
    directions, we still have the following important properties:
    \begin{equation}
        \label{complete-Cbu}
        \Omb(M, \sCbu(\SM))[[\h]] = \lim_{m} \Omb(M,
        \sCbu(\SM))[[\h]]  \, \big/ \, \mF^{m} \Omb(M, \sCbu(\SM))[[\h]]\,,
    \end{equation}
    \begin{equation}
        \label{cocomplete-Cbu}
        \Omb(M, \sCbu(\SM))[[\h]] = \bigcup_{m}
        \mF^{m} \Omb(M, \sCbu(\SM))[[\h]]\,.
    \end{equation}
    Property (\ref{complete-Cbu}) follows from the fact that
    the local sections of the sheaf $C^k(\SM)$ are continuous
    ($\cO(U)$-polylinear) maps from $\G(U, \SM)^{\otimes \, k}$ to
    $\G(U, \SM)$ in the $y$-adic topology on $\G(U, \SM)$\,.

    Let us consider two different geometric Fedosov differentials
    \begin{equation}
        \label{DDD1}
        D = \n - \de + A\,,\;\;\; \tD = \widetilde{\n} - \de + \tA,
    \end{equation}
%    and
%    \begin{equation}
%        \label{DDD-tilde}
 %       \tD = \widetilde{\n} - \de + \tA
 %   \end{equation}
    and let $\tau$, $\tilde{\tau}$ be the corresponding
    isomorphisms (see (\ref{tau-iso1})),
    \begin{equation}
        \label{la-D}
        \begin{array}{c}
         \tau: \mathcal{X}^{\bul}(M)[[\h]]
        \stackrel{\cong}{\longrightarrow}
        \G(M, \cTp)[[\h]] \cap \ker D\,,\\[0.3cm]
        \tilde{\tau} : \mathcal{X}^{\bul}(M)[[\h]]
        \stackrel{\cong}{\longrightarrow}
        \G(M, \cTp)[[\h]] \cap \ker \tD\,.
        \end{array}
    \end{equation}
    The geometric Fedosov differential $\tD$ can be rewritten as
    \begin{equation}
        \label{H}
        \tD = D + H,
        \quad
        \textrm{ where}
        \quad
        H \in \mF^1 \Om^1(M, \cT^{1}_{poly})[[\h]].
    \end{equation}
    Since $\tD^2 = 0$, the element $H$ satisfies the Maurer-Cartan
    equation
    \[
    D H + \frac{1}{2}[H, H]_{SN} = 0.
    \]

    Let us consider the natural extension of the map $\si$ (\ref{sigma}) to $\Omb(M,
    \cTp)[[\h]]$,
    \[
    \si(\ga) = \ga\Big|_{y^i = d x^i = 0}\,.
    \]
    For example, the subspace $\Om^0(M, \cT^{p}_{poly})[[\h]] \cap
    \ker \si$ consists of fiberwise polyvectors of the form
    \[
    \ga = \sum_{k \ge 0} \sum_{q \ge 1} \h^k \, \ga^{i_1 i_2 \dots
      i_p}_{k; j_1 j_2 \dots j_q} (x) \, y^{j_1} y^{j_2} \dots y^{j_q} \,
    \pa_{y^{i_1}} \wedge  \pa_{y^{i_2}} \wedge \dots \wedge
    \pa_{y^{i_p}},
    \]
    (with summation in $q$ starting with $1$). The element
    $H$ (\ref{H}) is a Maurer-Cartan element of the
    the following truncation of the DGLA $(\Omb(M, \scTp), D,
    [\,,\,]_{SN})$:
    \begin{equation}
        \label{DGLA-for-B}
        \cL_{\cT} =
        \bigoplus_{k\ge 0} \Om^0(M, \cT^{k+1}_{poly})[[\h]]
        \cap \ker \si
        \;\; \oplus \;\;
        \bigoplus_{k\ge 0} \Om^{\ge 1}(M, \cT^{k+1}_{poly})[[\h]].
    \end{equation}
    At the level of the associated
    graded complex
    \[
    \bigoplus_m \mF^m \cL_{\cT} \big/ \mF^{m+1}\cL_{\cT},
    \]
    the differential $D$ (\ref{DDD1}) boils down to $-\de$\,.
    Due to \eqref{Hodge} and Remark~\ref{remark-Hodge},
    the associated graded complex of $\cL_{\cT}$ is
    acyclic. Using properties (\ref{complete-cT}) and
    (\ref{mF-cT-bounded}), we conclude that, for all $m$, the sub DGLAs $\mF^m
    \cL_{\cT}$ and the sub DGLA $\cL_{\cT}$ are
    acyclic.  Theorem~\ref{MC-teo} then implies that every
    Maurer-Cartan element of the DGLA (\ref{DGLA-for-B}) can be
    brought to zero via the action of the group
    \begin{equation}
        \label{mG-for-cL-cT}
        \exp
        \Big(
        \mF^1 \,\Om^0(M, \cT^{1}_{poly})[[\h]]\cap \ker \si
        \Big).
    \end{equation}
    Since $H$ is a Maurer-Cartan element, it follows
    that there exists an element
    \begin{equation}
        \label{xi-for-B}
        X \in
        \mF^1 \,\Om^0(M, \cT^{1}_{poly})[[\h]]
        \cap
        \ker \si
    \end{equation}
    such that
    \begin{equation}
        \label{B-and-xi}
        H =
        \frac{\exp([\cdot, X]_{SN})- 1 }{[\cdot, X]_{SN}} D X.
    \end{equation}

Since components of $H$ have degrees in $y$ greater than or equal to
$1$ and the contracting homotopy $\de^{-1}$ for $\de$ raises the
degree in $y$ by $1$, we conclude that one can find the element $X$
(\ref{xi-for-B}) satisfying (\ref{B-and-xi}) as well as the
additional property
   \begin{equation}
        \label{xi-yy}
        \pa_{y^i} X  \Big|_{y = 0}  = 0,
        \quad
        \textrm{for all}
        \quad i.
    \end{equation}
In other words, we can find $X$ whose components have degrees in
fiber coordinates $y$ greater than or equal to $2$.

    It follows from (\ref{B-and-xi}) that the operator $e^{X}$
    intertwines the differentials $D$ and $\tD$:
    \begin{equation}
        \label{exp-xi-D-tD}
        \tD = e^{-X} \circ D \circ e^{X}.
    \end{equation}
    Furthermore, combining equation (\ref{exp-xi-D-tD}) with property
    (\ref{xi-yy}) we deduce that, for every formal Poisson structure
    $\pi$,
    \begin{equation}
        \label{exp-xi-la-pi}
        \exp(-[\cdot, X]_{SN})\tilde{\tau}(\pi) = \tau(\pi).
    \end{equation}
    Indeed, \eqref{exp-xi-D-tD} implies that both sides of
    (\ref{exp-xi-la-pi}) are $D$-flat. Then, (\ref{xi-yy}) implies
    that
    \[
    \si (\tau(\pi)) = \si(\,\exp(-[\cdot, X]_{SN})\tilde{\tau}(\pi)\,).
    \]
    Therefore, since every $D$-flat section $\ga$ is uniquely
    determined by its image $\si(\ga)$, we conclude that
    (\ref{exp-xi-la-pi}) holds.  Combining equations (\ref{B-and-xi})
    and (\ref{exp-xi-la-pi}), we deduce that
    \begin{equation}
        \label{exp-xi-nado}
        H + \tilde{\tau}(\pi) =
        \exp([\cdot, X]_{SN})\tau(\pi)
        +
        \frac{\exp([\cdot, X]_{SN})- 1 }{[\cdot, X]_{SN}} DX.
    \end{equation}
Hence the Maurer-Cartan elements $\tau(\pi)$ and $H+
\tilde{\tau}(\pi)$ of the DGLA $(\Omb(M, \scTp)[[\h]], D,
[\,,\,]_{SN})$ are equivalent.

Let $K^{tw}$ be the $L_{\infty}$ quasi-isomorphism
$$
K^{tw}: (\Omb(M, \scTp)[[\h]], D, [\,,\,]_{SN})  \brarrow (\Omb(M,
\sCbu(\SM))[[\h]], D + \pa^{\Hoch}, [\,,\,]_G)
$$
of Subsection~\ref{subsec:sequence}, and let $\tK^{tw}$ be the
analogous $L_{\infty}$ quasi-isomorphism obtained by replacing $D$
with the other geometric Fedosov differential $\tD$ (see
(\ref{DDD1})).
%DGLA
%\begin{equation}
%\label{Om-cTp-D} (\Omb(M, \scTp)[[\h]], D, [\,,\,]_{SN})
%\end{equation}
% to the DGLA
%\begin{equation}
%\label{Om-sCbu-D} (\Omb(M, \sCbu(\SM))[[\h]], D + \pa^{\Hoch},
%[\,,\,]_G)
%\end{equation}
%obtained along the lines of Subsection~\ref{subsec:sequence}.
%Similarly, let $\tK^{tw}$ be the $L_{\infty}$ quasi-isomorphism from
%the DGLA $(\Omb(M, \scTp)[[\h]], \tD, [\,,\,]_{SN})$ to the DGLA
%\begin{equation}
%\label{Om-sCbu-tD} (\Omb(M, \sCbu(\SM))[[\h]], \tD + \pa^{\Hoch},
%[\,,\,]_G)
%\end{equation}
%with another geometric Fedosov differential $\tD$ in (\ref{DDD1}).
Let $\mu$ be the Maurer-Cartan element of the DGLA
\begin{equation}\label{Om-sCbu-D}
(\Omb(M, \sCbu(\SM))[[\h]], D + \pa^{\Hoch}, [\,,\,]_G)
\end{equation}
corresponding to the Maurer-Cartan element $\tau(\pi)$ via the
$L_{\infty}$ quasi-isomorphism $K^{tw}$, i.e.,
    \begin{equation}
        \label{mu}
        \mu =
        \sum_{n=1}^{\infty}\frac{1}{n!}
        K^{tw}_n (\tau(\pi), \tau(\pi), \dots, \tau(\pi)).
    \end{equation}
    Similarly, we let $\tmu$ be the Maurer-Cartan element in
 $(\Omb(M, \sCbu(\SM))[[\h]],\tD + \pa^{\Hoch},[\,,\,]_G)$ corresponding
    to the Maurer-Cartan element $\tilde{\tau}(\pi)$ via $\tK^{tw}$:
    \begin{equation}
        \label{tmu}
        \tmu =
        \sum_{n=1}^{\infty}\frac{1}{n!}
        \tK^{tw}_n
        \left(
            \tilde{\tau}(\pi), \tilde{\tau}(\pi), \dots,
            \tilde{\tau}(\pi)
        \right).
    \end{equation}
    Combining $\tmu$ with $H = \tD - D$, we get a Maurer-Cartan element
$H + \tmu$ of the DGLA (\ref{Om-sCbu-D}).

\begin{claim}
\label{claim-H-tmu-H-ttau-pi} The Maurer-Cartan element $H + \tmu$
corresponds to the Maurer-Cartan element $H + \tilde{\tau}(\pi)$ via
the $L_{\infty}$ quasi-isomorphism $K^{tw}$, that is,
    \begin{equation}
        \label{B+tmu}
        H + \tmu
        =
        \sum_{n=1}^{\infty}\frac{1}{n!}
        K^{tw}_n \left(
            H + \tilde{\tau}(\pi),
            H + \tilde{\tau}(\pi),
            \ldots, H + \tilde{\tau}(\pi)
        \right).
    \end{equation}
\end{claim}
\begin{subproof}
   We notice that $\tK^{tw}$ is obtained from $K^{tw}$ via twisting
    by the Maurer-Cartan element $H$.  Therefore the right hand side
    of (\ref{B+tmu}) can be rewritten as
    \begin{align*}
        \sum_{n=1}^{\infty}\frac{1}{n!}
        & K^{tw}_n \left(
            H + \tilde{\tau}(\pi),
            H + \tilde{\tau}(\pi),
            \ldots, H + \tilde{\tau}(\pi)
        \right)
        \\
        &=
        \sum_{n=1}^{\infty}\frac{1}{n!}
        K^{tw}_n \left(H, H, \ldots, H\right)
        +
        \sum_{n=1}^{\infty}\frac{1}{n!}
        \tK^{tw}_n \left(
            \tilde{\tau}(\pi), \tilde{\tau}(\pi), \ldots,
            \tilde{\tau}(\pi)
        \right).
    \end{align*}
    Using the properties P~\ref{P-K1} and P~\ref{P-arg-vect}, we
    rewrite the first sum in the previous equation as
    \[
    \sum_{n=1}^{\infty}\frac{1}{n!}
    K^{tw}_n (H, H, \dots, H) = H.
    \]
    This proves Claim \ref{claim-H-tmu-H-ttau-pi}.
\end{subproof}

    Since the Maurer-Cartan elements $\tau(\pi)$ and $H +
    \tilde{\tau}(\pi)$ are equivalent in the DGLA
\begin{equation}
\label{Om-cTp-D} (\Omb(M, \scTp)[[\h]], D, [\,,\,]_{SN}),
\end{equation}
     Claim~\ref{claim-H-tmu-H-ttau-pi} and
    Proposition~\ref{MC-teo-Linf}
    from Appendix~\ref{App-B} imply that the Maurer-Cartan elements
    $\mu$ and $H + \tmu$ are equivalent in the DGLA (\ref{Om-sCbu-D}).
    Furthermore, since the $L_{\infty}$ quasi-isomorphism $K^{tw}$ is
    compatible with the filtrations (\ref{mF-Cbu}) and
    (\ref{mF-cT}),
    we conclude that the transformation connecting $\mu$ and $H +
    \tmu$ has the form
    \begin{equation}
        \label{exp-eta}
        \exp(\eta),
    \end{equation}
    where $\eta$ is an element of $\mF^1 \Omb(M, \sCbu(\SM))[[\h]]$ of
    total degree $0$.

In general the element $\mu$ (\ref{mu}) have components in exterior
degrees $0$, $1$, and $2$. Let us show that the components of
exterior degrees $1$ and $2$ can be eliminated by a transformation
of the form (\ref{exp-eta}).

\begin{claim}
\label{claim-how-to-kill-1-2} There exists an element $\eta\in \mF^1
\Omb(M, \sCbu(\SM))[[\h]]$ of total degree $0$ such that the
Maurer-Cartan element
\begin{equation}
\label{Pifib} \Pifib = \exp([\cdot, \eta]_{G}) \mu +
    \frac{\exp([\cdot, \eta]_{G})-1}{[\cdot, \eta]_{G}} \,
    (D\eta + \pa^{\Hoch}\eta)\,
\end{equation}
belongs to $\Om^0(M, C^2(\SM))[[\h]]$\,.
\end{claim}
\begin{subproof}
Let us denote by $\mu_1$ (resp. by $\mu_2$) the component of $\mu$
of exterior degree $1$ (resp. $2$). According to the definition of
    $K^{tw}$ (\ref{define-K-tw}), we have
    \begin{align}
        \label{mu-1}
&        \mu_1 =  \sum_{n=1}\frac{1}{n!} K_{n+1}
        \left(
            \mu^D_U, \tau(\pi), \tau(\pi), \dots, \tau(\pi)
        \right)\\
        \label{mu-2}
 &       \mu_2 =  \sum_{n=1}\frac{1}{n!} K_{n+2}
        \left(
            \mu^D_U,
            \mu^D_U, \tau(\pi), \tau(\pi), \dots, \tau(\pi)
        \right),
    \end{align}
    where $\mu^D_U$ is defined in (\ref{mu-D-U}).  Since the
    series $\pi$ \eqref{pi} starts with $\h$, we have
    \begin{align}
        \label{mu-1-mF}
        &\mu_1 =  \sum_{n=1}\frac{1}{n!} K_{n+1} \big( -d x^i
        \pa_{y^i} , \pi^y, \pi^y, \dots, \pi^y \big)
        \quad {\rm mod} \quad
        \mF^1 \Om^1(M, C^1(\SM))[[\h]]\\
        \label{mu-2-mF}
       & \mu_2 =  \sum_{n=1}\frac{1}{n!} K_{n+2} \big( -d x^i
        \pa_{y^i} , -d x^j  \pa_{y^j}, \pi^y, \pi^y, \dots, \pi^y \big)
        \quad {\rm mod} \quad
        \mF^1 \Om^2(M, C^0(\SM))[[\h]]\,,
    \end{align}
    where
    \[
    \pi^y = \pi^{ij}(x) \frac{\pa}{\pa y^i} \wedge \frac{\pa}{\pa
      y^j}.
    \]
    (Note that since $\pi$ is a series in $\h$, the coefficients $\pi^{ij}$ are $\h$-dependent.)
    We claim that
    \[
    K_{n+1} \big( -d x^i \pa_{y^i} , \pi^y, \pi^y, \dots, \pi^y \big)
    = 0\,,\;\;\;\;
    K_{n+2} \big( -d x^i \pa_{y^i} , -d x^j  \pa_{y^j}, \pi^y, \pi^y,
    \dots, \pi^y \big) = 0
    \]
    for all $n \ge 1$. The latter equality follows from the
    fact that the components of the vector $d x^i \pa_{y^i}$ and the
    bivector $\pi^y$ do not depend on $y$. As for the former equality,
    we note that every term in
    \[
    K_{n+1} \big( -d x^i \pa_{y^i} , \pi^y, \pi^y, \dots, \pi^y
    \big)(a)\,, \qquad a \in \G(M, \SM)
    \]
    contains a $y$-derivative of the expression $\pi^{ij}(x) \pa_{y^i}
    \pa_{y^j} a(x,y)$ as a factor, and $\pi^{ij}(x) \pa_{y^i} \pa_{y^j}
    a(x,y) =0$ due to the antisymmetry of $\pi$.  Thus both components
    $\mu_1$ and $\mu_2$ belong to $\mF^1 \Omb(M, \sCbu(\SM))[[\h]]$.
    On the other hand, the differential $D+ \pa^{\Hoch}$ boils down to
    $-\de + \pa^{\Hoch}$ at the level of the associated graded
    complex
    \[
    \bigoplus_m  \mF^m \Omb(M, \sCbu(\SM))[[\h]] /  \mF^{m+1} \Omb(M,
    \sCbu(\SM))[[\h]]\,.
    \]
Thus Claim \ref{claim-how-to-kill-1-2} follows from the fact that
the differential $\de$ is acyclic in positive exterior degree.
\end{subproof}

Since $\Pifib$ \eqref{Pifib} has exterior degree $0$, the
Maurer-Cartan equation for $\Pifib$,
    \[
    D \Pifib + \pa^{\Hoch}\Pifib + \frac{1}{2} [\Pifib, \Pifib]_G =
    0,
    \]
    is equivalent to the pair of equations
    \begin{align}
        \label{D-Pi}
&        D \Pifib = 0\\
        \label{pa-Hoch-Pi}
&        \pa^{\Hoch} \Pifib + \frac{1}{2} [\Pifib, \Pifib]_G = 0.
    \end{align}
    Equation~(\ref{pa-Hoch-Pi}) implies that $\Pifib$ gives us a new
    associative product on $\SM[[\h]]$:
    \begin{equation}
        \label{bul-Pi}
        a_1 \diamond a_2 = a_1 a_2 + \Pifib(a_1,a_2),
    \end{equation}
    where $a_1, a_2\in \G(M, \SM)[[\h]]$.  Equation~(\ref{D-Pi})
    implies that $D$ is a derivation of the product $\diamond$\,.

Similarly, the Maurer-Cartan element $\tmu$ is equivalent to a
Maurer-Cartan element $\tPifib \in \Om^0(M, C^2(\SM))[[\h]]$ of the
DGLA
\begin{equation}
\label{Om-sCbu-tD} (\Omb(M, \sCbu(\SM))[[\h]], \tD + \pa^{\Hoch},
[\,,\,]_G).
\end{equation}
Just as $\Pifib$, the element $\tPifib$ gives us an
    associative product on $\SM[[\h]]$ by
    \begin{equation}
        \label{bul-tPi}
        a_1 \tdiamond a_2 = a_1 a_2 + \tPifib(a_1, a_2),
    \end{equation}
and the differential $\tD$ is a derivation of $\tdiamond$.

    To explain how $\Pifib$ and $\tPifib$ are related to the corresponding
    star products $*$ and $\tstar$ on $M$, recall that we have the
 isomorphisms
    \begin{equation}
        \label{tau-nezalezh}
        \tau : \cO(M)[[\h]] \to \G(M, \SM)[[\h]] \cap \ker D
        \quad
        \textrm{and}
        \quad
        \ttau : \cO(M)[[\h]] \to \G(M, \SM)[[\h]] \cap \ker \tD.
    \end{equation}
These isomorphisms are constructed by iterating the following
equations:
    \begin{align}
        \label{iter-tau-nezalezh}
        &\tau(f) = f + \de^{-1}(\n \tau(f)+ A\cdot
        \tau(f))\,, \qquad f \in \cO(M)[[\h]],\\
        \label{iter-ttau-nezalezh}
        &\ttau(f) = f + \de^{-1}(\tn \ttau(f)+
        \tA\cdot \ttau(f))\,, \qquad f \in \cO(M)[[\h]],
    \end{align}
    in degrees in the fiber coordinates $y$'s, respectively.  The
    star products $*$, corresponding to $\Pifib$, and $*'$, corresponding
    to $\tPifib$, are defined by
    \begin{equation}
        \label{star-Pi-nezalezh}
        f_1 * f_2
        = f_1 f_2 + \Pifib(\tau(f_1), \tau(f_2)) \Big|_{y=0}
        \quad
        \textrm{and}
        \quad
        f_1 \tstar f_2
        = f_1 f_2 + \tPifib(\ttau(f_1), \ttau(f_2)) \Big|_{y=0},
    \end{equation}
    respectively, where $f_1, f_2\in \cO(M)[[\h]]$.  Our final goal is
    to show that $*$ is equivalent to $\tstar$.

Let us combine $\tPifib$ with the difference $H = \tD- D$ to get a
Maurer-Cartan element $H + \tPifib$ of the DGLA \eqref{Om-sCbu-D}.
Next, we will show that the Maurer-Cartan elements $\Pifib$ and $H +
\tPifib$ of the DGLA (\ref{Om-sCbu-D}) are connected by an
equivalence transformation of a special form.

\begin{claim}
\label{claim-H-tPi-Pi} There exists an element $\psi \in \mF^1 \,
\Om^0(M, C^1(\SM))[[\h]] $ such that
    \begin{equation}
        \label{konets1}
        H + \tPifib
        = \exp([\cdot, \psi]_G) \Pifib
        + \frac{\exp([\cdot, \psi]_G) -1}{[\cdot, \psi]_G}
        \, (D \psi + \pa^{\Hoch} \psi).
    \end{equation}
\end{claim}
\begin{subproof}
    Since $\tmu$ is equivalent to
    $\tPifib$ in the DGLA (\ref{Om-sCbu-tD}), one can check that the Maurer-Cartan
    element $H + \tmu$ is equivalent to $H + \tPifib$ in the DGLA
    (\ref{Om-sCbu-D}).
Using that the Maurer-Cartan element $\mu$ is equivalent to $\Pifib$
and $H+ \tmu$ in (\ref{Om-sCbu-D}), we conclude that the
Maurer-Cartan elements $\Pifib$ and $H+\tPifib$ are also equivalent.
In addition, using Claim \ref{claim-how-to-kill-1-2}, we see that
the equivalence transformation which connects the Maurer-Cartan
elements $\Pifib$ and $H+ \tPifib$ has the form
$$
\exp(\psi),
$$
where $\psi$ is an element of $\mF^1 \Omb(M, \sCbu(\SM))[[\h]]$ of
total degree $0$.

    In general $\psi$ may have two non-zero components,
    \[
    \psi = \psi_0 + \psi_1,
    \]
    where $\psi_0 \in \mF^1 \Om^0(M, C^1(\SM))[[\h]]$ and $\psi_1 \in
    \mF^1 \Om^1(M, C^0(\SM))[[\h]]$.
Our purpose is to show that $\psi_1$ can be eliminated by adjusting
    $\psi$ via the following transformation\footnote{Following E.
      Getzler \cite{Ezra-higher,Ezra-Lie} the Goldman-Millson
      groupoid of the DGLA (\ref{Om-sCbu-D}) can be upgraded to a
      2-groupoid. The transformation (\ref{adjusting-psi}) is an
      example of a $2$-morphism in this 2-groupoid.}:
    \begin{equation}
        \label{adjusting-psi}
        \psi \mapsto \mathrm{CH}\big(D \theta +
        \partial^{Hoch} \theta + [\Pifib, \theta]_G\,\, ,\, \, \psi\big),
    \end{equation}
    where
    $\te \in \mF^1\, \Om^0(M, C^0(\SM))[[\h]]$
    and $\CH$ is the Campbell-Hausdorff series (\ref{eq:CH}).
    The key point is that the element $\exp(D \te + \pa^{\Hoch} \te +
    [\Pifib, \te]_G)$ leaves the Maurer-Cartan element $\Pifib$
    unchanged.  Hence the element $\psi$ in (\ref{konets1}) can
    always be replaced by the right-hand side of
    \eqref{adjusting-psi}.

    % $\CH\big(D \te + \pa^{\Hoch} \te +[\Pifib, \te]_G\,, \, \psi \big)$.

Let us suppose that
    \begin{equation}
        \label{psi-1-mF-m}
        \psi_1 \in \mF^m \, \Om^1(M, C^0(\SM))[[\h]],
    \end{equation}
    for $m \ge 1$.  Combining the contributions to $\Om^2(M,
    C^0(\SM))[[\h]]$ in (\ref{konets1}) we see that
    \begin{equation}
        \label{psi-1}
        \de \psi_1 = 0 \qquad {\rm mod} \qquad \mF^{m+1} \,
        \Om^2(M, C^0(\SM))[[\h]].
    \end{equation}
    Using the acyclicity of $\de$ in positive exterior degrees, we
    conclude that there exists a
    \[
    \te_m \in \mF^m \Om^0(M, C^0(\SM))[[\h]]
    \]
    such that
    \[
    \psi_1 - \de \te_m \in  \mF^{m+1} \, \Om^1(M, C^0(\SM))[[\h]].
    \]
    The latter means that the $\Om^1$-component of $\CH\big(D \te + \pa^{\Hoch}
    \te + [\Pifib, \te]_G\,,\, \psi \big)$ lies in the ``smaller'' filtration
    subspace $\mF^{m+1} \, \Om^1(M, C^0(\SM))[[\h]]$.  Iterating
    this argument infinitely many times and using the completeness of
    the filtration $\mF^{\bul}$, we conclude that there exists an
    element
    \[
    \te \in \mF^1\, \Om^0(M, C^0(\SM))[[\h]]
    \]
    such that
    \[
    \CH\big(D \te + \pa^{\Hoch} \te + [\Pifib, \te]_G \,,\, \psi \big)
    \in \Om^0(M, C^1(\SM))[[\h]].
    \]
    This completes the proof of Claim \ref{claim-H-tPi-Pi}\,.
\end{subproof}

Since the element $\psi$ has zero exterior degree, equation
(\ref{konets1}) splits into its homogeneous exterior degree
components:
    \begin{align}
        \label{konets1-1}
        &H = \frac{ \exp([\cdot, \psi]_G) -1}{[\cdot, \psi]_G} \, D
        \psi,\\
        \label{konets1-0}
        &\tPifib
        = \exp([\cdot, \psi]_G) \Pifib
        + \frac{ \exp([\cdot, \psi]_G) -1}{[\cdot, \psi]_G}
        \,
        \pa^{\Hoch} \psi.
    \end{align}
    Equation~(\ref{konets1-1}) implies that the operator
    \[
    T_{\psi} = \exp(\psi) : \G(M, \SM)[[\h]] \to \G(M, \SM)[[\h]]
    \]
    intertwines the geometric Fedosov differentials $D$ and $\tD$:
    \begin{equation}
        \label{T-psi-D-tD}
        T_{\psi} \circ \tD = D \circ T_{\psi}.
    \end{equation}
    Similarly, (\ref{konets1-0}) implies that $T_{\psi}$
    intertwines the fiberwise products (\ref{bul-Pi}),
    (\ref{bul-tPi}):
    \begin{equation}
        \label{T-psi-bul-tbul}
        a_1 \tdiamond a_2
        = T_{-\psi}(T_{\psi}(a_1) \diamond T_{\psi} (a_2))
        \quad
        \textrm{for}
        \quad
        a_1, a_2 \in \G(M, \SM)[[\h]].
    \end{equation}
    Using $T_{\psi}$, we define a $\bbC[[\h]]$-linear map
    \begin{equation}
        \label{T-for-nezalezh}
        T : \cO(M)[[\h]] \to  \cO(M)[[\h]], \;\;\; T(f) = \si(T_{\psi}\circ \ttau
        (f)),
    \end{equation}
    for $f \in \cO(M)[[\h]]$, where $\si$ is defined in (\ref{sigma})
    and $\ttau$ is defined in (\ref{iter-ttau-nezalezh}).  Just as
    $\tau$, the map $\ttau$ satisfies property
    \eqref{tau-property}. Combining this observation with the fact
    that $\psi$ belongs to the first filtration subspace, it follows
    that
    \[
    T = \id + \h T_1 + \h^2 T_2 + \cdots,
    \]
    where $T_1, T_2, \dots$ are differential operators on $M$.  Since
    $T_{\psi}$ intertwines the geometric Fedosov differentials $D$ and
    $\tD$, we get that
    \begin{equation}
        \label{D-T-psi-tau}
        D T_{\psi} \circ \ttau (f) = 0,
    \end{equation}
    for all $f\in \cO(M)[[\h]]$.  On the other hand, every $D$-flat
    section $\ga$ of $\SM[[\h]]$ is uniquely determined by its image
    $\si(\ga)$\,. Hence (\ref{T-for-nezalezh}) and (\ref{D-T-psi-tau})
    imply that
    \begin{equation}
        \label{T-psi-T-tau}
        T_{\psi} \ttau (f) = \tau (T (f))
    \end{equation}
    for all $f\in \cO(M)[[\h]]$.  Combining this observation with
    (\ref{T-psi-bul-tbul}), we conclude that $T$ intertwines the
    star products (\ref{star-Pi-nezalezh}):
    \[
    T (f_1) * T (f_2) = T(f_1\, \tstar \, f_2)\,, \qquad \quad f_1, f_2
    \in \cO(M)[[\h]].
    \]

    Thus we proved that the correspondence between equivalence classes
    of star products and equivalence classes of formal Poisson
    structures produced by the sequence of $\Linf$
    quasi-isomorphisms (\ref{upper}) does not depend on the choice of
    the connection/Fedosov differential.

%    If instead of the sequence (\ref{upper}) we use the direct
%    $L_{\infty}$ quasi-isomorphism $\cK$ (\ref{cK}) then we still get
%    the same correspondence between equivalence classes. The latter
%    follows easily from Lemma~\ref{lemma-equiv} in
%    Appendix~\ref{App-B}.
\end{proof}

%
% Here comes the bibliography
%

~\\

\noindent\textsc{Instituto Nacional de Matematica Pura e Aplicada\\
Estrada Dona Castorina 110 \\
Rio de Janeiro, 22460-320, Brasil\\
\emp{E-mail address:} {\bf henrique@impa.br} }

~\\

\noindent\textsc{Department of Mathematics,
University of California at Riverside, \\
900 Big Springs Drive,\\
Riverside, CA 92521, USA \\
\emp{E-mail address:} {\bf vald@math.ucr.edu}}

~\\

\noindent\textsc{Fakult\"at f\"ur Mathematik und Physik\\
Physikalisches Institut\\
Hermann Herder Strasse 3\\
D 79104 Freiburg, Germany \\
\emp{E-mail address:} {\bf Stefan.Waldmann@physik.uni-freiburg.de}
}

\end{document}